\documentclass{amsproc}

\usepackage{amsmath}
\usepackage{amssymb}
\usepackage{rotating}
\usepackage{graphicx}
\usepackage[tableposition=below]{caption}

\DeclareGraphicsExtensions{.png,.pdf,.eps}
\usepackage[all,2cell,cmtip]{xy}
\usepackage{tikz}
\usepackage{color}
\usepackage{longtable}
\usepackage{soul}
\newtheorem{theorem}{Theorem}[section]
\newtheorem{lemma}[theorem]{Lemma}
\newtheorem{corollary}[theorem]{Corollary}
\newtheorem{proposition}[theorem]{Proposition}
\newtheorem{conjecture}[theorem]{Conjecture}

\theoremstyle{definition}

\newtheorem{definition}[theorem]{Definition}
\newtheorem{notation}[theorem]{Notation}

\newtheorem{remark}[theorem]{Remark}

\newtheorem{example}[theorem]{Example}
\newtheorem{question}{Question}

\def\C{{\mathbb C}}

\def\G{{\mathbb G}}

\def\N{{\mathbb N}}
\def\P{{\mathbb P}}
\def\Q{{\mathbb Q}}
\def\R{{\mathbb R}}
\def\Z{{\mathbb Z}}

\def\cC{{\mathcal C}}
\def\cD{{\mathcal D}}
\def\cE{{\mathcal E}}
\def\cF{{\mathcal F}}

\def\cI{{\mathcal I}}
\def\cJ{{\mathcal J}}

\def\cL{{\mathcal L}}
\def\cM{{\mathcal M}}
\def\cN{{\mathcal{N}}}
\def\cO{{\mathcal{O}}}

\def\cU{{\mathcal U}}

\def\cX{{\mathcal X}}
\def\cY{{\mathcal Y}}
\def\cZ{{\mathcal Z}}
\def\Q{{\mathbb{Q}}}
\def\G{{\mathbb{G}}}

\def\fg{{\mathfrak g}}
\def\fh{{\mathfrak h}}
\def\fp{{\mathfrak p}}

\def\fsp{{\mathfrak sp}}
\def\sl{\operatorname{SL}}
\def\so{\operatorname{SO}}
\def\sp{\operatorname{Sp}}

\def\eps{\varepsilon}

\def\lra{\longrightarrow}

\def\ra{\rightarrow}
\def\lra{\longrightarrow}
\def\rat{\dashrightarrow}

\def\hs{\hspace{0.1 em}}

\def\operatorname#1{\mathop{\rm #1}\nolimits}

\def\Proj{\operatorname{Proj}}
\def\Aut{\operatorname{Aut}}

\def\Chow{\operatorname{Chow}}

\def\Exc{\operatorname{Exc}}

\def\Hom{\operatorname{Hom}}

\def\Pic{\operatorname{Pic}}
\def\Sing{\operatorname{Sing}}
\def\Sym{\operatorname{Sym}}

\def\Spec{\operatorname{Spec}}
\def\Sym{\operatorname{Sym}}

\def\codim{\operatorname{codim}}

\def\id{\operatorname{id}}

\def\rk{\operatorname{rk}}
\def\supp{\operatorname{supp}}

\def\det{\operatorname{det}}
\def\coker{\operatorname{coker}}

\def\rat{\operatorname{RatCurves}}

\def\sga{{\langle}}
\def\sgc{{\rangle}}

\def\qed{\hspace{\fill}$\rule{2mm}{2mm}$}
\def\NE{{\operatorname{NE}}}

\def\Nef{{\operatorname{Nef}}}

\def\Nu{{\operatorname{N_1}}}

\newcommand{\cNE}[1]{\overline{\NE}(#1)}

\def\ev{{\operatorname{ev}}}

\def\tr{\operatorname{tr}}

\def\ad{\operatorname{ad}}
\def\Gl{\operatorname{GL}}

\newcommand{\sgn}{\operatorname{sgn}}
\newcommand{\Chi}{\ensuremath \raisebox{2pt}{$\chi$}}

\makeindex

\begin{document}

\title[A survey on the Campana-Peternell Conjecture]{A survey on the Campana-Peternell Conjecture}

\author[R. Mu\~noz]{Roberto Mu\~noz}
\address{Departamento de Matem\'atica Aplicada, ESCET, Universidad
Rey Juan Carlos, 28933-M\'ostoles, Madrid, Spain}
\email{roberto.munoz@urjc.es}
\thanks{First author partially supported by the Spanish government project MTM2012-32670. Second author were partially supported by PRIN project ``Geometria delle variet\`a algebriche'' and the Department of Mathematics of the University of Trento. Third author supported by the Korean National Researcher Program 2010-0020413 of NRF. Fourth author partially supported by JSPS KAKENHI Grant Number 26800002. Fifth author supported  by Polish NCN grant 2013/08/A/ST1/00804.This paper was written during the period in which the third author was a Visiting Researcher at the Korea Institute for Advanced Study (KIAS). He would like to thank this institution for its support and hospitality.}

\author[G. Occhetta]{Gianluca Occhetta}
\address{Dipartimento di Matematica, Universit\`a di Trento, via
Sommarive 14 I-38123 Povo di Trento (TN), Italy} 
\email{gianluca.occhetta@unitn.it}

\author[L.E. Sol\'a Conde]{Luis E. Sol\'a Conde}
\address{Korea Institute for Advanced Study, 85 Hoegiro, Dongdaemun-gu, Seoul, 130--722, Republic of Korea}
\email{lesolac@gmail.com}

\author[K. Watanabe]{Kiwamu Watanabe}
\address{Course of Mathematics, Programs in Mathematics, Electronics and Informatics,
Graduate School of Science and Engineering, Saitama University.
Shimo-Okubo 255, Sakura-ku Saitama-shi, 338-8570 Japan}
\email{kwatanab@rimath.saitama-u.ac.jp}

\author[J. Wi\'sniewski]{Jaros\l{}aw A. Wi\'sniewski}
\address{Instytut Matematyki UW, Banacha 2, 02-097 Warszawa, Poland}
\email{J.Wisniewski@uw.edu.pl} 

\subjclass[2010]{Primary 14J45; Secondary 14E30, 14M17}

\begin{abstract}
In 1991 Campana and Peternell proposed, as a natural algebro-geometric extension of Mori's characterization of the projective space, the problem of classifying the complex projective Fano manifolds whose tangent bundle is nef, conjecturing that the only varieties satisfying these properties are rational homogeneous. In this paper we review some background material related to this problem, with special attention to the partial results recently obtained by the authors.
\end{abstract}

\maketitle

\tableofcontents


\section{Introduction}\label{sec:intro}

Throughout this paper, unless otherwise stated, we will work in the
framework of smooth complex projective varieties.  Any variety $X$
comes naturally equipped with the sheaf of its K\"ahler differentials
$\Omega_X$ and its tensor algebra. Thus, properties of the objects
related to differential forms can be used for setting up the
classification of algebraic varieties. The fundamental example is the
classical trichotomy of smooth algebraic curves which are divided into
three unequal classes depending on the global properties of the sheaf
of differential forms or, dually, of their tangent bundle. The most
obvious property differentiating these classes is the sign of the
degree of the tangent bundle, equal to the Euler characteristic for
curves defined over complex numbers, which implies fundamentally
different behaviour of other invariants of the underlying variety.

This natural idea is extended to higher dimensions in a number of
ways. Firstly, one considers the canonical divisor $K_X$ associated to
$\det \Omega_X$ and pluricanonical systems $|mK_X|$, for $m>0$, which
give rise to the definition of Kodaira dimension of a
variety. Secondly, the numerical properties of the canonical divisor
(provided it is ${\mathbb Q}$-Cartier), that is the sign of
intersection of $K_X$ with curves on $X$, are fundamental for the
minimal model program in which the classification of higher
dimensional varieties is modelled on the trichotomy for
curves. Clearly, the situation is by far more complicated because the
higher dimensional varieties can be built from lower dimensional ones
on which $K_X$ may have different behaviour. In this scheme of
classification the {\em building blocks} are those varieties on which
either $K_X$ or $-K_X$ is positive, or it is numerically
trivial. Here, the notion of positivity of a line bundle, or more
generally, of a ${\mathbb Q}$-Cartier divisor, is generalized to
higher dimensions by ampleness; see the classical book
\cite{Hartshorne-ample} for discussion of the notion of ampleness, its
extensions and analytic counterparts.

It is plausible to expect that once the sign of the canonical divisor
is fixed a more refined classification can be achieved by
investigating the structure of the tangent sheaf $T_X$ of $X$. For
example, the celebrated theorem of Beauville, \cite{Beauville}, asserts
that complex K\"ahler manifolds with $K_X$ numerically trivial are, up
to a finite \'etale cover, products of complex tori, Calabi-Yau
varieties and hyper-K\"ahler (or irreducible symplectic) varieties.
Beauville's theorem is an incarnation of a result of de Rham
about decomposition of Riemannian varieties with respect to their
holonomy groups. The theorem of de Rham asserts that the structure on
the tangent bundle resulting from the action of holonomy translates to
the global structure of the variety.

The varieties for which $-K_X$ is ample are called Fano. Due to the
result of Campana, \cite{Campana}, and Koll\'ar, Miyaoka and Mori,
\cite{KMM}, we know that there is only a finite number of deformation
types of smooth Fano varieties in every dimension. They can be studied
well in the framework of the minimal model program, yet their
classification is a challenging problem especially in the case when
their Picard group is ${\mathbb Z}$. On the other hand, since $-K_X=\det T_X$ we can ask
questions about positivity of $T_X$ itself.

Given a vector bundle $\cE$ on $X$, we denote by $\P(\cE)$ the
Grothendieck projectivization of $\cE$, that is, the projective
bundle
$$
\P(\cE):=\Proj_X\left(\bigoplus_{r\geq
    0}S^r\cE\right)\stackrel{p}{\longrightarrow}X
$$
with $\cO(1):=\cO_{\P(\cE)}(1)$ denoting the relative hyperplane
section bundle which satisfies $p_*\cO(1)=\cE$. We say that $\cE$ is
ample if $\cO_{\P(\cE)}(1)$ is ample on $\P(\cE)$. The theorem of Mori
\cite{Mori}, which is a cornerstone of the minimal model program,
asserts that $\P^n$ is the only manifold of dimension $n$ with ample
tangent bundle, as it was conjectured by Hartshorne. An analytic
counterpart of Hartshorne's conjecture is known as Frankel conjecture:
every compact K\"ahler manifold of dimension $n$ with positive
bisectional curvature is biholomorphic to $\P^n$. An analytic proof of
this conjecture was provided by Siu and Yau in \cite{SiuYau}.

Nefness is an (algebro-geometric) positivity property, that appears as
a natural generalization of the concept of ampleness. More concretely,
given a line bundle $L$ on a projective variety $X$, we say that $L$
is {\em nef} if $L \cdot C\geq 0$ for every irreducible curve $C\subset
X$. In other words: nef line bundles are those whose numerical classes
are limits of ample classes.

Generalizing the definition of nefness from line bundles to vector
bundles of any rank we have the following definition:

\begin{definition}\label{def:nef} A vector bundle $\cE$ is {\it nef}
  if and only if the tautological line bundle $\cO(1)$ of $\P(\cE)$
 is a nef line bundle.
\end{definition}

We refer the interested reader to \cite{CP} for an account on the
general properties of nef vector bundles.
The following conjecture formulated and proved for complex 3-folds in
\cite{CP} naturally extends the one by Hartshorne.

\begin{conjecture}[Campana--Peternell Conjecture]\label{conj:CPconj}
  Any Fano manifold whose tangent bundle is nef is rational homogeneous.
\end{conjecture}

We note that an apparently harder question about all manifolds with
nef tangent bundle reduces to the one above because of a result by
Demailly, Peternell and Schneider, \cite{DPS}, who proved the
following: Any compact K\"ahler manifold with nef tangent bundle
admits a finite \'etale cover with smooth Albanese map whose
fibers are Fano manifolds with nef tangent bundle. We also note that, in the framework of complex geometry,
K\"ahler manifolds with nonnegative bisectional curvature (a condition known to be stronger than nefness) have been characterized by Mok, see
\cite{Mk0}; within the class of Fano manifolds of Picard number one, they correspond to irreducible Hermitian symmetric spaces. Finally, \ref{conj:CPconj} is known to be true for toric varieties; indeed, by the work of Fujino and Sato \cite{FS}, the only smooth toric variety whose nef and pseudoeffective cones coincide are products of projective spaces.

Another version of the same problem is the following.

\begin{conjecture}\label{conj:CPconj2}
  Let $X$ be a Fano manifold, and assume that $T_X$ is nef. Then $T_X$ is globally generated.
\end{conjecture}

For brevity let us introduce the following definition:

\begin{definition}\label{def:CP} A smooth complex Fano manifold with
  nef tangent bundle will be called a {\em CP-manifold}.
\end{definition}

According to \cite{Pandharipande} a projective manifold is convex if
every morphism $f: \P^1\rightarrow X$ is unobstructed which means that
$H^1(\P^1,f^*TX)=0$. Clearly, every CP-manifold is convex. Thus
a natural extension of Campana--Peternell conjecture concerns convex
Fano manifolds, see \cite{Pandharipande} for a discussion of this
problem.

\medskip

In the present survey we discuss Campana--Peternell conjecture and
related questions.  The structure of the paper is the following. Since
the problem is to find out whether one may recognize homogeneity
properties on a CP-manifold, we start by recalling some basic facts on
homogeneous manifolds in Section \ref{sec:Prelimi}. The section is
completed with a review of some of the properties of families of rational
curves on algebraic varieties, that we will need later. In Section
\ref{sec:CPvar} we will show that every contraction of a CP-manifold
$X$ is smooth and that the Mori cone of $X$ is simplicial, two
important properties that were known to hold for rational homogenous
spaces.  The fourth section contains some positive answers to the
conjecture in low dimensions, including the classification of
CP-manifolds that have two $\P^1$-fibrations.

In general, it is not even known whether the nefness of $T_X$ implies
its semiampleness. If this property holds, one might study the variety $\P(T_X)$
by looking at the contraction associated to the tautological line
bundle $\cO(1)$. In fact, if this morphism exists, it satisfies a number of
interesting properties, that may eventually lead us to conclude the
homogeneity of $X$ under some extra assumptions. This is the case for
instance if $T_X$ is big and $1$-ample. We discuss this in Section
\ref{sec:symp}.

Finally, the last section reviews some results recently obtained by
the authors, and discuss the way in which these results may be used to
attack the problem of Campana-Peternell.

\subsection{Glossary of notations}
{\hspace{0.2cm}}

\begin{flushleft}
\begin{longtable}{l p{8.8cm}}
$X$& Smooth complex projective variety\\
$\Omega_X$& Sheaf of differentials of $X$\\
$T_X=\Omega_X^\vee$& Tangent bundle of $X$\\
$m$& Dimension of $X$\\
$\cO(K_X)=\bigwedge^m\Omega_X$& Canonical line bundle of $X$\\
$\Pic(X)$& Picard group of $X$\\
$N^1(X)=(\Pic(X)/\equiv)\otimes_\Z\R$& Vector space of numerical classes of $\R$-divisors\\
$N_1(X)(=N^1(X)^\vee)$& Vector space of numerical classes of real $1$-cycles\\
$n=\rho(X)=\dim(N^1(X))$& Picard number of $X$\\
$\NE(X)$& Mori cone of $X$\\
$\Nef(X)$& Nef cone of $X$\\
$R_1,R_2,\dots$& Extremal rays of $\cNE{X}$\\
$\Gamma_1,\Gamma_2,\dots$& Minimal rational curves generating $R_1,R_2,\dots$\\
$p_i:\cU_i\to\cM_i$& Family of deformations of $\Gamma_i$\\
$q_i:\cU_i\to X$& Evaluation morphism\\
$\cC_{i,x}\subset\P(\Omega_{X,x})$& VMRT of the family $\cM_i$ at $x$\\
$\pi_I:X\to X_I$& Contraction of the face generated by $R_i$, $i\in I$\\
$K_I$& Relative canonical divisor of $\pi_I$\\
$\P(\cE)=\Proj_X\left(\Sym(\cE)\right)$\hspace{-0.2cm}& (Grothendieck) Projectivization of a bundle $\cE$\\
$\cO(1)=\cO_{\P(\cE)}(1)$& Tautological line bundle\\
$\cX=\P(T_X)$& Projectivization of the tangent bundle of $X$\\
$\phi:\cX\to \cY$& Crepant contraction of $\cX$\\ 
$\overline{\Gamma}_i$& Minimal section of $\cX$ over $\Gamma_i$\\
$\overline{p}_i:\overline{\cU}_i\to\overline{\cM}_i$& Family of deformations of $\overline{\Gamma}_i$\\
$\overline{q}_i:\overline{\cU}_i\to \cX$& Evaluation morphism\\
$E(a_1^{k_1},\dots,a_r^{k_r})$& $\bigoplus_{j=1}^r\cO_{\P^1}(a_j)^{\oplus k_j}$\\

\end{longtable}
\end{flushleft} 

\section{Preliminaries}\label{sec:Prelimi}

For the reader's convenience we will recall in this preliminary section some basic background on rational homogeneous manifolds and the positivity of their tangent bundles. Then we will briefly review some well known results on deformations of rational curves on an algebraic variety $X$, paying special attention to the case in which $X$ is a CP-manifold.

\subsection{Homogeneous manifolds}\label{ssec:homogman}

A smooth complex variety $X$ is said to be {\it homogeneous} if $X$ admits a transitive action of an algebraic group $G$. In this paper, we will only consider projective homogeneous manifolds, but we will refer to them simply as {\it homogeneous manifolds}.

For any projective manifold $X$, the automorphism group scheme ${\Aut}(X)$ is defined as the scheme representing the automorphism functor \cite[Theorem~3.7]{MO}, and the Lie algebra of ${\Aut}(X)$ is identified with that of the derivations of $\cO_X$:
\begin{eqnarray}
 {\rm Lie}({\Aut}(X)) \cong H^0(X, T_X).
\end{eqnarray}
Since our base field has characteristic zero, we know that ${\Aut}(X)$ is reduced thanks to Cartier \cite{Oo}.
This leads to the following characterization of homogeneous manifolds from the viewpoint of the spannedness of the tangent bundle.

\begin{proposition} Let $X$ be a projective manifold. Then $X$ is homogeneous if and only if $T_X$ is globally generated.
In particular, the tangent bundle of a homogeneous manifold is nef.
\end{proposition}

\begin{proof} Let $G$ be the identity component of $\Aut(X)$. Then $G$ is an algebraic group with Lie algebra $H^0(X, T_X)$.
The evaluation map is denoted by
$${\rm ev}: H^0(X, T_X) \otimes \cO_X \to T_X.$$
On the other hand,
for any point $x \in X$, consider the orbit map
$$
\mu_x: G \to X;\,\,\, g \mapsto gx.
$$
Since the differential of $\mu_x$ at the identity $e \in G$ coincides with the evaluation at $x$, then our claim follows.
\end{proof}

The following structure theorem, due to A. Borel and R. Remmert, tells us that there are basically two types of homogeneous manifolds: projective algebraic groups (i.e. {\it abelian varieties}) and quotients of simple Lie groups.

\begin{theorem}[\cite{BR}] Any homogeneous manifold $X$ is isomorphic to a product $A \times Y_1\times\dots\times Y_k$, where $A$ is an abelian variety and every $Y_i$ is projective variety of the form $G_i/P_i$, where $G_i$ is a simple Lie group and $P_i\subset G_i$ is a parabolic subgroup.
\end{theorem}

A product of simple Lie groups is called {\it semisimple} (equivalently, they are usually defined as the Lie groups that have no notrivial normal connected solvable subgroups), and a projective quotient $G/P$ of a semisimple Lie group is called a {\it rational homogeneous manifold}.

Abelian varieties and rational homogeneous manifolds may be already distinguished at the level of the positivity of their tangent spaces.
In fact, the group structure of an abelian variety forces its tangent bundle (hence its canonical divisor) to be trivial, whereas the fact that semisimple Lie groups are affine varieties implies that their projective quotients are Fano manifolds. In other words, rational homogeneous manifolds are the first (conjecturally, the only) examples of CP-manifolds.

\begin{proposition}\label{rathomogfano} For a rational homogeneous manifold $X$, the anticanonical divisor $-K_X$ is ample and globally generated.
\end{proposition}

\begin{proof} The following argument is due to S. Mori (see for example \cite[V. Theorem~1.4]{kollar}).
Writing $X$ as a quotient $G/P$ of a semisimple Lie group $G$, the fact that $G$ is affine and $G/P$ is projective implies that $P$ is positive dimensional.

On the other hand, $P$ might be identified with the isotropy subgroup of a point $x\in X$. Consider the orbit map $\mu_{x}: G \to X$, $g \mapsto gx$.

Since its differential $d\mu_{x}$ at the identity $e$ sends $T_{P,e}$ to $\{{\bf{0}}\} \subset T_{X,x},$
any section $\sigma \in H^0(X, T_X) \setminus \{{\bf 0} \}$ contained in $T_{P,e}$ vanishes at $x$. Furthermore, we have $$h^0(X, T_X)=\dim G=\dim X+ \dim P>\dim X.$$ These facts imply that we have a section $s \in H^0(X, \cO_X(-K_X)) \setminus \{{\bf 0} \}$ which vanishes at $x$, but is not everywhere zero. Since $X$ is homogeneous under $G$, $\cO_X(-K_X)$ is non-trivial and globally generated.

Let us consider the morphism $\varphi : X \to Y$ defined by the linear system $|\cO_X(-K_X)|$, which is $G$-equivariant, and a fiber $F=\varphi^{-1}(y)$. If $P'$ is the stabilizer of $\varphi(y)$, then $F$ is homogeneous under $P'$. If $F^0$ is a positive-dimensional irreducible component of $F$, then $-K_{F^0}$ is non-trivial by the same argument as above. However this contradicts the fact that
 $$\cO_F(-K_{F^0})=\cO_X(-K_X)|_{F^0}=\varphi^{\ast}\cO_Y(1)|_{F^0}=\cO_{F^0}.$$ Thus $\varphi$ is finite. As a
consequence, $-K_X$ is ample.
\end{proof}

\subsection{Rational homogeneous manifolds and Dynkin diagrams}\label{ssec:rathom}

One of the most remarkable properties of semisimple Lie groups and rational homogeneous manifolds is that they may be completely described in terms of certain  combinatorial objects named Dynkin diagrams, whose (brief) description is the goal of this section.

Let us then consider a semisimple Lie group $G$, and let us denote by $\fg$ its associated Lie algebra, that determines $G$ via the exponential map.

\subsubsection*{\bf Cartan decomposition}\label{sssec:decomp}
We start by choosing a {\it Cartan subalgebra}, i.e. an abelian subalgebra $\fh\subset\fg$ of maximal dimension. Being $\fh$ abelian, its adjoint action on $\fg$
defines an eigenspace decomposition,
$$
\fg=\fh\oplus\bigoplus_{\alpha\in\fh^\vee\setminus\{0\}} \fg_\alpha,\mbox{ where } \fg_\alpha:=\left\{g\in\fg\,|\,[h,g]=\alpha(h)g,\mbox{ for all }h\in\fh\right\},
$$
called {\it Cartan decomposition} of $\fg$. The
elements $\alpha\in\fh^\vee\setminus\{0\}$ for which $\fg_\alpha\neq 0$ are called {\it roots} of $\fg$, and the set of these elements will be denoted by $\Phi$.
The eigenspaces $\fg_\alpha$ (so-called {\it root spaces}) are one-dimensional, for every $\alpha\in\Phi$.
Moreover, $\Phi$ satisfies that given $\alpha\in \Phi$, then $k\alpha\in\Phi$ iff $k=\pm 1$, and that $[\fg_\alpha,\fg_\beta]=\fg_{\alpha+\beta}$ iff $0\neq\alpha+\beta\in\Phi$. In particular the set $\Phi$ contains all the information necessary to reconstruct completely the Lie algebra $\fg$ (thus the group $G$) out of it.

\subsubsection*{\bf Root system and Weyl group of $\fg$}\label{ssseweyl}

The key property of $\Phi$, that allows us to list all the possible semisimple Lie algebras, is the behaviour of its group of symmetries.

The {\it Killing form} $\kappa(X,Y):=\tr(\ad_X\circ\ad_Y)$ defines, on every semisimple Lie algebra, a nondegenerate
bilinear form on $\fh$, whose restriction to the real vector space $E$ generated by $\Phi$ is positive definite. Within the euclidean space $(E,\kappa)$, every root $\alpha\in\Phi$ defines a {\it reflection} $\sigma_\alpha$, given by:
$$
\sigma_\alpha(x)=x-\sga x,\alpha\sgc\alpha,\quad\mbox{where}\quad\sga x,\alpha\sgc:=2\dfrac{\kappa(x,\alpha)}{\kappa(\alpha,\alpha)}.
$$
One may then show that the group $W\subset\so(E,\kappa)$ generated by the $\sigma_\alpha$'s,  called the {\it Weyl group of $\fg$}, leaves the set $\Phi$ invariant. Furthermore, one may show that $\sga \alpha,\beta\sgc$ is an integer, a property that allows us to say that $\Phi$ is a {\it root system} in $(E,\kappa)$ (see \cite[VI.1 Def. 1]{Bourb}). The important point to remark here is that, as we will see precisely in the next paragraph, a reduced root system may be reconstructed out of a very small amount of data: a well chosen basis and the reflections with respect to its elements (that will be determined by a matrix of integers).

\subsubsection*{\bf Cartan matrix of $\fg$}\label{sssec:cartan}

Set $n:=\dim_\C(\fh)$ and $D:=\{1,2,\dots,n\}$.
It is known that for every root system one may always find a {\it base of simple roots}, i.e. a basis of $\fh^\vee$ formed by elements of $\Phi$ satisfying that the coordinates of every element of $\Phi$ are integers, all of them nonnegative or all of them nonpositive.
A base $\Delta=\left\{\alpha_i\right\}_{i\in D}$ provides a decomposition of the set of roots according to their sign $\Phi=\Phi^+\cup\Phi^-$, where $\Phi^-=-\Phi^+$.  Moreover, every positive root can be obtained from simple roots by means of reflections $\sigma_{\alpha_i}$. It is then clear that the matrix $M$ whose coefficients are $\sga\alpha_i,\alpha_j\sgc$, the so-called  {\it Cartan matrix of $\fg$}, encodes the necessary information to reconstruct $\fg$ from the set of simple roots $\Delta$.

\subsubsection*{\bf Dynkin diagrams}\label{ssec:dynkinhomo}

The coefficients of the Cartan matrix $M$ of $\fg$ are subject to arithmetic restrictions:
\begin{itemize}
\item $\sga \alpha_i,\alpha_i\sgc=2$ for all $i$,
\item $\sga \alpha_i,\alpha_j\sgc=0$ if and only if $\sga \alpha_j,\alpha_i\sgc=0$, and
\item if $\sga \alpha_i,\alpha_j\sgc\neq 0$, $i \neq j$, then $\sga \alpha_i,\alpha_j\sgc\in\Z^-$ and $ \sga \alpha_i,\alpha_j\sgc\sga\alpha_j,\alpha_i\sgc=1,2$ or $3$.
\end{itemize}
which allow us to represent $M$ by a {\it Dynkin diagram}, that we denote by $\cD$: it consists of a graph whose set of nodes is $D$ and where the nodes $i$ and $j$ are joined by $\sga \alpha_j,\alpha_i\sgc\sga \alpha_i,\alpha_j\sgc$ edges. When two nodes $i$ and $j$ are joined by a double or triple edge, we add to it an arrow, pointing to $i$ if $\sga \alpha_i,\alpha_j\sgc>\sga \alpha_j,\alpha_i\sgc$. One may prove that $\cD$ is independent of the choices made (Cartan subalgebra, base of simple roots), hence, summing up:
\begin{theorem}
 There is a one to one correspondence between isomorphism classes of semisimple Lie algebras and Dynkin diagrams of reduced root systems.
\end{theorem}

Furthermore, the classification theorem of root systems tells us that every reduced root system is a disjoint union of mutually orthogonal irreducible root subsystems, each of them corresponding to one of the {\it connected finite Dynkin diagrams} ${\rm A}_n$, ${\rm B}_n$, ${\rm C}_n$, ${\rm D}_n$ ($n\in\N$), ${\rm E}_6$, ${\rm E}_7$, ${\rm E}_8$, ${\rm F}_4$, ${\rm G}_2$:


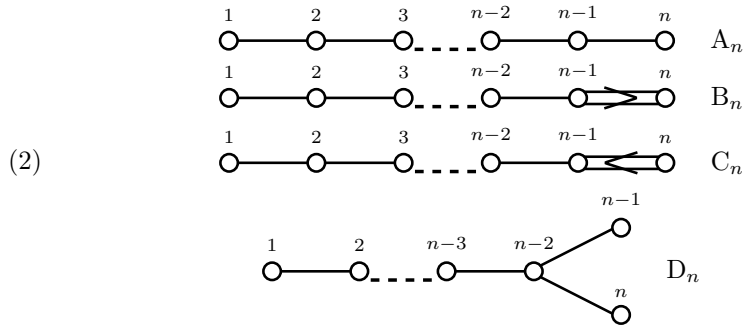
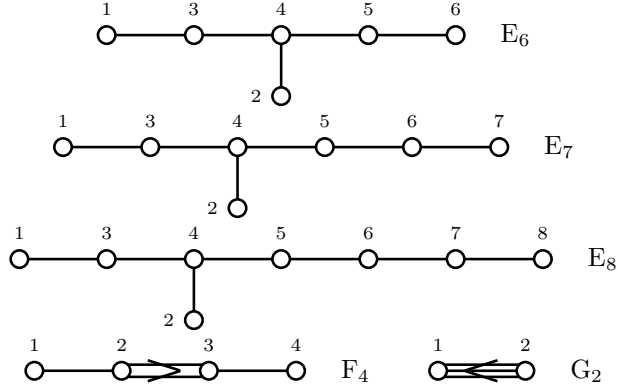
\begin{equation}\label{eq:dynkins}
\begin{array}{c}\vspace{-0.1cm}
\ifx\du\undefined
  \newlength{\du}
\fi
\setlength{\du}{3.3\unitlength}
\begin{tikzpicture}
\pgftransformxscale{1.000000}
\pgftransformyscale{1.000000}

\definecolor{dialinecolor}{rgb}{0.000000, 0.000000, 0.000000} 
\pgfsetstrokecolor{dialinecolor}
\definecolor{dialinecolor}{rgb}{0.000000, 0.000000, 0.000000} 
\pgfsetfillcolor{dialinecolor}


\pgfsetlinewidth{0.300000\du}
\pgfsetdash{}{0pt}
\pgfsetdash{}{0pt}

\pgfpathellipse{\pgfpoint{-6\du}{0\du}}{\pgfpoint{1\du}{0\du}}{\pgfpoint{0\du}{1\du}}
\pgfusepath{stroke}
\node at (-6\du,0\du){};

\pgfpathellipse{\pgfpoint{4\du}{0\du}}{\pgfpoint{1\du}{0\du}}{\pgfpoint{0\du}{1\du}}
\pgfusepath{stroke}
\node at (4\du,0\du){};

\pgfpathellipse{\pgfpoint{14\du}{0\du}}{\pgfpoint{1\du}{0\du}}{\pgfpoint{0\du}{1\du}}
\pgfusepath{stroke}
\node at (14\du,0\du){};

\pgfpathellipse{\pgfpoint{24\du}{0\du}}{\pgfpoint{1\du}{0\du}}{\pgfpoint{0\du}{1\du}}
\pgfusepath{stroke}
\node at (24\du,0\du){};

\pgfpathellipse{\pgfpoint{34\du}{0\du}}{\pgfpoint{1\du}{0\du}}{\pgfpoint{0\du}{1\du}}
\pgfusepath{stroke}
\node at (34\du,0\du){};

\pgfpathellipse{\pgfpoint{44\du}{0\du}}{\pgfpoint{1\du}{0\du}}{\pgfpoint{0\du}{1\du}}
\pgfusepath{stroke}
\node at (44\du,0\du){};

\pgfsetlinewidth{0.300000\du}
\pgfsetdash{}{0pt}
\pgfsetdash{}{0pt}
\pgfsetbuttcap

{\draw (-5\du,0\du)--(3\du,0\du);}
{\draw (5\du,0\du)--(13\du,0\du);}
{\draw (25\du,0\du)--(33\du,0\du);}
{\draw (35\du,0\du)--(43\du,0\du);}



\pgfsetlinewidth{0.400000\du}
\pgfsetdash{{1.000000\du}{1.000000\du}}{0\du}
\pgfsetdash{{1.000000\du}{1.000000\du}}{0\du}
\pgfsetbuttcap
{\draw (15.3\du,-1\du)--(23\du,-1\du);}

\node[anchor=west] at (48\du,0\du){${\rm A}_n$};

\node[anchor=south] at (-6\du,1.1\du){$\scriptstyle 1$};

\node[anchor=south] at (4\du,1.1\du){$\scriptstyle 2$};

\node[anchor=south] at (14\du,1.1\du){$\scriptstyle 3$};

\node[anchor=south] at (24\du,1.1\du){$\scriptstyle n-2$};

\node[anchor=south] at (34\du,1.1\du){$\scriptstyle n-1$};

\node[anchor=south] at (44\du,1.1\du){$\scriptstyle n$};

\end{tikzpicture} \vspace{-0.1cm}\\
\ifx\du\undefined
  \newlength{\du}
\fi
\setlength{\du}{3.3\unitlength}
\begin{tikzpicture}
\pgftransformxscale{1.000000}
\pgftransformyscale{1.000000}

\definecolor{dialinecolor}{rgb}{0.000000, 0.000000, 0.000000} 
\pgfsetstrokecolor{dialinecolor}
\definecolor{dialinecolor}{rgb}{0.000000, 0.000000, 0.000000} 
\pgfsetfillcolor{dialinecolor}


\pgfsetlinewidth{0.300000\du}
\pgfsetdash{}{0pt}
\pgfsetdash{}{0pt}

\pgfpathellipse{\pgfpoint{-6\du}{0\du}}{\pgfpoint{1\du}{0\du}}{\pgfpoint{0\du}{1\du}}
\pgfusepath{stroke}
\node at (-6\du,0\du){};

\pgfpathellipse{\pgfpoint{4\du}{0\du}}{\pgfpoint{1\du}{0\du}}{\pgfpoint{0\du}{1\du}}
\pgfusepath{stroke}
\node at (4\du,0\du){};

\pgfpathellipse{\pgfpoint{14\du}{0\du}}{\pgfpoint{1\du}{0\du}}{\pgfpoint{0\du}{1\du}}
\pgfusepath{stroke}
\node at (14\du,0\du){};

\pgfpathellipse{\pgfpoint{24\du}{0\du}}{\pgfpoint{1\du}{0\du}}{\pgfpoint{0\du}{1\du}}
\pgfusepath{stroke}
\node at (24\du,0\du){};

\pgfpathellipse{\pgfpoint{34\du}{0\du}}{\pgfpoint{1\du}{0\du}}{\pgfpoint{0\du}{1\du}}
\pgfusepath{stroke}
\node at (34\du,0\du){};

\pgfpathellipse{\pgfpoint{44\du}{0\du}}{\pgfpoint{1\du}{0\du}}{\pgfpoint{0\du}{1\du}}
\pgfusepath{stroke}
\node at (44\du,0\du){};

\pgfsetlinewidth{0.300000\du}
\pgfsetdash{}{0pt}
\pgfsetdash{}{0pt}
\pgfsetbuttcap

{\draw (-5\du,0\du)--(3\du,0\du);}
{\draw (5\du,0\du)--(13\du,0\du);}
{\draw (25\du,0\du)--(33\du,0\du);}
{\draw (34.65\du,0.7\du)--(43.35\du,0.7\du);}
{\draw (34.65\du,-0.7\du)--(43.35\du,-0.7\du);}


{\pgfsetcornersarced{\pgfpoint{0.300000\du}{0.300000\du}}\definecolor{dialinecolor}{rgb}{0.000000, 0.000000, 0.000000}
\pgfsetstrokecolor{dialinecolor}
\draw (37\du,-1.2\du)--(40.8\du,0\du)--(37\du,1.2\du);}

\pgfsetlinewidth{0.400000\du}
\pgfsetdash{{1.000000\du}{1.000000\du}}{0\du}
\pgfsetdash{{1.000000\du}{1.000000\du}}{0\du}
\pgfsetbuttcap
{\draw (15.3\du,-1\du)--(23\du,-1\du);}

\node[anchor=west] at (48\du,0\du){${\rm B}_n$};

\node[anchor=south] at (-6\du,1.1\du){$\scriptstyle 1$};

\node[anchor=south] at (4\du,1.1\du){$\scriptstyle 2$};

\node[anchor=south] at (14\du,1.1\du){$\scriptstyle 3$};

\node[anchor=south] at (24\du,1.1\du){$\scriptstyle n-2$};

\node[anchor=south] at (34\du,1.1\du){$\scriptstyle n-1$};

\node[anchor=south] at (44\du,1.1\du){$\scriptstyle n$};

\end{tikzpicture} \vspace{-0.1cm}\\
\ifx\du\undefined
  \newlength{\du}
\fi
\setlength{\du}{3.3\unitlength}
\begin{tikzpicture}
\pgftransformxscale{1.000000}
\pgftransformyscale{1.000000}

\definecolor{dialinecolor}{rgb}{0.000000, 0.000000, 0.000000} 
\pgfsetstrokecolor{dialinecolor}
\definecolor{dialinecolor}{rgb}{0.000000, 0.000000, 0.000000} 
\pgfsetfillcolor{dialinecolor}


\pgfsetlinewidth{0.300000\du}
\pgfsetdash{}{0pt}
\pgfsetdash{}{0pt}

\pgfpathellipse{\pgfpoint{-6\du}{0\du}}{\pgfpoint{1\du}{0\du}}{\pgfpoint{0\du}{1\du}}
\pgfusepath{stroke}
\node at (-6\du,0\du){};

\pgfpathellipse{\pgfpoint{4\du}{0\du}}{\pgfpoint{1\du}{0\du}}{\pgfpoint{0\du}{1\du}}
\pgfusepath{stroke}
\node at (4\du,0\du){};

\pgfpathellipse{\pgfpoint{14\du}{0\du}}{\pgfpoint{1\du}{0\du}}{\pgfpoint{0\du}{1\du}}
\pgfusepath{stroke}
\node at (14\du,0\du){};

\pgfpathellipse{\pgfpoint{24\du}{0\du}}{\pgfpoint{1\du}{0\du}}{\pgfpoint{0\du}{1\du}}
\pgfusepath{stroke}
\node at (24\du,0\du){};

\pgfpathellipse{\pgfpoint{34\du}{0\du}}{\pgfpoint{1\du}{0\du}}{\pgfpoint{0\du}{1\du}}
\pgfusepath{stroke}
\node at (34\du,0\du){};

\pgfpathellipse{\pgfpoint{44\du}{0\du}}{\pgfpoint{1\du}{0\du}}{\pgfpoint{0\du}{1\du}}
\pgfusepath{stroke}
\node at (44\du,0\du){};

\pgfsetlinewidth{0.300000\du}
\pgfsetdash{}{0pt}
\pgfsetdash{}{0pt}
\pgfsetbuttcap

{\draw (-5\du,0\du)--(3\du,0\du);}
{\draw (5\du,0\du)--(13\du,0\du);}
{\draw (25\du,0\du)--(33\du,0\du);}
{\draw (34.65\du,0.7\du)--(43.35\du,0.7\du);}
{\draw (34.65\du,-0.7\du)--(43.35\du,-0.7\du);}


{\pgfsetcornersarced{\pgfpoint{0.300000\du}{0.300000\du}}\definecolor{dialinecolor}{rgb}{0.000000, 0.000000, 0.000000}
\pgfsetstrokecolor{dialinecolor}
\draw (40.8\du,-1.2\du)--(37\du,0\du)--(40.8\du,1.2\du);}

\pgfsetlinewidth{0.400000\du}
\pgfsetdash{{1.000000\du}{1.000000\du}}{0\du}
\pgfsetdash{{1.000000\du}{1.00000\du}}{0\du}
\pgfsetbuttcap
{\draw (15.3\du,-1\du)--(23\du,-1\du);}

\node[anchor=west] at (48\du,0\du){${\rm C}_n$};

\node[anchor=south] at (-6\du,1.1\du){$\scriptstyle 1$};

\node[anchor=south] at (4\du,1.1\du){$\scriptstyle 2$};

\node[anchor=south] at (14\du,1.1\du){$\scriptstyle 3$};

\node[anchor=south] at (24\du,1.1\du){$\scriptstyle n-2$};

\node[anchor=south] at (34\du,1.1\du){$\scriptstyle n-1$};

\node[anchor=south] at (44\du,1.1\du){$\scriptstyle n$};

\end{tikzpicture} \vspace{-0.1cm}\\\ifx\du\undefined
  \newlength{\du}
\fi
\setlength{\du}{3.3\unitlength}
\begin{tikzpicture}
\pgftransformxscale{1.000000}
\pgftransformyscale{1.000000}

\definecolor{dialinecolor}{rgb}{0.000000, 0.000000, 0.000000} 
\pgfsetstrokecolor{dialinecolor}
\definecolor{dialinecolor}{rgb}{0.000000, 0.000000, 0.000000} 
\pgfsetfillcolor{dialinecolor}


\pgfsetlinewidth{0.300000\du}
\pgfsetdash{}{0pt}
\pgfsetdash{}{0pt}

\pgfpathellipse{\pgfpoint{-6\du}{0\du}}{\pgfpoint{1\du}{0\du}}{\pgfpoint{0\du}{1\du}}
\pgfusepath{stroke}
\node at (-6\du,0\du){};

\pgfpathellipse{\pgfpoint{4\du}{0\du}}{\pgfpoint{1\du}{0\du}}{\pgfpoint{0\du}{1\du}}
\pgfusepath{stroke}
\node at (4\du,0\du){};

\pgfpathellipse{\pgfpoint{14\du}{0\du}}{\pgfpoint{1\du}{0\du}}{\pgfpoint{0\du}{1\du}}
\pgfusepath{stroke}
\node at (14\du,0\du){};

\pgfpathellipse{\pgfpoint{24\du}{0\du}}{\pgfpoint{1\du}{0\du}}{\pgfpoint{0\du}{1\du}}
\pgfusepath{stroke}
\node at (24\du,0\du){};

\pgfpathellipse{\pgfpoint{34\du}{5\du}}{\pgfpoint{1\du}{0\du}}{\pgfpoint{0\du}{1\du}}
\pgfusepath{stroke}
\node at (34\du,5\du){};

\pgfpathellipse{\pgfpoint{34\du}{-5\du}}{\pgfpoint{1\du}{0\du}}{\pgfpoint{0\du}{1\du}}
\pgfusepath{stroke}
\node at (34\du,-5\du){};

\pgfsetlinewidth{0.300000\du}
\pgfsetdash{}{0pt}
\pgfsetdash{}{0pt}
\pgfsetbuttcap

{\draw (-5\du,0\du)--(3\du,0\du);}
{\draw (15\du,0\du)--(23\du,0\du);}
{\draw (24.65\du,0.6\du)--(33.1\du,4.9\du);}
{\draw (24.65\du,-0.6\du)--(33.1\du,-4.9\du);}



\pgfsetlinewidth{0.400000\du}
\pgfsetdash{{1.000000\du}{1.000000\du}}{0\du}
\pgfsetdash{{1.000000\du}{1.000000\du}}{0\du}
\pgfsetbuttcap
{\draw (5.3\du,-1\du)--(13\du,-1\du);}

\node[anchor=west] at (38\du,0\du){${\rm D}_n$};

\node[anchor=south] at (-6\du,1.1\du){$\scriptstyle 1$};

\node[anchor=south] at (4\du,1.1\du){$\scriptstyle 2$};

\node[anchor=south] at (14\du,1.1\du){$\scriptstyle n-3$};

\node[anchor=south] at (24\du,1.1\du){$\scriptstyle n-2$};

\node[anchor=south] at (34\du,6.1\du){$\scriptstyle n-1$};

\node[anchor=south] at (34\du,-4.1\du){$\scriptstyle n$};

\end{tikzpicture} \vspace{-0.1cm}
\end{array}
\end{equation}
\vspace{-0.2cm}
\begin{equation*}
\begin{array}{c}\vspace{-0.1cm}
\ifx\du\undefined
  \newlength{\du}
\fi
\setlength{\du}{3.3\unitlength}
\begin{tikzpicture}
\pgftransformxscale{1.000000}
\pgftransformyscale{1.000000}

\definecolor{dialinecolor}{rgb}{0.000000, 0.000000, 0.000000} 
\pgfsetstrokecolor{dialinecolor}
\definecolor{dialinecolor}{rgb}{0.000000, 0.000000, 0.000000} 
\pgfsetfillcolor{dialinecolor}


\pgfsetlinewidth{0.300000\du}
\pgfsetdash{}{0pt}
\pgfsetdash{}{0pt}

\pgfpathellipse{\pgfpoint{-6\du}{0\du}}{\pgfpoint{1\du}{0\du}}{\pgfpoint{0\du}{1\du}}
\pgfusepath{stroke}
\node at (-6\du,0\du){};

\pgfpathellipse{\pgfpoint{14\du}{-7\du}}{\pgfpoint{1\du}{0\du}}{\pgfpoint{0\du}{1\du}}
\pgfusepath{stroke}
\node at (14\du,-7\du){};

\pgfpathellipse{\pgfpoint{4\du}{0\du}}{\pgfpoint{1\du}{0\du}}{\pgfpoint{0\du}{1\du}}
\pgfusepath{stroke}
\node at (4\du,0\du){};

\pgfpathellipse{\pgfpoint{14\du}{0\du}}{\pgfpoint{1\du}{0\du}}{\pgfpoint{0\du}{1\du}}
\pgfusepath{stroke}
\node at (14\du,0\du){};

\pgfpathellipse{\pgfpoint{24\du}{0\du}}{\pgfpoint{1\du}{0\du}}{\pgfpoint{0\du}{1\du}}
\pgfusepath{stroke}
\node at (24\du,0\du){};

\pgfpathellipse{\pgfpoint{34\du}{0\du}}{\pgfpoint{1\du}{0\du}}{\pgfpoint{0\du}{1\du}}
\pgfusepath{stroke}
\node at (34\du,0\du){};

\pgfsetlinewidth{0.300000\du}
\pgfsetdash{}{0pt}
\pgfsetdash{}{0pt}
\pgfsetbuttcap

{\draw (-5\du,0\du)--(3\du,0\du);}
{\draw (5\du,0\du)--(13\du,0\du);}
{\draw (15\du,0\du)--(23\du,0\du);}
{\draw (25\du,0\du)--(33\du,0\du);}
{\draw (14\du,-1\du)--(14\du,-6\du);}

\node[anchor=west] at (38\du,0\du){${\rm E}_6$};

\node[anchor=south] at (-6\du,1.1\du){$\scriptstyle 1$};

\node[anchor=south] at (4\du,1.1\du){$\scriptstyle 3$};

\node[anchor=south] at (14\du,1.1\du){$\scriptstyle 4$};

\node[anchor=south] at (24\du,1.1\du){$\scriptstyle 5$};

\node[anchor=south] at (34\du,1.1\du){$\scriptstyle 6$};

\node[anchor=east] at (12.9\du,-7\du){$\scriptstyle 2$};

\end{tikzpicture} \vspace{-0.1cm}
\end{array}
\end{equation*}
\vspace{-0.2cm}
\begin{equation*}
\begin{array}{c}\vspace{-0.1cm}
\ifx\du\undefined
  \newlength{\du}
\fi
\setlength{\du}{3.3\unitlength}
\begin{tikzpicture}
\pgftransformxscale{1.000000}
\pgftransformyscale{1.000000}

\definecolor{dialinecolor}{rgb}{0.000000, 0.000000, 0.000000} 
\pgfsetstrokecolor{dialinecolor}
\definecolor{dialinecolor}{rgb}{0.000000, 0.000000, 0.000000} 
\pgfsetfillcolor{dialinecolor}


\pgfsetlinewidth{0.300000\du}
\pgfsetdash{}{0pt}
\pgfsetdash{}{0pt}

\pgfpathellipse{\pgfpoint{-6\du}{0\du}}{\pgfpoint{1\du}{0\du}}{\pgfpoint{0\du}{1\du}}
\pgfusepath{stroke}
\node at (-6\du,0\du){};

\pgfpathellipse{\pgfpoint{14\du}{-7\du}}{\pgfpoint{1\du}{0\du}}{\pgfpoint{0\du}{1\du}}
\pgfusepath{stroke}
\node at (14\du,-7\du){};

\pgfpathellipse{\pgfpoint{4\du}{0\du}}{\pgfpoint{1\du}{0\du}}{\pgfpoint{0\du}{1\du}}
\pgfusepath{stroke}
\node at (4\du,0\du){};

\pgfpathellipse{\pgfpoint{14\du}{0\du}}{\pgfpoint{1\du}{0\du}}{\pgfpoint{0\du}{1\du}}
\pgfusepath{stroke}
\node at (14\du,0\du){};

\pgfpathellipse{\pgfpoint{24\du}{0\du}}{\pgfpoint{1\du}{0\du}}{\pgfpoint{0\du}{1\du}}
\pgfusepath{stroke}
\node at (24\du,0\du){};

\pgfpathellipse{\pgfpoint{34\du}{0\du}}{\pgfpoint{1\du}{0\du}}{\pgfpoint{0\du}{1\du}}
\pgfusepath{stroke}
\node at (34\du,0\du){};

\pgfpathellipse{\pgfpoint{44\du}{0\du}}{\pgfpoint{1\du}{0\du}}{\pgfpoint{0\du}{1\du}}
\pgfusepath{stroke}
\node at (44\du,0\du){};

\pgfsetlinewidth{0.300000\du}
\pgfsetdash{}{0pt}
\pgfsetdash{}{0pt}
\pgfsetbuttcap

{\draw (-5\du,0\du)--(3\du,0\du);}
{\draw (5\du,0\du)--(13\du,0\du);}
{\draw (15\du,0\du)--(23\du,0\du);}
{\draw (25\du,0\du)--(33\du,0\du);}
{\draw (35\du,0\du)--(43\du,0\du);}
{\draw (14\du,-1\du)--(14\du,-6\du);}

\node[anchor=west] at (48\du,0\du){${\rm E}_7$};

\node[anchor=south] at (-6\du,1.1\du){$\scriptstyle 1$};

\node[anchor=south] at (4\du,1.1\du){$\scriptstyle 3$};

\node[anchor=south] at (14\du,1.1\du){$\scriptstyle 4$};

\node[anchor=south] at (24\du,1.1\du){$\scriptstyle 5$};

\node[anchor=south] at (34\du,1.1\du){$\scriptstyle 6$};

\node[anchor=south] at (44\du,1.1\du){$\scriptstyle 7$};

\node[anchor=east] at (12.9\du,-7\du){$\scriptstyle 2$};

\end{tikzpicture} \vspace{-0.1cm}
\end{array}
\end{equation*}
\vspace{-0.2cm}
\begin{equation*}
\begin{array}{c}\vspace{-0.1cm}
\ifx\du\undefined
  \newlength{\du}
\fi
\setlength{\du}{3.3\unitlength}
\begin{tikzpicture}
\pgftransformxscale{1.000000}
\pgftransformyscale{1.000000}

\definecolor{dialinecolor}{rgb}{0.000000, 0.000000, 0.000000} 
\pgfsetstrokecolor{dialinecolor}
\definecolor{dialinecolor}{rgb}{0.000000, 0.000000, 0.000000} 
\pgfsetfillcolor{dialinecolor}


\pgfsetlinewidth{0.300000\du}
\pgfsetdash{}{0pt}
\pgfsetdash{}{0pt}

\pgfpathellipse{\pgfpoint{-6\du}{0\du}}{\pgfpoint{1\du}{0\du}}{\pgfpoint{0\du}{1\du}}
\pgfusepath{stroke}
\node at (-6\du,0\du){};

\pgfpathellipse{\pgfpoint{14\du}{-7\du}}{\pgfpoint{1\du}{0\du}}{\pgfpoint{0\du}{1\du}}
\pgfusepath{stroke}
\node at (14\du,-7\du){};

\pgfpathellipse{\pgfpoint{4\du}{0\du}}{\pgfpoint{1\du}{0\du}}{\pgfpoint{0\du}{1\du}}
\pgfusepath{stroke}
\node at (4\du,0\du){};

\pgfpathellipse{\pgfpoint{14\du}{0\du}}{\pgfpoint{1\du}{0\du}}{\pgfpoint{0\du}{1\du}}
\pgfusepath{stroke}
\node at (14\du,0\du){};

\pgfpathellipse{\pgfpoint{24\du}{0\du}}{\pgfpoint{1\du}{0\du}}{\pgfpoint{0\du}{1\du}}
\pgfusepath{stroke}
\node at (24\du,0\du){};

\pgfpathellipse{\pgfpoint{34\du}{0\du}}{\pgfpoint{1\du}{0\du}}{\pgfpoint{0\du}{1\du}}
\pgfusepath{stroke}
\node at (34\du,0\du){};

\pgfpathellipse{\pgfpoint{44\du}{0\du}}{\pgfpoint{1\du}{0\du}}{\pgfpoint{0\du}{1\du}}
\pgfusepath{stroke}
\node at (44\du,0\du){};

\pgfpathellipse{\pgfpoint{54\du}{0\du}}{\pgfpoint{1\du}{0\du}}{\pgfpoint{0\du}{1\du}}
\pgfusepath{stroke}
\node at (54\du,0\du){};

\pgfsetlinewidth{0.300000\du}
\pgfsetdash{}{0pt}
\pgfsetdash{}{0pt}
\pgfsetbuttcap

{\draw (-5\du,0\du)--(3\du,0\du);}
{\draw (5\du,0\du)--(13\du,0\du);}
{\draw (15\du,0\du)--(23\du,0\du);}
{\draw (25\du,0\du)--(33\du,0\du);}
{\draw (35\du,0\du)--(43\du,0\du);}
{\draw (45\du,0\du)--(53\du,0\du);}
{\draw (14\du,-1\du)--(14\du,-6\du);}

\node[anchor=west] at (58\du,0\du){${\rm E}_8$};

\node[anchor=south] at (-6\du,1.1\du){$\scriptstyle 1$};

\node[anchor=south] at (4\du,1.1\du){$\scriptstyle 3$};

\node[anchor=south] at (14\du,1.1\du){$\scriptstyle 4$};

\node[anchor=south] at (24\du,1.1\du){$\scriptstyle 5$};

\node[anchor=south] at (34\du,1.1\du){$\scriptstyle 6$};

\node[anchor=south] at (44\du,1.1\du){$\scriptstyle 7$};

\node[anchor=south] at (54\du,1.1\du){$\scriptstyle 8$};

\node[anchor=east] at (12.9\du,-7\du){$\scriptstyle 2$};

\end{tikzpicture} \vspace{-0.1cm}
\end{array}
\end{equation*}
\vspace{-0.2cm}
\begin{equation*}
\begin{array}{c}\vspace{-0.1cm}
\ifx\du\undefined
  \newlength{\du}
\fi
\setlength{\du}{3.3\unitlength}
\begin{tikzpicture}
\pgftransformxscale{1.000000}
\pgftransformyscale{1.000000}

\definecolor{dialinecolor}{rgb}{0.000000, 0.000000, 0.000000} 
\pgfsetstrokecolor{dialinecolor}
\definecolor{dialinecolor}{rgb}{0.000000, 0.000000, 0.000000} 
\pgfsetfillcolor{dialinecolor}


\pgfsetlinewidth{0.300000\du}
\pgfsetdash{}{0pt}
\pgfsetdash{}{0pt}

\pgfpathellipse{\pgfpoint{-6\du}{0\du}}{\pgfpoint{1\du}{0\du}}{\pgfpoint{0\du}{1\du}}
\pgfusepath{stroke}
\node at (-6\du,0\du){};

\pgfpathellipse{\pgfpoint{4\du}{0\du}}{\pgfpoint{1\du}{0\du}}{\pgfpoint{0\du}{1\du}}
\pgfusepath{stroke}
\node at (4\du,0\du){};

\pgfpathellipse{\pgfpoint{14\du}{0\du}}{\pgfpoint{1\du}{0\du}}{\pgfpoint{0\du}{1\du}}
\pgfusepath{stroke}
\node at (14\du,0\du){};

\pgfpathellipse{\pgfpoint{24\du}{0\du}}{\pgfpoint{1\du}{0\du}}{\pgfpoint{0\du}{1\du}}
\pgfusepath{stroke}
\node at (24\du,0\du){};

\pgfsetlinewidth{0.300000\du}
\pgfsetdash{}{0pt}
\pgfsetdash{}{0pt}
\pgfsetbuttcap

{\draw (-5\du,0\du)--(3\du,0\du);}
{\draw (15\du,0\du)--(23\du,0\du);}
{\draw (4.65\du,0.7\du)--(13.35\du,0.7\du);}
{\draw (4.65\du,-0.7\du)--(13.35\du,-0.7\du);}


{\pgfsetcornersarced{\pgfpoint{0.300000\du}{0.300000\du}}\definecolor{dialinecolor}{rgb}{0.000000, 0.000000, 0.000000}
\pgfsetstrokecolor{dialinecolor}
\draw (7\du,-1.2\du)--(10.8\du,0\du)--(7\du,1.2\du);}

\node[anchor=west] at (28\du,0\du){${\rm F}_4$};

\node[anchor=south] at (-6\du,1.1\du){$\scriptstyle 1$};

\node[anchor=south] at (4\du,1.1\du){$\scriptstyle 2$};

\node[anchor=south] at (14\du,1.1\du){$\scriptstyle 3$};

\node[anchor=south] at (24\du,1.1\du){$\scriptstyle 4$};

\end{tikzpicture} \vspace{-0.1cm}\hspace{0.6cm}
\ifx\du\undefined
  \newlength{\du}
\fi
\setlength{\du}{3.3\unitlength}
\begin{tikzpicture}
\pgftransformxscale{1.000000}
\pgftransformyscale{1.000000}

\definecolor{dialinecolor}{rgb}{0.000000, 0.000000, 0.000000} 
\pgfsetstrokecolor{dialinecolor}
\definecolor{dialinecolor}{rgb}{0.000000, 0.000000, 0.000000} 
\pgfsetfillcolor{dialinecolor}


\pgfsetlinewidth{0.300000\du}
\pgfsetdash{}{0pt}
\pgfsetdash{}{0pt}

\pgfpathellipse{\pgfpoint{-6\du}{0\du}}{\pgfpoint{1\du}{0\du}}{\pgfpoint{0\du}{1\du}}
\pgfusepath{stroke}
\node at (-6\du,0\du){};

\pgfpathellipse{\pgfpoint{4\du}{0\du}}{\pgfpoint{1\du}{0\du}}{\pgfpoint{0\du}{1\du}}
\pgfusepath{stroke}
\node at (4\du,0\du){};

\pgfsetlinewidth{0.300000\du}
\pgfsetdash{}{0pt}
\pgfsetdash{}{0pt}
\pgfsetbuttcap

{\draw (-5\du,0\du)--(3\du,0\du);}
{\draw (-5.35\du,0.7\du)--(3.35\du,0.7\du);}
{\draw (-5.35\du,-0.7\du)--(3.35\du,-0.7\du);}


{\pgfsetcornersarced{\pgfpoint{0.300000\du}{0.300000\du}}\definecolor{dialinecolor}{rgb}{0.000000, 0.000000, 0.000000}
\pgfsetstrokecolor{dialinecolor}
\draw (0.8\du,-1.2\du)--(-3\du,0\du)--(0.8\du,1.2\du);}

\node[anchor=west] at (8\du,0\du){${\rm G}_2$};

\node[anchor=south] at (-6\du,1.1\du){$\scriptstyle 1$};

\node[anchor=south] at (4\du,1.1\du){$\scriptstyle 2$};

\end{tikzpicture} 
\end{array}
\end{equation*}\par
\medskip
The connected components of the Dynkin diagram $\cD$ determine the simple Lie groups that are factors of the semisimple Lie group $G$, each of them corresponding to one of the Dynkin diagrams above.

In the above list we have introduced a numbering for the nodes of every Dynkin diagram (we have followed the standard reference \cite[p. 58]{Hum}). We note also that the classical Lie groups $\sl_{n+1}$, $\so_{2n+1}$, $\sp_{2n}$ and $\so_{2n}$ correspond to the diagrams ${\rm A}_n$, ${\rm B}_n$, ${\rm C}_n$ and ${\rm D}_n$, respectively.

\subsubsection*{\bf Rational homogeneous manifolds and marked Dynkin diagrams}\label{sssec:rathom}

Not only semisimple Lie groups, but also their projective quotients $G/P$  may be represented by means of Dynkin diagrams. The key point for this is that a subgroup $P$ of $G$ for which $G/P$ is projective,  called a {\it parabolic subgroup}, is determined by a set of simple roots of $G$ in the following way: given a subset $I\subset D$, let $\Phi^+(I)$ be the subset of $\Phi^+$ generated by the simple roots in $D\setminus I$. If moreover $I$ intersects every connected component of the Dynkin diagram $\cD$, then the subspace
\begin{equation}
\fp(I):=\fh\oplus\bigoplus_{\alpha\in\Phi^+} \fg_{-\alpha}\oplus\bigoplus_{\alpha\in\Phi^+(I)} \fg_\alpha
\label{eq:cartanparab}
\end{equation}
is a {\it parabolic subalgebra} of $\fg$, determining a parabolic subgroup  $P(I)\subset G$. Conversely, every parabolic subgroup is constructed in this way. In the most common notation, we represent $F(I):=G/P(I)$ by marking on the Dynkin diagram $\cD$ of $G$ the nodes corresponding to $I$.

\begin{example}
Given an $(n+1)$-dimensional complex vector space $V$,
the rational homogeneous manifolds for the group $\sl_{n+1}=\sl_{n+1}(V)$, are determined by the different markings of the Dynkin diagram $A_n$. For instance, numbering the nodes of $A_n$  as in (\ref{eq:dynkins}), the projective space $\P^n=\P(V^\vee)$, its dual $\P(V)$, and $\P(T_{\P^n})$ correspond to the marking of $I=\{1\}$, $\{n\}$ and $\{1,n\}$, respectively:

$$
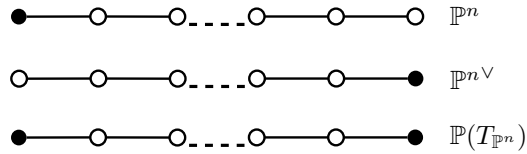
\begin{array}{l} \vspace{0.1cm}
\ifx\du\undefined
  \newlength{\du}
\fi
\setlength{\du}{3\unitlength}
\begin{tikzpicture}
\pgftransformxscale{1.000000}
\pgftransformyscale{1.000000}

\definecolor{dialinecolor}{rgb}{0.000000, 0.000000, 0.000000} 
\pgfsetstrokecolor{dialinecolor}
\definecolor{dialinecolor}{rgb}{0.000000, 0.000000, 0.000000} 
\pgfsetfillcolor{dialinecolor}


\pgfsetlinewidth{0.300000\du}
\pgfsetdash{}{0pt}
\pgfsetdash{}{0pt}

\pgfpathellipse{\pgfpoint{-10.994954\du}{10.989768\du}}{\pgfpoint{0.981754\du}{0\du}}{\pgfpoint{0\du}{0.994352\du}}
\pgfusepath{fill}
\node at (-10.994954\du,11.184768\du){};

\pgfpathellipse{\pgfpoint{-0.994954\du}{11.006435\du}}{\pgfpoint{0.981754\du}{0\du}}{\pgfpoint{0\du}{0.994352\du}}
\pgfusepath{stroke}
\node at (-0.994954\du,11.201435\du){};

\pgfpathellipse{\pgfpoint{9.000474\du}{10.989768\du}}{\pgfpoint{0.981754\du}{0\du}}{\pgfpoint{0\du}{0.994352\du}}
\pgfusepath{stroke}
\node at (9.000474\du,11.184768\du){};

\pgfpathellipse{\pgfpoint{19.005046\du}{11.006435\du}}{\pgfpoint{0.981754\du}{0\du}}{\pgfpoint{0\du}{0.994352\du}}
\pgfusepath{stroke}
\node at (19.005046\du,11.201435\du){};

\pgfpathellipse{\pgfpoint{28.998379\du}{11.006435\du}}{\pgfpoint{0.981754\du}{0\du}}{\pgfpoint{0\du}{0.994352\du}}
\pgfusepath{stroke}
\node at (28.998379\du,11.201435\du){};

\pgfpathellipse{\pgfpoint{38.998379\du}{11.006435\du}}{\pgfpoint{0.981754\du}{0\du}}{\pgfpoint{0\du}{0.994352\du}}
\pgfusepath{stroke}
\node at (38.998379\du,11.201435\du){};

\pgfsetlinewidth{0.300000\du}
\pgfsetdash{}{0pt}
\pgfsetdash{}{0pt}
\pgfsetbuttcap

{\draw (-10.013295\du,10.989768\du)--(-1.976613\du,11.006435\du);}

{\draw (-0.013295\du,11.006435\du)--(8.018720\du,10.989768\du);}

{\draw (19.986800\du,11.006435\du)--(28.016626\du,11.006435\du);}

{\draw (29.980133\du,11.006435\du)--(38.016626\du,11.006435\du);}

\pgfsetlinewidth{0.400000\du}
\pgfsetdash{{1.000000\du}{1.000000\du}}{0\du}
\pgfsetdash{{1.000000\du}{1.000000\du}}{0\du}
\pgfsetbuttcap
{\draw (10.479373\du,10.006435\du)--(17.462706\du,10.006435\du);}

\node[anchor=west] at (42.037706\du,11.038980\du){$\P^n$};
\end{tikzpicture} \hspace{0.72cm}\vspace{0.1cm}\\
\ifx\du\undefined
  \newlength{\du}
\fi
\setlength{\du}{3\unitlength}
\begin{tikzpicture}
\pgftransformxscale{1.000000}
\pgftransformyscale{1.000000}

\definecolor{dialinecolor}{rgb}{0.000000, 0.000000, 0.000000} 
\pgfsetstrokecolor{dialinecolor}
\definecolor{dialinecolor}{rgb}{0.000000, 0.000000, 0.000000} 
\pgfsetfillcolor{dialinecolor}


\pgfsetlinewidth{0.300000\du}
\pgfsetdash{}{0pt}
\pgfsetdash{}{0pt}

\pgfpathellipse{\pgfpoint{-10.994954\du}{10.989768\du}}{\pgfpoint{0.981754\du}{0\du}}{\pgfpoint{0\du}{0.994352\du}}
\pgfusepath{stroke}
\node at (-10.994954\du,11.184768\du){};

\pgfpathellipse{\pgfpoint{-0.994954\du}{11.006435\du}}{\pgfpoint{0.981754\du}{0\du}}{\pgfpoint{0\du}{0.994352\du}}
\pgfusepath{stroke}
\node at (-0.994954\du,11.201435\du){};

\pgfpathellipse{\pgfpoint{9.000474\du}{10.989768\du}}{\pgfpoint{0.981754\du}{0\du}}{\pgfpoint{0\du}{0.994352\du}}
\pgfusepath{stroke}
\node at (9.000474\du,11.184768\du){};

\pgfpathellipse{\pgfpoint{19.005046\du}{11.006435\du}}{\pgfpoint{0.981754\du}{0\du}}{\pgfpoint{0\du}{0.994352\du}}
\pgfusepath{stroke}
\node at (19.005046\du,11.201435\du){};

\pgfpathellipse{\pgfpoint{28.998379\du}{11.006435\du}}{\pgfpoint{0.981754\du}{0\du}}{\pgfpoint{0\du}{0.994352\du}}
\pgfusepath{stroke}
\node at (28.998379\du,11.201435\du){};

\pgfpathellipse{\pgfpoint{38.998379\du}{11.006435\du}}{\pgfpoint{0.981754\du}{0\du}}{\pgfpoint{0\du}{0.994352\du}}
\pgfusepath{fill}
\node at (38.998379\du,11.201435\du){};

\pgfsetlinewidth{0.300000\du}
\pgfsetdash{}{0pt}
\pgfsetdash{}{0pt}
\pgfsetbuttcap

{\draw (-10.013295\du,10.989768\du)--(-1.976613\du,11.006435\du);}

{\draw (-0.013295\du,11.006435\du)--(8.018720\du,10.989768\du);}

{\draw (19.986800\du,11.006435\du)--(28.016626\du,11.006435\du);}

{\draw (29.980133\du,11.006435\du)--(38.016626\du,11.006435\du);}

\pgfsetlinewidth{0.400000\du}
\pgfsetdash{{1.000000\du}{1.000000\du}}{0\du}
\pgfsetdash{{1.000000\du}{1.000000\du}}{0\du}
\pgfsetbuttcap
{\draw (10.479373\du,10.006435\du)--(17.462706\du,10.006435\du);}

\node[anchor=west] at (42.037706\du,11.038980\du){${\P^n}^\vee$};
\end{tikzpicture} \hspace{0.52cm}\vspace{0.1cm}\\
\ifx\du\undefined
  \newlength{\du}
\fi
\setlength{\du}{3\unitlength}
\begin{tikzpicture}
\pgftransformxscale{1.000000}
\pgftransformyscale{1.000000}

\definecolor{dialinecolor}{rgb}{0.000000, 0.000000, 0.000000} 
\pgfsetstrokecolor{dialinecolor}
\definecolor{dialinecolor}{rgb}{0.000000, 0.000000, 0.000000} 
\pgfsetfillcolor{dialinecolor}


\pgfsetlinewidth{0.300000\du}
\pgfsetdash{}{0pt}
\pgfsetdash{}{0pt}

\pgfpathellipse{\pgfpoint{-10.994954\du}{10.989768\du}}{\pgfpoint{0.981754\du}{0\du}}{\pgfpoint{0\du}{0.994352\du}}
\pgfusepath{fill}
\node at (-10.994954\du,11.184768\du){};

\pgfpathellipse{\pgfpoint{-0.994954\du}{11.006435\du}}{\pgfpoint{0.981754\du}{0\du}}{\pgfpoint{0\du}{0.994352\du}}
\pgfusepath{stroke}
\node at (-0.994954\du,11.201435\du){};

\pgfpathellipse{\pgfpoint{9.000474\du}{10.989768\du}}{\pgfpoint{0.981754\du}{0\du}}{\pgfpoint{0\du}{0.994352\du}}
\pgfusepath{stroke}
\node at (9.000474\du,11.184768\du){};

\pgfpathellipse{\pgfpoint{19.005046\du}{11.006435\du}}{\pgfpoint{0.981754\du}{0\du}}{\pgfpoint{0\du}{0.994352\du}}
\pgfusepath{stroke}
\node at (19.005046\du,11.201435\du){};

\pgfpathellipse{\pgfpoint{28.998379\du}{11.006435\du}}{\pgfpoint{0.981754\du}{0\du}}{\pgfpoint{0\du}{0.994352\du}}
\pgfusepath{stroke}
\node at (28.998379\du,11.201435\du){};

\pgfpathellipse{\pgfpoint{38.998379\du}{11.006435\du}}{\pgfpoint{0.981754\du}{0\du}}{\pgfpoint{0\du}{0.994352\du}}
\pgfusepath{fill}
\node at (38.998379\du,11.201435\du){};

\pgfsetlinewidth{0.300000\du}
\pgfsetdash{}{0pt}
\pgfsetdash{}{0pt}
\pgfsetbuttcap

{\draw (-10.013295\du,10.989768\du)--(-1.976613\du,11.006435\du);}

{\draw (-0.013295\du,11.006435\du)--(8.018720\du,10.989768\du);}

{\draw (19.986800\du,11.006435\du)--(28.016626\du,11.006435\du);}

{\draw (29.980133\du,11.006435\du)--(38.016626\du,11.006435\du);}

\pgfsetlinewidth{0.400000\du}
\pgfsetdash{{1.000000\du}{1.000000\du}}{0\du}
\pgfsetdash{{1.000000\du}{1.000000\du}}{0\du}
\pgfsetbuttcap
{\draw (10.479373\du,10.006435\du)--(17.462706\du,10.006435\du);}

\node[anchor=west] at (42.037706\du,11.038980\du){$\P(T_{\P^n})$};
\end{tikzpicture} 
\end{array}
$$

\noindent More generally, the homogeneous manifold determined by marking the ordered set of nodes $I=\{i_1,\dots,i_k\}$, which is called the {\it flag manifold} associated to $I$ in $\P^n$, is isomorphic to the parameter space of {\it flags} of linear spaces of the form $\P^{i_1}\subset\dots\subset \P^{i_k}\subset\P^n$. Similar descriptions may be done for the rest of the classical simple Lie groups.
\end{example}

\subsubsection*{\bf Contractions of rational homogeneous manifolds}\label{ssection:contraction}

We finally show how to write the contractions of a manifold $F(I)$ in terms of marked Dynkin diagrams.

From the above construction it immediately follows that given two subsets $J\subset I\subset D$, the inclusion $P(I)\subset P(J)$ provides a proper surjective morphism $p^{I,J}:F(I)\to F(J)$. Moreover, the fibers of this morphism are rational homogeneous manifolds, determined by the marked Dynkin diagram obtained from $\cD$ by removing the nodes in $J$ and marking the nodes in $I\setminus J$.

The following result states that these maps are the only contractions of $F(I)$:

\begin{proposition}\label{prop:RHsimp}
Every rational homogeneous manifold $F(I)=G/P(I)$ is a Fano manifold, whose contractions are all of the form $p^{I,J}$, $J\subset I\subset D$. In particular, the Picard number of $F(I)$ is $\sharp(I)$ and the Mori cone $\NE(F(I))\subset N_1(F(I))$ is simplicial.
\end{proposition}

\subsection{Rational curves on CP-manifolds}\label{ssec:ratcurves}

In this section we review some results on families of rational curves on Fano manifolds, that will be useful later on. For simplicity, we focus only on the case of CP-manifolds, and refer to \cite{Hw, KS, kollar, Mk} for more details and general results on this topic.

\begin{notation}\label{not:RConCP}
Let $X$ be a CP-manifold of dimension $m$. We denote by $\Hom(\P^1, X)$ the scheme parametrizing the morphisms from $\P^1$ to $X$, and by $\ev:\Hom(\P^1, X)\times\P^1\to X$ the natural evaluation map (see \cite[Chapter II]{kollar}).
\end{notation}

The following lemma shows that every rational curve on a CP-manifold is free, i.e. its deformations dominate $X$.

\begin{lemma}\label{lem:RConCP}
Let $X$ be a CP-manifold and let $f:\P^1\to X$ be a nonconstant morphism. Then $\Hom(\P^1, X)$ is smooth at $[f]$ and, being $H$ the irreducible component of $\Hom(\P^1, X)$ containing $[f]$, the restriction of the evaluation morphism $H\times\P^1\to X$ is dominant.
\end{lemma}

\begin{proof}
Being $T_X$ nef by hypothesis, for any nonconstant morphism $f:\P^1 \to X$ it holds that $f^*T_X$ is globally generated, and in particular $H^1(\P^1,f^*T_X)=0$. Then the smoothness of $\Hom(\P^1, X)$ follows by \cite[I.2.16]{kollar}, and the dominancy of $\ev$ by the standard description of its differential in terms of the evaluation of global sections of $f^*T_X$ (\cite[II.3.4]{kollar}).
\end{proof}

\begin{notation}\label{not:RConCP2}
Taking quotient by the automorphims of $\P^1$, one produces the scheme parametrizing rational curves on $X$ (see \cite[II.2]{kollar}), whose normalization is denoted by ${\rat}^n(X)$.
An irreducible component $\cM$ of $\rat^n(X)$ is called a {\it minimal rational component} if it contains a rational curve of minimal anticanonical degree. Since all rational curves parametrized by $\cM$ are numerically equivalent, we may set $d:=-K_X \cdot C$ for $[C] \in \cM$.
For a minimal rational component $\cM$, we have an associated universal family and an evaluation morphism:

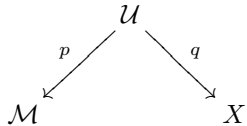
\[\xymatrix@=25pt{&
 \cU \ar[dl]_{p}  \ar[dr]^{q}  &   \\
 \cM &&X  \\
} \]
Finally, we denote by ${\cM}_x$ the normalization of the subscheme of $\cM$ parametrizing rational curves passing through $x \in X$.
\end{notation}

\begin{proposition}\label{prop:RCbasic} With the same notation as above, we have the following:
\begin{enumerate}
\item $\cM$ is a smooth projective variety of dimension $m+d-3$.
\item $q$ is a smooth morphism and $p$ is a smooth $\P^1$-fibration.
\item For a general point $x \in X$, ${\cM}_x$ is isomorphic to $q^{-1}(x)$. In particular, ${\cM}_x$ is a projective manifold of dimension $d-2$.
\item $d$ is at least $2$.
\item If $d = 2$, then $q$ is an isomorphism. In particular, $X$ admits a smooth $\P^1$-fibration structure $p: X \to \cM$.
\item If a rational curve $C$ is a general member of $\cM$, then $C$ is {\em standard}, i.e.  denoting by $f: \P^1 \to X$ its normalization, $f^{\ast}T_X \cong {\cO}_{\P^1}(2)\oplus {\cO}_{\P^1}(1)^{\oplus d-2} \oplus {\cO}_{\P^1}^{\oplus m-d+1}$.
\end{enumerate}
\end{proposition}

\begin{proof} Since $T_X$ is nef, we see that $q$ is a smooth morphism and $\cM$ is a smooth projective variety of dimension $m+d-3$ by \cite[II. Theorem~1.7, Theorem~2.15, Corollary~3.5.3]{kollar}. The projectivity of $\cM$ follows from \cite[II. Proposition~2.14.1]{kollar}. Moreover, $p$ is a smooth $\P^1$-fibration by \cite[II. Corollary~2.12]{kollar}. Hence $\rm (1)$ and $\rm (2)$ hold. The third statement follows from \cite[Theorem~3.3]{Ke2}.
Since the dimension of ${\cM}_x$ is non-negative, we see that $d$ is at least $2$.
To prove $\rm (5)$, assume that $d= 2$.
By $\rm (1)$ and $\rm (2)$, $q$ is an \'etale covering. Since a Fano manifold is simply-connected, $q$ is an isomorphism. The last statement follows from \cite[IV. Corollary~2.9]{kollar}.
\end{proof}

\begin{remark}\label{rem:P1fibbun}
The terms {\it smooth $\P^1$-fibration} and {\it $\P^1$-bundle} are used in the literature with different meanings. In this paper, we will use the first to refer to smooth maps whose fibers are $\P^1$'s, whereas the second is reserved to projectivizations of rank two vector bundles. For instance, we may only say, in general, that the map $p:\cU\to \cM$ is a smooth $\P^1$-fibration, but its restriction $p^{-1}(\cM_x)\to \cM_x$ is a $\P^1$-bundle, since there exists a divisor that has degree one on the fibers.
\end{remark}

\begin{definition}\label{def:VMRT} For a general point $x \in X$, we define the {\it tangent map}
$${\tau}_x : {\cM}_x \dashrightarrow \P(T_{X,x}^\vee)$$
by assigning to each member of $\cM_x$ smooth at $x$ its tangent direction at $x$. We denote by $\cC_x \subset \P(T_{X,x}^\vee)$ the closure of the image of ${\tau}_x$, which is called the {\it variety of minimal rational tangents} (VMRT) at $x$.
\end{definition}

We introduce the fundamental results of the theory of VMRT.

\begin{theorem}[{\cite[Theorem~3.4]{Ke2},\cite[Theorem~1]{HM2}}] Under the setting of Definition~\ref{def:VMRT},
\begin{enumerate}
\item $\tau_x$ is a finite morphism, and
\item $\tau_x$ is birational onto $\cC_x$.
\end{enumerate}
Hence $\tau_x: \cM_x \to \cC_x$ is the normalization.
\end{theorem}

\begin{remark}\label{rem:O1} Under the setting of Definition~\ref{def:VMRT}, the tangent map can be considered as a morphism $\tau_x: q^{-1}(x) \cong {\cM}_x \to \cC_x$ by Proposition~\ref{prop:RCbasic}~$\rm (3)$.
Let $\cU_x$ be the universal family associated to $\cM_x$ and $K_{\cU_x/\cM_x}$ the relative canonical divisor of $\cU_x \to \cM_x$. Via the composition of natural morphisms $T_{\cU_x/\cM_x}|_{q^{-1}(x)} \subset T_{\cU_x}|_{q^{-1}(x)} $ and $T_{\cU_x}|_{q^{-1}(x)} \to T_{X,x}\otimes \cO_{q^{-1}(x)}$, it follows from \cite[Theorem~3.3, Theorem~3.4]{Ke2} that $T_{\cU_x/\cM_x}|_{q^{-1}(x)} $ is a subbundle of $T_{X,x}\otimes \cO_{q^{-1}(x)}$. This yields a morphism $q^{-1}(x) \to \P(T_{X,x}^{\vee})$, which is nothing but the tangent map $\tau_x$. Hence the tangent map $\tau_x$ satisfies
\begin{equation}
\tau_x^{\ast}\cO_{\P(T_{X,x}^\vee)}(1)=\cO_{q^{-1}(x)}(K_{\cU_x/\cM_x}).
\end{equation}
\end{remark}

\begin{proposition}\cite[Proposition~2.7]{Ar}\label{prop:imm} Under the setting of Definition~\ref{def:VMRT},
$\tau_x$ is immersive at $[C] \in \cM_x$ if and only if $C$ is standard.
\end{proposition}

\begin{remark}\label{rem:VMRTlines} Assume here that there exists a very ample line bundle $L$ on $X$ that has degree one on the curves parametrized by $\cM$ (note that this holds for rational homogeneous manifolds). The bundle $L$ provides an embedding $X \subset \P^N$, under which the curves parametrized by $\cM$ are lines. Then, at every $x\in X$, the map $\tau_x$ is injective, because any line through $x$ is uniquely determined by its tangent direction. Moreover, if $x$ is a general point, $\tau_x: \cM_x \to \P(T_{X,x}^{\vee})$ is an immersion (Proposition~\ref{prop:imm}), hence an embedding.
\end{remark}

\begin{example}\label{eg:G2} There exist two rational homogeneous manifolds of Picard number $1$ and of type ${\rm G}_2$. One is a $5$-dimensional quadric hypersurface $Q^5$, whose VMRT is well known to be $Q^3$, and  the other is a $5$-dimensional homogeneous manifold, that we denote by $K({\rm G}_2)$. It is known that $K({\rm G}_2)$ is covered by lines and $Q^5$ is isomorphic to its minimal rational component $\cM$. The universal family $\cU$ of $\cM$ is also a homogeneous manifold of type ${\rm G}_2$, dominating both $Q^5$ and $K({\rm G}_2)$:
$$
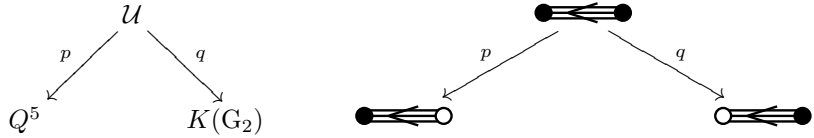
\xymatrix@=25pt{
 & \cU \ar[dl]_{p}  \ar[dr]^{q}  &   \\
 Q^5 & & \hspace{-0.5cm}K({\rm G}_2) \\
} \hspace{1cm} \xymatrix@=25pt{
 & \ifx\du\undefined
  \newlength{\du}
\fi
\setlength{\du}{3\unitlength}
\begin{tikzpicture}
\pgftransformxscale{1.000000}
\pgftransformyscale{1.000000}

\definecolor{dialinecolor}{rgb}{0.000000, 0.000000, 0.000000} 
\pgfsetstrokecolor{dialinecolor}
\definecolor{dialinecolor}{rgb}{0.000000, 0.000000, 0.000000} 
\pgfsetfillcolor{dialinecolor}


\pgfsetlinewidth{0.300000\du}
\pgfsetdash{}{0pt}
\pgfsetdash{}{0pt}

\pgfpathellipse{\pgfpoint{-6\du}{0\du}}{\pgfpoint{1\du}{0\du}}{\pgfpoint{0\du}{1\du}}
\pgfusepath{stroke}
\node at (-6\du,0\du){};
\pgfpathellipse{\pgfpoint{-6\du}{0\du}}{\pgfpoint{1\du}{0\du}}{\pgfpoint{0\du}{1\du}}
\pgfusepath{fill}
\node at (-6\du,0\du){};

\pgfpathellipse{\pgfpoint{4\du}{0\du}}{\pgfpoint{1\du}{0\du}}{\pgfpoint{0\du}{1\du}}
\pgfusepath{stroke}
\node at (4\du,0\du){};
\pgfpathellipse{\pgfpoint{4\du}{0\du}}{\pgfpoint{1\du}{0\du}}{\pgfpoint{0\du}{1\du}}
\pgfusepath{fill}
\node at (4\du,0\du){};

\pgfsetlinewidth{0.300000\du}
\pgfsetdash{}{0pt}
\pgfsetdash{}{0pt}
\pgfsetbuttcap

{\draw (-5\du,0\du)--(3\du,0\du);}
{\draw (-5.35\du,0.7\du)--(3.35\du,0.7\du);}
{\draw (-5.35\du,-0.7\du)--(3.35\du,-0.7\du);}


{\pgfsetcornersarced{\pgfpoint{0.300000\du}{0.300000\du}}\definecolor{dialinecolor}{rgb}{0.000000, 0.000000, 0.000000}
\pgfsetstrokecolor{dialinecolor}
\draw (0.8\du,-1.2\du)--(-3\du,0\du)--(0.8\du,1.2\du);
}
\end{tikzpicture}  \ar[dl]_{p}  \ar[dr]^{q}  &   \\
 \ifx\du\undefined
  \newlength{\du}
\fi
\setlength{\du}{3\unitlength}
\begin{tikzpicture}
\pgftransformxscale{1.000000}
\pgftransformyscale{1.000000}

\definecolor{dialinecolor}{rgb}{0.000000, 0.000000, 0.000000} 
\pgfsetstrokecolor{dialinecolor}
\definecolor{dialinecolor}{rgb}{0.000000, 0.000000, 0.000000} 
\pgfsetfillcolor{dialinecolor}


\pgfsetlinewidth{0.300000\du}
\pgfsetdash{}{0pt}
\pgfsetdash{}{0pt}

\pgfpathellipse{\pgfpoint{-6\du}{0\du}}{\pgfpoint{1\du}{0\du}}{\pgfpoint{0\du}{1\du}}
\pgfusepath{stroke}
\node at (-6\du,0\du){};
\pgfpathellipse{\pgfpoint{-6\du}{0\du}}{\pgfpoint{1\du}{0\du}}{\pgfpoint{0\du}{1\du}}
\pgfusepath{fill}
\node at (-6\du,0\du){};

\pgfpathellipse{\pgfpoint{4\du}{0\du}}{\pgfpoint{1\du}{0\du}}{\pgfpoint{0\du}{1\du}}
\pgfusepath{stroke}
\node at (4\du,0\du){};

\pgfsetlinewidth{0.300000\du}
\pgfsetdash{}{0pt}
\pgfsetdash{}{0pt}
\pgfsetbuttcap

{\draw (-5\du,0\du)--(3\du,0\du);}
{\draw (-5.35\du,0.7\du)--(3.35\du,0.7\du);}
{\draw (-5.35\du,-0.7\du)--(3.35\du,-0.7\du);}


{\pgfsetcornersarced{\pgfpoint{0.300000\du}{0.300000\du}}\definecolor{dialinecolor}{rgb}{0.000000, 0.000000, 0.000000}
\pgfsetstrokecolor{dialinecolor}
\draw (0.8\du,-1.2\du)--(-3\du,0\du)--(0.8\du,1.2\du);
}
\end{tikzpicture}  & & \ifx\du\undefined
  \newlength{\du}
\fi
\setlength{\du}{3\unitlength}
\begin{tikzpicture}
\pgftransformxscale{1.000000}
\pgftransformyscale{1.000000}

\definecolor{dialinecolor}{rgb}{0.000000, 0.000000, 0.000000} 
\pgfsetstrokecolor{dialinecolor}
\definecolor{dialinecolor}{rgb}{0.000000, 0.000000, 0.000000} 
\pgfsetfillcolor{dialinecolor}


\pgfsetlinewidth{0.300000\du}
\pgfsetdash{}{0pt}
\pgfsetdash{}{0pt}

\pgfpathellipse{\pgfpoint{-6\du}{0\du}}{\pgfpoint{1\du}{0\du}}{\pgfpoint{0\du}{1\du}}
\pgfusepath{stroke}
\node at (-6\du,0\du){};

\pgfpathellipse{\pgfpoint{4\du}{0\du}}{\pgfpoint{1\du}{0\du}}{\pgfpoint{0\du}{1\du}}
\pgfusepath{stroke}
\node at (4\du,0\du){};
\pgfpathellipse{\pgfpoint{4\du}{0\du}}{\pgfpoint{1\du}{0\du}}{\pgfpoint{0\du}{1\du}}
\pgfusepath{fill}
\node at (4\du,0\du){};

\pgfsetlinewidth{0.300000\du}
\pgfsetdash{}{0pt}
\pgfsetdash{}{0pt}
\pgfsetbuttcap

{\draw (-5\du,0\du)--(3\du,0\du);}
{\draw (-5.35\du,0.7\du)--(3.35\du,0.7\du);}
{\draw (-5.35\du,-0.7\du)--(3.35\du,-0.7\du);}


{\pgfsetcornersarced{\pgfpoint{0.300000\du}{0.300000\du}}\definecolor{dialinecolor}{rgb}{0.000000, 0.000000, 0.000000}
\pgfsetstrokecolor{dialinecolor}
\draw (0.8\du,-1.2\du)--(-3\du,0\du)--(0.8\du,1.2\du);
}
\end{tikzpicture}  \\
}
$$
Classically $\cU$ is described as the projectivization $\P(\cE)$ of a {\it Cayley bundle} $\cE$ on $Q^5$ (see \cite[1.3]{O}).
On the other hand, via the above diagram, $K({\rm G}_2)$ can be seen as a family of special lines on $Q^5$. In particular, for any $x \in K({\rm G}_2)$, $q^{-1}(x)$ is isomorphic to $\P^1$. By Remark~\ref{rem:O1}, we have $\tau_x^{\ast}\cO_{\P(T_{K({\rm G}_2),x}^\vee)}(1)=\cO_{q^{-1}(x)}(K_{\cU_x/\cM_x}).$ Since $$K_{\cU_x/\cM_x} \cdot {q^{-1}(x)}=K_{\cU} \cdot {q^{-1}(x)}-K_{\cM} \cdot p_{\ast}({q^{-1}(x)})=-2+5=3,$$ the VMRT $\cC_x$ of $K({\rm G}_2)$ has degree $3$. Since $\cC_x$ is smooth, it is the twisted cubic curve in its linear span. We refer to \cite{O} for more details on $K({\rm G}_2)$.
\end{example}

\subsection{VMRT's of rational homogeneous manifolds of Picard number one}\label{ssec:VMRTPic1}

Representation theory provides a description of the VMRT's of rational homogeneous manifolds (see for instance \cite{ Hw, LM, Mk2}). In this section we will confine to the case of homogeneous manifolds of Picard number one, for which the VMRT and its embedding are described in the following table:

\captionsetup[longtable]{skip=1em}
\setlength{\LTleft}{1.2 cm}
\setlength{\LTpost}{3pt}
\begin{longtable}{c|c|c|c|c}
\hline
 $\cD$& node $r$ & $X$ & $\hbox{VMRT}$ & $\hbox{embed.}$  \\ \hline \hline
 ${\rm A}_n$& $\leq n$ &$G(r-1,n)$&$\P^{r-1}\times \P^{n-r}$& $\cO(1,1)$\\ \hline
 ${\rm B}_n$& $\leq n-2$&$OG(r-1,2n)$&$\P^{r-1}\times Q^{2(n-r)-1}$& $\cO(1,1)$\\
  &$n-1$ &$OG(n-2,{2n})$&$\P^{n-2}\times \P^1$& $\cO(1,2)$\\
  &$n$&$S_n$&$G(n-2,n)$& $\cO(1)$\\ \hline
 ${\rm C}_n$& $1$ &$\P^{2n-1}$&$\P^{2n-2}$& $\cO(1)$\\
 & $\leq n-1$ &$LG(r-1,{2n-1})$&(see Prop. \ref{prop:VMRTshort2})&\\
  & $n$ &$LG(n-1,{2n-1})$&$\P^{n-1}$& $\cO(2)$\\ \hline
  ${\rm D}_n$& $\leq n-3$&$OG(r-1,{2n-1})$&$\P^{r-1}\times Q^{2(n-r-1)}$& $\cO(1,1)$\\
  & $n-2$&$OG(n-3,{2n-1})$&$\P^{1}\times \P^{1}\times \P^{n-3}$& $\cO(1,1,1)$\\
  & $n-1, n$&$S_{n-1}$&$G(n-3,{n-1})$& $\cO(1)$\\ \hline
 ${\rm E}_k$& $1$&${\rm E}_k(1)$&$S_{k-2}$& $\cO(1)$\\
  & $2$&${\rm E}_k(2)$&$G(2,{k-1})$& $\cO(1)$\\
  & $3$&${\rm E}_k(3)$&$\P^1\times G(1,{k-2})$& $\cO(1,1)$\\
  & $4$&${\rm E}_k(4)$&$\P^1\times \P^2 \times \P^{k-4}$& $\cO(1,1,1)$\\
  & $5$&${\rm E}_k(5)$&$G(2,{4})\times \P^{k-5}$& $\cO(1,1)$\\
  & $6$&${\rm E}_k(6)$&$S_4\times \P^{k-6}$& $\cO(1,1)$\\
  & $7$&${\rm E}_k(7)$&${\rm E}_6(6)\times \P^{k-7}$& $\cO(1,1)$\\
  & $8$&${\rm E}_8(8)$&${\rm E}_7(7)$& $\cO(1)$\\ \hline
 ${\rm F}_4$& $1$&${\rm F}_4(1)$&$LG(2,{5})$& $\cO(1)$\\
 & $2$&${\rm F}_4(2)$&$\P^1\times \P^2$& $\cO(1,2)$\\
 & $3$&${\rm F}_4(3)$&(see Prop. \ref{prop:VMRTshort2})&\\
 & $4$&${\rm F}_4(4)$&(see Prop. \ref{prop:VMRTshort2})&\\
 \hline
 ${\rm G}_2$& $1$&$Q^5$&$Q^3$& $\cO(1)$\\
 & $2$&$K({\rm G}_2)$&$\P^1$& $\cO(3)$\\
   \hline
   \caption[VMRTs of RH]{VMRT's of homogeneous manifolds of Picard number one} \label{tab:VMRTs}
\end{longtable}

The notation we have used in this table is the following:
\begin{itemize}
\item $Q^k$: smooth $k$-dimensional quadric.
\item $G(k,n)$: {\it Grassmannian} of $k$-linear projective subspaces of $\P^n$.
\item $OG(k,n)$: {\it Orthogonal Grassmannian} parametrizing the $k$-linear projective subspaces of $\P^n$ that are isotropic with respect to a nondegenerate symmetric form. It is irreducible except for $n=2k+1$.
\item $S_{n}$: {\it Spinor variety} $OG(n-1,2n)$; alternatively it may be described as each one of the two irreducible components of $OG(n,2n+1)$.
\item $LG(k,n)$: {\it Lagrangian Grassmannian} parametrizing the $k$-linear projective subspaces of $\P^n$ that are isotropic with respect to a nondegenerate skew-symmetric form ($n$ odd).
\item ${\rm E}_n(k)$, (resp. ${\rm F}_4(k)$): rational homogeneous manifold of type ${\rm E}_n$, (resp. ${\rm F}_4$), associated to the maximal parabolic subgroup determined by the set of nodes $I=\{k\}$.
\item $\cO(1)$: (very) ample generator of the Picard group of a rational homogeneous manifold (in each case).
\item $\cO(a_1,a_2,\dots)$: on a product of rational homogeneous manifolds of Picard number one, $Y_1\times Y_2\times\dots$, this represents the tensor product of the pullbacks of the $\cO(a_i)$'s (by the $i$-th natural projection).
\item In every case, the embedding of the VMRT is given by the complete linear system of the indicated line bundle.
\end{itemize}

Let us discuss briefly the contents of the table; we refer the interested reader to the original paper of Landsberg and Manivel, \cite{LM} for details.

We start by recalling that the root system of a simple Lie algebra may have elements of different length, as elements of the euclidean space $(E,\kappa)$. In fact, one can already see this at the level of simple roots:
\begin{itemize}
\item either $\cD$ has a multiple edge and then the length of a simple root may take two values, depending on whether there is an arrow in the diagram pointing in the direction of the corresponding node (and we say that the root is {\it short}) or not (and we say that the root is {\it long}),
\item or $\cD$ has only simple edges and then all the roots have the same length (by definition, they are long).
\end{itemize}

Given any connected Dynkin diagram $\cD$ a node $r$, we denote by $N(r)\subset D$ the set of nodes sharing an edge with $r$ in $\cD$. Then:
\begin{proposition}\label{prop:VMRTlong}
Let $X=G/P(r)$ be the rational homogeneous manifold of Picard number one determined by the connected Dynkin diagram $\cD$ marked at the node $r$. Assume moreover that $r$ is a node associated to a long root of $\fg$. Then the VMRT of $X$ at every point is a rational homogeneous manifold associated to the Dynkin diagram obtained from $\cD$ by removing the node $r$, and marking the nodes $N(r)$. In particular, it is a product of homogeneous manifolds of Picard number one.
\end{proposition}

If the node $r$ is a short root, the VMRT may still be computed, but it is homogeneous only in certain cases:

\begin{proposition}\label{prop:VMRTshort}
With the same notation as in the Proposition \ref{prop:VMRTlong}, assume that the pair $(\cD,r)$ is one of the following:
$$
({\rm B}_n,n),\quad ({\rm C}_n,1),\quad ({\rm G}_2,1).
$$
Then the VMRT of $X$ at every point is isomorphic, respectively, to
$$
G(n-2,n),\quad \P^{2n-2},\quad Q^3.
$$
\end{proposition}

\begin{proof}
The three homogeneous manifolds may be obtained also from the pairs
$$
({\rm D}_{n+1},n+1),\quad ({\rm A}_{2n-1},1),\quad ({\rm B}_3,1).
$$
Then the result follows from \ref{prop:VMRTlong}.
\end{proof}

Finally, the remaining cases are not homogeneous, but they have been described in the following way (see \cite{LM}):

\begin{proposition}\label{prop:VMRTshort2}
With the same notation as in the Proposition \ref{prop:VMRTlong}, the VMRT's of the homogeneous manifolds determined by the pairs $(\cD,r)=({\rm C}_n,r),$ $r=2,\dots,n-1$, $({\rm F}_4,3)$, $({\rm F}_4,4)$ may be described as follows:
\begin{itemize}
\item[$({\rm C}_n,r)$] Blow-up of $\P^{2n-r-1}$ along a $\P^{r-1}$, or $\P(\cO_{\P^{r-1}}(2)\oplus \cO_{\P^{r-1}}(1)^{2n-2r})$ embedded by the complete linear system of the tautological bundle $\cO(1)$. This is a codimension $(r-1)$ linear section of the VMRT of $({\rm A}_{2n-1},r)$, which is $\P^{r-1}\times\P^{2n-r-1}$. 
\item[$({\rm F}_4,3)$] Non-trivial smooth $Q^4$-fibration over $\P^1$.
\item[$({\rm F}_4,4)$] Smooth hyperplane section of $S_4$, (which is the VMRT of $({\rm E}_6,1)$).
\end{itemize}
\end{proposition}

The knowledge of the VMRT's of the rational homogeneous manifolds of Picard number one is particularly important due to the following result proved by Hong and Hwang. Within the class of Fano manifolds of Picard number one, certain rational homogeneous manifolds are determined by $\cC_x$ and its embedding in $\P(T_{X,x}^\vee)$.

\begin{theorem}[Special case of \cite{HH}]\label{them:HH} Let $X$ be a Fano manifold of Picard number one, $S=G/P$ a rational homogeneous manifold corresponding to a long simple root, and $\cC_o \subset \P(T_{S,o}^\vee)$ the VMRT at a reference point $o \in S$. Assume $\cC_o \subset \P(T_{S,o}^\vee)$ and $\cC_x \subset \P(T_{X,x}^\vee)$ are isomorphic as projective subvarieties. Then $X$ is isomorphic to $S$.
\end{theorem}

\section{Fano varieties with nef tangent bundle}\label{sec:CPvar}


In this section we prove that the basic properties of the contractions of CP-manifolds are analogous to those of rational homogeneous manifolds (see Proposition \ref{prop:RHsimp}). More concretely,  we will see that the Mori cone of a CP-manifold is simplicial (Corollary \ref{prop:simplicial}), and that its contractions are smooth fibrations, whose fibers and targets are CP-manifolds (Theorem \ref{thm:smooth}). This statement, originally due to Demailly, Peternell and Schneider \cite[Theorem~5.2]{DPS}, motivates an inductive approach to Conjecture \ref{conj:CPconj}.

We start by noting that, by Lemma \ref{lem:RConCP}, the nefness of $T_X$ implies that the deformations of every rational curve dominate $X$. Since moreover $X$ is Fano, the nontrivial fibers of its contractions contain rational curves, and one immediately gets the following:

\begin{proposition}\label{prop:fybertype} Every Mori contraction $\pi:X \to Y$ of a CP-manifold $X$ is of fiber type, that is $\dim Y<\dim X$.
\end{proposition}

As a consequence, one may show that the Mori cone of a CP-manifold $X$ is the convex hull of a basis of $N_1(X)$:

\begin{corollary}\label{prop:simplicial} The Mori cone $\NE(X)$ of a CP-manifold $X$ is simplicial.
\end{corollary}

\begin{proof} Assume by contradiction the existence of extremal rays $R_1,\dots, R_k$ such that  $k >\rho(X)=n$ and choose a rational curve $\Gamma_i$ of minimal anticanonical degree among those spanning the corresponding ray $R_i$. Without loss of generality we can assume that $[\Gamma_k]$ can be written as $[\Gamma_k]=\sum_{i=1}^{n}a_i [\Gamma_i]$ ($a_i \in \Q$ for $i=1, \dots, n$) and that $a_1<0$ by the extremality of $R_k$.

For $i \in \{2,\dots,n\}$ take the unsplit (by the minimality of the degree of $\Gamma_i$) family $\cM_i$ of rational curves containing $\Gamma_i$. Using  the {\it rational connectedness relation with respect to $(\cM_2,\dots,\cM_k)$}, cf. \cite[IV. 4.16]{kollar}, one can prove,  see \cite[Lemma~2.4]{CO}, that the classes $[\Gamma_2],\dots, [\Gamma_n]$ are lying in an $(n-1)$-dimensional face of $\NE(X)$, being every contraction of $X$ of fiber type. A supporting divisor $H$ of this face provides a contradiction: $H\cdot \Gamma_i=0$ for $i=2, \dots, n$, $H\cdot \Gamma_1>0$ so that $H\cdot \Gamma_k <0$, contradicting that $H$ is nef.
\end{proof}

The main result of this section is the following:

\begin{theorem}\label{thm:smooth} Let $\pi:X \to Y$ be a Mori contraction of a CP-manifold. Then  $\pi:X \to Y$ is smooth, and $Y$ and the fibers of $\pi$ are CP-manifolds.
\end{theorem}

\begin{remark} The second part of the statement is an easy consequence of the smoothness of $\pi$ via the exact sequences defining the relative tangent bundle and the normal bundles to the fibers.
Furthermore we may assume that $\pi$ is  an elementary contraction, i.e. that its relative Picard number is one. Otherwise, by the Cone Theorem, we can factor $\pi$ as $\pi_2 \circ \pi_1$ where $\pi_1$ is elementary and $\pi_2:X_1 \to Y$ is a contraction of a CP-manifold with smaller relative Picard number. Then the general result follows by induction (for details, see \cite[Lemma~4.7]{1-ample}).
\end{remark}

The proof we present here is based on \cite[4.2]{1-ample}, and will be divided in several steps.

\begin{lemma}[{\cite[Lemmas 4.10 and 4.11]{1-ample}}]\label{lem:equidim}
Let $\pi:X \to Y$ be an elementary Mori contraction of a CP-manifold. Then $\pi$ is equidimensional and all its fibers are irreducible.
\end{lemma}

\begin{proof}
Let $F \subset X$ be a general fiber of $\pi$ and consider a component $M$ of the Chow variety $\Chow(X)$ containing $F$. Normalizing, if necessary, we consider the universal family $U$ of cycles over $M$ and the evaluation map $e:U \to X$, fitting in a commutative diagram:
$$\xymatrix{
U\ar[r]^e\ar[d]_u & X\ar[d]^\pi\\
M\ar[r]_k & Y}$$
Since any cycle algebraically equivalent to $F$ is contracted by $\pi$, then $e$ (and hence $k$) is birational.
Since $u$ is equidimensional, it is enough to show that $k$ is an isomorphism, which will follow, being $\pi$ elementary, by showing that $e$ is an isomorphism. Let us prove this last assertion by contradiction. Assume the existence of $x \in X$ such that $e^{-1}(x)$ is positive dimensional and consider the variety $Z(x)$ swept out on $X$ by the cycles of the family $M$ by $x$, that is, $$Z(x)=e(u^{-1}(u(e^{-1}(x)))),$$ which satisfies $\dim Z(x)>\dim F$. The general fiber of $u$ is rationally chain connected, hence, since rationally chain connected fibers on a equidimensional morphism form a countable union of closed sets (see \cite[IV, 3.5.2]{kollar}), any fiber of $u$ is rationally chain connected, and so is $Z(x)$. Let us now observe that for any chain of rational curves $\Gamma$ in a fiber of $\pi$, there exists a smoothing with a smooth point $y$ of the chain $\Gamma$ fixed, see \cite[II, 7.6.1]{kollar}. This implies that every point $y'\in\Gamma$ lies in the closure of a component of the set of points of $Z(x)$ that can be joined with $y$ by an irreducible rational curve in $Z(x)$. Since the base field $\C$ is uncountable, it follows that any pair of general points in $Z(x)$ can be joined by a rational curve in the fiber $\pi^{-1}(x)$. Let $C$ be a general curve in a family of rational curves joining two general points of $Z(x)$ and let $f:\P^1 \to C$ be its normalization. Since deformations of $C$ by a point sweep out $Z(x)$  we get:
$$r^+(f^*T_X)=\dim(X)-h^0(f^*\Omega_X)\geq \dim Z(x) >\dim F,$$
where $r^+ (f^*T_X)$ is the number of positive summands of $f^*T_X$.
But, by semicontinuity of cohomology, for small deformations $f_t$ of $f$ we get the same inequality.
Since deformations of $f$ dominate $X$, some small deformation $f_{t_0}$ is contained in a general fiber; on such a fiber  the normal bundle is trivial, hence $r^+(f_{t_0}^*T_X) \le \dim F$, a contradiction.

The argument showing that rational chains joining points of $Z(x)$ can be smoothened fixing a point can be used to prove that any fiber of $\pi$ is rationally connected and then irreducible.
\end{proof}

Due to the local nature of the statement of Theorem \ref{thm:smooth}, we can reduce the problem to the case in which $Y$ affine. In fact, we can assume that the coordinate ring $A(Y)=H^0(X, \cO_X)$.

\begin{lemma}[{\cite[Lemma~4.12]{1-ample}}]\label{lem:normaltrivial}
With the same notation as in \ref{lem:equidim}, every fiber of $\pi$ with its reduced structure is smooth, Fano and its normal bundle is trivial.
\end{lemma}

\begin{proof}
Let $y \in Y$ be a closed point and take $g_1, \dots,g_s \in A(Y)$ to be generators of the maximal ideal of $y$. The ideal sheaf of $F':=\pi^{-1}(y)$ is generated by the $g_i$'s. Denote by $\cI$ the ideal sheaf of the reduced structure $F$ of $F'$ in $X$.  By the product of differentials, the $g_i$'s provide global sections of $\Sym^{r_i}(\Omega_X \otimes \cO_F)$ where $r_i$ is the maximal integer for which $g_i \in \cI^{r_i}$. Hence we get divisors $D(g_i) \in |\cO_{\P(\Omega_X \otimes \cO_F)}(r_i)|$. The exact sequence
$$\cI/\cI^2 \to \Omega_X \otimes \cO_F \to \Omega_F \to 0$$ shows that the base locus $B=\cap_{i=1}^s D(g_i)$ of these divisors  contains $\P(\Omega_F)$, being equal over a general point of $F$. Now, the rational curves in the fibers enter into the picture. Since any two points of $F$ can be joined by a chain of rational curves, it is enough to show that a rational curve through a smooth point of $F$ is contained in the smooth locus of $F$. Take $f:\P^1 \to X$ such that $f(\P^1) \subset F$ contains a smooth point of $F$. The divisors $D(g_i)$ provide divisors on $\cO_{(\P(f^*\Omega_X))}(r_i)$.  It can be shown (see \cite[Lemma~4.6]{1-ample}) that the intersection $B'$ of these divisors with any fiber of $\P(f^*\Omega_X)$ does not depend on the fiber. The equality at the smooth points of $F$ leads consequently to the equality $B'=\P(f^*\Omega_F)$ and to the fact that $f^*\Omega_F$ is a sheaf of constant rank along $f(\P^1)$. Then $F$ is smooth along $f(\P^1)$ and consequently $F$ is smooth as remarked above. Moreover, this shows that the bundle $f^*\cN_{F/X}$ is trivial on any rational curve on $F$, and this implies that $\cN_{F/X}$ is trivial, see \cite[Proposition~2.4]{kyoto}. Hence the fiber $F$ is Fano by adjunction.
\end{proof}

\begin{proof}[Proof of Theorem \ref{thm:smooth}]

With the notation as above consider the blow-up $\beta: X' \to X$ of $X$ along the reduced structure $F$ of $\pi^{-1}(y)$. Since $\cN_{F/X}$ is trivial, then there exist ($\dim Y$) global sections of $-E|_E=\cO_E(1)$ spanning it. The restriction to the exceptional divisor $E$ of $-K_{X'}-2E$ is $\cO_E(\dim Y+1)-\beta^*K_X|_E$, hence ample.
Then $h^1(X',-2E)=0$ by the Kawamata-Viehweg vanishing theorem. This implies that any section of $-E|_E$ extends to $X'$ and descends to a function on $Y$ vanishing at $y$.  Then we have regular functions $g_1, \dots,g_{\dim Y} \in A(Y)$ vanishing at $y$ and whose pull-backs generate an ideal sheaf of $F$ at every point. Finally it follows that $g_1, \dots, g_{\dim Y}$ are regular generators of the maximal ideal of $y \in Y$, and then $Y$ is smooth at $y$ and $\pi:X \to Y$ is smooth.
\end{proof}

The smoothness of the contractions has interesting consequences on the Mori cones of $X$ and $Y$.

\begin{proposition}\label{prop:relativepicards} For every contraction $\pi:X \to Y$ of a CP-manifold and for every $y \in Y$ the following properties hold:

{\rm (1)} $\rho(\pi^{-1}(y))=\rho(X)-\rho(Y)$, and

{\rm (2)} $j_*(\NE(\pi^{-1}(y)))=\NE(X) \cap j_*(\Nu(\pi^{-1}(y)))$, where $j_*$ is the linear map induced by the inclusion $j:\pi^{-1}(y) \to X$.
\end{proposition}

\begin{proof} Let us follow the ideas of the proof of \cite[Lemma~3.3 and Example 3.8]{casagrande}. By the relative version of the Cone Theorem, the relative cone of curves $\NE(X/Y) \subset \ker \pi_*$ is closed and poyhedral, and one can choose an extremal ray $R$ of $\NE(X/Y)$. Since $\pi$ is smooth, by \cite[Proposition~1.3]{defofnef}, the locus of curves in $R$ dominates $Y$, hence $R \subseteq j_*\NE(\pi^{-1}(y))$. This proves that  the subspace $N_1(\pi^{-1}(y),X)$ of $N_1(X)$ spanned by curves in $\pi^{-1}(y)$ is equal to $\ker \pi_*$. We  must prove now  that $\dim N_1(\pi^{-1}(y),X)=\rho(\pi^{-1}(y))$. Dually, this dimension is equal to the dimension of the image of the restriction of divisors: $\Pic (X) \otimes \Q \to \Pic(\pi^{-1}(y))\otimes \Q$. This can be computed (see for instance \cite[Chapter 3]{voisin}) as the dimension of the linear subspace of $\Pic(\pi^{-1}(y))\otimes \Q=H^2(\pi^{-1}(y),\Q)$ invariant for the monodromy action of the fundamental group $\pi_1(Y,y)$, which is trivial, being $Y$ Fano and hence simply connected. This proves (1).

The statement (2) is just a consequence of the fact that any curve $C$ in an extremal ray of $\NE(X)$ contracted by $\pi$ deforms to a curve meeting $\pi^{-1}(y)$ (any contraction is of fiber type), and is therefore contained in this fiber.
\end{proof}

By the simpliciality of $X$ proved above, the number of extremal rays of $X$ equals its Picard number. Let us finish this section by fixing the notation that will be used in the sequel:

\begin{notation}\label{not:cpmanifold}
 Given a CP-manifold $X$ of Picard number $n$, we will denote by $R_i$, $i =1, \dots, n$, the extremal rays of $\NE(X)$. For every $i$ the corresponding elementary contraction will be denoted by $\pi_i:X\to X_i$, and its relative canonical divisor by $K_i$. We will denote by $\Gamma_i$ a rational curve of minimal degree such that $[\Gamma_i] \in R_i$, general in the corresponding unsplit family of rational curves $\cM_i$,  by $p_i:\cU_i\to\cM_i$ the universal family of curves parametrized by $\cM_i$, and by $q_i:\cU_i\to X$ the evaluation morphism. In the particular case $n=1$, we will skip subindices, and write $\cM$, $\cU$, etc, instead of $\cM_1$, $\cU_1$,\dots
\end{notation} 

\section{Results in low dimension}\label{sec:lowdim}


The purpose of this section is to survey the results of classification of low-dimensional CP-manifolds. In dimensions one and two the picture is simple: the only smooth Fano curve is $\P^1$, which is in particular homogeneous; smooth Fano surfaces are either $\P^2$, or a blow-up of $\P^2$, or $\P^1 \times \P^1$ and, among these manifolds, only the homogeneous manifolds $\P^2$ and $\P^1 \times \P^1$ have nef tangent bundle.

Let us recall the following general definition: for a Fano manifold $X$ of dimension $m \geq 2$, the {\it pseudoindex} $i_X$ is defined as
\begin{eqnarray}
i_X:={\rm min}\{-K_X \cdot C\,|\, C \subset X~{\rm rational~curve}\}.
\end{eqnarray}

The pseudoindex is upper bounded for any Fano manifold and those reaching the extremal values are classified:

\begin{theorem}[\cite{CMSB,Ke}, \cite{Mi}]\label{them:pibig} Let $X$ be a Fano manifold of dimension $m \geq 2$. Then $i_X \leq m+1$. Furthermore,
\begin{enumerate}
\item if $i_X=m+1$, then $X$ is a projective space $\P^m$;
\item if $i_X=m$, then $X$ is a smooth quadric hypersurface $Q^m$.
\end{enumerate}
\end{theorem}

On the other hand by Proposition~\ref{prop:RCbasic} $\rm (4)$, if $X$ is a CP-manifold we have $i_X \geq 2$. Moreover, if $i_X=2$ then it follows from Proposition~\ref{prop:RCbasic} $\rm (5)$ that $X$ admits a smooth $\P^1$-fibration structure, which implies that its Picard number is bigger than one. Hence, when the Picard number is one, the pseudoindex is lower bounded by $3$. When equality holds we have the following classification result (see \cite{Hw4, Mk}):

\begin{theorem}\label{thm:CPindex} Let $X$ be an $m$-dimensional CP-manifold of Picard number one, $m\geq 2$. Then $i_X \geq 3$ and if $i_X=3$ then $X$ is $\P^2$, $Q^3$ or $K({\rm G}_2)$, where $K({\rm G}_2)$ is the $5$-dimensional contact homogeneous manifold of type ${\rm G}_2$ defined in Example~\ref{eg:G2}.
\end{theorem}

\begin{proof}[Sketch of the proof of \ref{thm:CPindex}]

By Proposition~\ref{prop:RCbasic}, a minimal rational component $\cM$ is a projective manifold of dimension $m$. Furthermore, $q: \cU \to X$ is a smooth morphism of relative dimension one and $p: \cU \to \cM$ is a smooth $\P^1$-fibration. By using the hyperbolicity of the moduli space of curves, one can prove that $q$ is also a smooth $\P^1$-fibration (see \cite[Lemma~1.2.2]{Mk}). Hence, the universal family $\cU$ admits two smooth $\P^1$-fibrations, and the result follows from Theorem \ref{thm:pic2} below.
\end{proof}

\begin{remark}\label{rem:FT}
The original proof of Theorem \ref{thm:CPindex}, due to Mok (\cite{Mk}) and complemented by Hwang (\cite{Hw4}), follows, at this point, a different line of argumentation. First of all, since $-K_X\cdot C=3$ for $[C] \in \cM$, it is possible to find a unisecant divisor for the fibers of $p: \cU \to \cM$ (see \cite[Corollary~1.3.1]{Mk} for details), so that in particular $\cU$ is the projectivization of a rank two vector bundle $\cE$ over $\cM$ which is proved to be stable (see \cite[Proposition~2.1.1]{Mk}).
The stability of $\cE$ implies numerical restrictions on its Chern classes of $\cE$. In fact, if the second Chow group $A_2(\cM)_\Q$ is isomorphic to $\Q$ then one gets Bogomolov inequality
\begin{equation}\label{eq:bogomolov}c_1(\cE)^2 \leq 4c_2(\cE).
\end{equation}
Let us comment that in \cite{Mk} the assumption on the fourth Betti number of $X$ to be one is used to prove $b_4(\cM)=1$. However, in \cite{Hw4} it is pointed out that this assumption on $b_4(X)$ can be removed, using the isomorphism $A_2(\cM)_\Q \simeq \Q$ which is a consequence of the fact that $\cU$ admits two smooth $\P^1$-fibrations.
As a consequence of the inequality (\ref{eq:bogomolov}) above (see \cite[Proposition~2.2.1]{Mk}), it is then shown that the VMRT $\cC_x$ at a general point $x \in X$ is a rational curve of degree at most $3$, which in particular implies that $\cC_x \subset \P(T_{X,x}^\vee)$ is projectively equivalent to the one of VMRT of $\P^2$, $Q^3$ or $K({\rm G}_2)$ (see Table \ref{tab:VMRTs}). At this point we could already conclude by using Theorem~\ref{them:HH}, but this result is posterior to \cite{Mk,Hw4}.
The original proof includes a case by case analysis in terms of the degree $d$ of $\cC_x$. In this study some analytic techniques previously developed by Hwang and Mok are needed to know when a distribution containing the one spanned by the VMRT's is integrable or not. This is enough to deal with the cases $d=1$ or $d=2$. For example, when $d=1$ the distribution spanned by the VMRT's is on one hand integrable and, on the other, not integrable unless it coincides with $TX$. This implies that $\dim X=2$ and $X=\P^2$. Similar arguments lead to $\dim X=3$ (and $X=Q^3$) when $d=2$ and to $\dim X=5$ when $d=3$. In this last case the way in which $X$ is proved to be isomorphic to $K({\rm G}_2)$ involves the recognition of its contact structure and a result of Hong (see \cite[Proposition~3.1.4]{Mk}) that allows to reconstruct a homogeneous contact manifold upon its VMRT's (an antecedent of Theorem~\ref{them:HH}).
\end{remark}

As a consequence of  the two previous theorems, we may obtain the complete list of CP-manifolds of dimension less than or equal to $4$ and Picard number one and, as a consequence we get the following:

\begin{corollary}\label{cor:CP3and4} Conjecture \ref{conj:CPconj} holds for manifolds of Picard number one and dimension $3$ or $4$.
\end{corollary}

An inductive argument on the dimension will then provide the complete classification of CP-manifolds of dimension at most $5$ and Picard number bigger than one:

\begin{theorem}{\cite{CP,CP2,Wa}}\label{them:CPlowdim} Conjecture \ref{conj:CPconj} holds for manifolds of Picard number greater than one and dimension at most $5$.
\end{theorem}

\begin{proof}[Sketch of the proof]
Let $X$ be a CP-manifold with Picard number greater than one. Then $X$ admits at least two contractions of extremal rays. By Proposition \ref{thm:smooth}, these contractions are smooth morphisms, and their fibers and targets are again CP-manifolds. Hence, induction applies.  Since the general strategies to tackle the cases of Picard number greater than one in \cite{CP,CP2,Wa} are the same, we focus on the easiest case, that is, the case of dimension $3$.

Assume therefore that $\dim X=3$. Then any fiber of an elementary contraction is either $\P^1$ or $\P^2$ and the target space is either $\P^1$, or $\P^2$ or $\P^1 \times \P^1$. This implies that any elementary contraction is either a $\P^1$-bundle over $\P^2$ or $\P^1 \times \P^1$, or a $\P^2$-bundle over $\P^1$. In fact, a smooth $\P^k$-fibration over a rational manifold is a $\P^k$-bundle, since the Brauer group of a rational manifold is trivial (cf., for instance, \cite[Proposition~2.5]{Wa}). According to a case-by-case analysis, it is possible to show that $X$ is isomorphic either to $\P^1 \times \P^1 \times \P^1$, or to $\P^1 \times \P^2$ or to $\P(T_{\P^2})$.
\end{proof}

\begin{remark}[Note Added in Proof]\label{rem:kanemitsu}
Recently we have been informed that the general strategy towards Conjecture \ref{conj:CPconj} that we will present in Section \ref{sec:FTman} can be used to give a proof of it in dimension $5$ (cf. \cite{Kane}).
\end{remark}

\subsection{Fano manifolds with two smooth $\P^1$-fibration structures}

We finish this section with the classification of Fano manifolds of Picard number $2$ admitting two different smooth $\P^1$-fibrations, that we have used to prove Theorem \ref{them:CPlowdim}. As we will see, a similar statement (Theorem \ref{thm:main}) holds for Fano manifolds of any Picard number $n$. In section \ref{sec:FTman} we will discuss this result and its possible use as a starting point to attack the Campana-Peternell conjecture in general.

\begin{theorem}\label{thm:pic2}
Let $X$ be a Fano manifold of Picard number $2$ which admits two different smooth $\P^1$-fibration structures. Then $X$ is isomorphic to $G/B$ with $G$ a semisimple Lie group of type ${\rm A}_1\times {\rm A}_1$, ${\rm A}_2$, ${\rm B}_2$ or ${\rm G}_2$, and $B$ a Borel subgroup of $G$.
\end{theorem}

If $X$ is a surface, then it is easy to see that $X$ is isomorphic to $\P^1\times \P^1$, that is the complete flag manifold of type ${\rm A}_1\times {\rm A}_1$. Hence we assume that the dimension of $X$ is at least $3$. Along the rest of this section, we use the following notation.

\begin{notation}\label{nota:pic2} Let $X$ be an $m$-dimensional Fano manifold of Picard number $2$ having two smooth $\P^1$-fibrations $\pi_1:X\to X_1$, $\pi_2:X\to X_2$. We assume that $m \geq 3$.  Let $K_i$ be the relative canonical divisor of $\pi_i$, $H_i$ the pull-back of the ample generator of $X_i$, $r_{X_i}$ the Fano index of $X_i$, $i=1,2$. Set $\nu_1:=K_1\cdot \Gamma_2$, $\mu_1:=H_1\cdot \Gamma_2$, $\nu_2:=K_2\cdot \Gamma_1$, $\mu_2:=H_2\cdot \Gamma_1$.
Without loss of generality, we may assume that $\nu_2 \geq \nu_1$.
\end{notation}
Note that, for $j=1,2$, $\mu_j> 0$, since $\pi_i$ does not contract the curve $\Gamma_j$ if $j \ne i$. Moreover, the ruled surface $S=\pi_1^{-1}(\pi_1(\Gamma_2))$ contains $\Gamma_2$ as a minimal section, being the family of deformations of $\Gamma_2$ unsplit. This gives $\nu_1 \geq 0$ and furthermore, if $\nu_1=0$ then $S$ is isomorphic to $\P^1 \times \P^1$. This implies $X=S$ because any $\Gamma_i$ meeting $S$ is contained in $S$ and $X$ is chain connected with respect to the families of deformations of the $\Gamma_i$'s. This contradicts our assumption $m \geq 3$. We have then proved:
\begin{equation}\label{eq:nusandmus}
\mu_j, \nu_j >0 \mbox{ for~} j=1,2.
\end{equation}

Although in principle the $\pi_j$'s are not necessarily $\P^1$-bundles, the argument of \cite{Hw4} quoted in Theorem~\ref{thm:CPindex} allows to prove that $H_j^2$ generates the $\Q$-vector space of codimension two cycles on $X$ that are pull-backs by the $\pi_j$. Since $K_j^2$ is of this type (this can be shown restricting to a ruled surface of the type $\pi_j^{-1}(C)$), the following Chern-Wu numerical relations hold (all details in \cite[Lemma~3]{MOSW}):
\begin{equation}\label{eq:chernwu}
K_j^2=\Delta_j H_j^2 \mbox{ for some~} \Delta_j \in \Q \mbox{ and~} j=1,2.
\end{equation}

Moreover, $-K_j+\frac{\nu_j}{\mu_j}H_j$ is nef but not big because it is numerically proportional to $H_i$ ($i \ne j$) then $(-K_j+\frac{\nu_j}{\mu_j}H_j)^m=0$. This equality and the Chern-Wu relations leads to the contradiction $\nu_j=0$ if $\Delta_j \geq 0$ (the precise computations in \cite[Lemma~3]{MOSW}). Then we have shown:
$$\Delta_j  <0 \mbox{ for~} j=1,2.$$
Let us show how these numerical considerations narrow down the possibilities of the values of $m$.

\begin{proposition}{\cite[Lemma~5 and Proof of Theorem~5]{MOSW}}\label{prop:235} Under the setting in Notation~\ref{nota:pic2}, we have $m=3,4$ or $6$.
\end{proposition}

\begin{proof} Let us set $b_1:=\sqrt{-\Delta_1}$ and $z_1:=\frac{\nu_1}{\mu_1}+ib_1\in\C$, $j=1,2$. Since, as said before, $-K_1+\frac{\nu_1}{\mu_1}H_1$ is nef and not big, then we have
$$0\leq \left(-K_1+\frac{\nu_1}{\mu_1}H_1\right)^{j}H_1^{m-j},$$ $j=0,\dots, m$, being $0$ when $j=m$. These inequalities reduce to (cf. \cite[Proposition~4.4]{kyoto}):
$$0 \leq \dfrac{\big(\frac{\nu_1}{\mu_1}+ib_1\big)^{j}- \big(\frac{\nu_1}{\mu_1}-ib_1\big)^{j}}{2ib_1},
$$ with equality when $j=m$.
This implies that the argument of $z_1$ is $\frac{\pi}{m}$ and then
$$ -\Delta_1=\frac{\nu_1^2}{\mu_1^2}{\rm tan}^2\left(\frac{\pi}{m}\right) \in \Q.
$$
Since the algebraic degree of ${\rm tan}(\frac{\pi}{m})$ over $\Q$ is known (see \cite[pp. 33-41]{Ni}), we conclude that $m$ has the stated values.
\end{proof}

Now it can be shown that the $\P^1$-fibrations are in fact $\P^1$-bundles. For this we need to find a unisecant divisor. Just computing $K_1\cdot H_1^{m-1}$ (see \cite[Proof of Theorem~5]{MOSW} for details) one can control the intersection numbers $\nu_1$ and $\nu_2$ and then (see \cite[Proposition~4.5]{Wa2}) choose the proper unisecant divisors. To be precise we get:

\begin{proposition}{\cite[Proof of Theorem~5]{MOSW}, \cite[Proposition~4.5]{Wa2}}\label{prop:nu} Under the setting in Notation~\ref{nota:pic2}, we have the following.
\begin{enumerate}
\item $\nu_1\nu_2=4\cos^2\left(\frac{\pi}{m}\right).$
\item $(\nu_1, \nu_2)=(1,1), (1,2)$ or $(1,3)$.
\item For $j=1, 2$, there exists a rank $2$ vector bundle $\cE_j$ on $X_j$ such that $X=\P(\cE_j)$ and $\pi_j$ is given by the projection $\P(\cE_j) \to X_j$.
\end{enumerate}
\end{proposition}

With this restrictions on $m$ we can control the Fano index of $X_i$ to get:

\begin{proposition}{\cite{MOSW}}\label{prop:muindex} Under the setting in Notation~\ref{nota:pic2}, we have the following.
\begin{enumerate}
\item $\mu_1=\mu_2$.
\item $(m, \mu_1, r_{X_1}, \mu_2, r_{X_2})=(3, 1, 3, 1, 3), (4, 1, 3, 1, 4)~ \mbox{or}~ (6, 1, 3, 1, 5).$
\end{enumerate}
\end{proposition}

\begin{proof}
Let us prove (1): take $L_i$ to be the tautological divisor of $\pi_i: X=\P(\cE_i) \to X_i$. Since $\{H_i, L_i\}$ is a $\Z$-basis of ${\rm Pic}X$ for each $i$,  there exist integers $\alpha, \beta, \gamma, \delta$ such that

\begin{equation}\label{eq:basechange}
\begin{pmatrix}H_2\\L_2\end{pmatrix}
 =B\begin{pmatrix}H_1\\L_1 \end{pmatrix},
\mbox{ where~}
B:=\begin{pmatrix}\vspace{0.2cm}
\alpha&\beta \\
\gamma&\delta
\end{pmatrix}
\mbox{with}~|{\rm det}B|=1. \nonumber
\end{equation}
By a straightforward computation, we see that
$$\alpha=-\frac{\mu_2}{\mu_1}L_1\cdot \Gamma_2,~~ \beta=\mu_2, ~~\gamma=\frac{1}{\mu_1}\Bigl( 1-(L_2\cdot \Gamma_1)(L_1\cdot \Gamma_2) \Bigr), ~~\delta=L_2 \cdot \Gamma_1.$$
From $|{\rm det}B|=1$ it follows that $\mu_1=\mu_2$ as desired.

To prove (2) recall that $\mu_1r_{X_1}=2+\nu_1$. Applying Proposition~\ref{prop:nu} and $\rm (1)$, we have
$$(m, \mu_1, r_{X_1}, \mu_2, r_{X_2})=(3, 1, 3, 1, 3), (3, 3, 1, 3, 1), (4, 1, 3, 1, 4)~ \mbox{or}~ (6, 1, 3, 1, 5).
$$ If $m=3$, then $X_i$ is a Fano manifold of dimension $2$ and Picard number $1$, that is $\P^2$. This implies that the case $(m, \mu_1, r_{X_1}, \mu_2, r_{X_2})=(3, 3, 1, 3, 1)$ does not occur.
\end{proof}

\begin{proof}[Proof of Theorem~\ref{thm:pic2}] By Proposition~\ref{prop:muindex}, we have the list of the possible values of $(m, \mu_1, r_{X_1}, \mu_2, r_{X_2})$. Since all cases are done in a similar way, we only deal with the case $(m, \mu_1, r_{X_1}, \mu_2, r_{X_2})=(6, 1, 3, 1, 5)$. From the Kobayashi-Ochiai Theorem (or Theorem~\ref{them:pibig}) it follows that $X_2$ is isomorphic to $Q^5$.

Let $H$ be the ample generator $H$ of $\Pic(Q^5)$; we have $H^4(X_2, \Z)=\Z[H^2]$, so  there exist integers $c_i$ such that $c_1(\cE_2)=c_1H$ and $c_2(\cE_2)=c_2H^2$. Assume without loss of generality that $\cE_2$ is normalized, that is $c_1=0$ or $-1$. Via numerical computations, we can show that $\cE_2$ is a stable vector bundle with $(c_1, c_2)=(-1, 1)$ (see \cite[Proposition~4.12]{kyoto} for details). 

Since a Cayley bundle on $Q^5$ is the only stable vector bundle of rank $2$ with $(c_1, c_2)=(-1, 1)$ we see that $X$ is isomorphic to the complete flag of type ${\rm G}_2$, described in  Example~\ref{eg:G2}.
\end{proof}


\section{Semiample tangent bundles}\label{sec:symp}


Associated with a CP-manifold $X$ one may consider two canonical auxiliary varieties, that may help  to understand the geometry of $X$. On one hand we have the projectivization of the cotangent bundle $\P(\Omega_X)$, which is the ambient space of the minimal rational tangents to $X$, defined in Section \ref{ssec:ratcurves}. On the other, we may consider the projectivization of the dual bundle, $\P(T_X)$, which we have already introduced to define the nefness of $T_X$. As a general philosophy, if the Campana-Peternell conjecture is true, one should be able to recognize the homogeneous structure of $X$ by looking at the loci of $\P(T_X)$ in which $\cO(1)$ is not ample. The expectancy is that $\cO(1)$ is semiample and that those loci appear as the exceptional loci of the associated contraction, as in the case of rational homogeneous manifolds:

\begin{example}\label{ex:nilpotent}
For every rational homogeneous manifold $X=G/P$,
one has $\fg=H^0(X,T_X)$, and the evaluation of global sections provides a generically finite morphism $\eps:\P(T_{X})\to \P(\fg)$, that contracts curves on which $\cO(1)$ has degree $0$. On the other hand, we have the adjoint action of $G$ on $\P(\fg)$, and it is well known that the image of $\eps$ may be described as the closure $\overline{O}$ of the quotient by the natural $\C^*$-action of a nilpotent orbit, i.e. the orbit by $G$ in $\fg^\vee$ of a nilpotent element. By abuse of notation, we will refer to $O$ as a {\it nilpotent orbit in} $\P(\fg)$. It is known that the orbit $O\subset\P(\fg)$ is the image of the set $\cX_0\subset\P(T_X)$ in which $\eps$ is finite, and that the boundary $\overline{O}\setminus O$ consists of a union of smaller nilpotent orbits, whose inverse images in $\P(T_X)$ correspond to irreducible components of the stratification of $\P(T_X)$ in terms of dimension of the fibers of $\eps$. Moreover, the geometry of nilpotent orbits and their boundaries can be written in terms of combinatorial objects associated to $\fg$, such as partitions and weighted Dynkin diagrams. We refer the reader to \cite{CoMc} for a complete account on nilpotent orbits, and to \cite{Fu2} for a survey on their resolutions.
\end{example}

In this section we will discuss the existence of a contraction of $\P(T_X)$ associated to the nef tautological line bundle $\cO(1)$, and we will study the basic properties of this contraction, in case it does exist. As a consequence, we will finally show that Conjecture \ref{conj:CPconj} holds in the particular case in which $T_X$ is big and $1$-ample.

 \begin{notation}\label{notn:cpmanifoldsections}
 We will denote by $\phi:\P(T_X)\to X$ the canonical projection, by  $\cO(1)$ the corresponding tautological line bundle, which is nef by definition of CP-manifold. In particular we may write $\cO(-K_{\P(T_X)})=\cO(m)$, where $m:=\dim(X)$.
Throughout this section we will always assume that $T_X$ is not ample, i.e. that $X$ is not a projective space. This hypothesis allows us to consider the following:
with the same notation as in \ref{not:cpmanifold}, for every $i$ we will denote by $\overline{\Gamma}_i$ a minimal section of $\P(T_X)$ over the minimal rational curve $\Gamma_i$,
 corresponding to a quotient $f_i^*(T_X)\to\cO_{\P^1}$ (being $f_i:\P^1\to \Gamma_i$ the normalization of $\Gamma_i$).
We denote by $\overline{f}_i$ the normalization of $\overline{\Gamma}_i$.
\end{notation}

\subsection{Semiampleness of $T_X$}\label{ssec:semiample}

The following statement is immediate:

\begin{lemma}\label{lem:bigO(1)}
With the same notation as above, the Mori cone
$\cNE{\P(T_X)}$
is generated by the class of a line in a fiber of $\phi:\P(T_X)\to X$ and by the classes of $\overline{\Gamma}_i$, $i=1,\dots,n$. Moreover, $\cO(1)$ is big on $\P(T_X)$ if and only if there exist an effective $\Q$-divisor $\Delta$ satisfying $\Delta\cdot \overline{\Gamma}_i<0$, for all $i=1,\dots,n$.
\end{lemma}

\begin{proof}
  Let $N_0\subset \NE(\P(T_X))$ be the cone generated by the classes
  of $\overline{\Gamma}_i$, $i=1,\dots,n$. The push-forward morphism
  $\phi_*:N_1(\P(TX))\to N_1(X)$ which sends the class of
  $\overline{\Gamma}_i$ to the class of $\Gamma_i$ induces an isomorphism
  of $N_0$ with $\NE(X)$. Thus $N_0$ is the facet of $\NE(\P(TX))$
  which is supported (orthogonal in the sense of intersection) by the
  numerical class $L$ of $\cO(1)$. Since
  $\NE(\P(T_X))\subset(\phi_*)^{-1}\NE(X)\cap\left\{Z\in N_1(\P(TX))|\,
  Z\cdot L\geq 0\right\}$ the first claim follows. For the second part, note
  that $\cO(1)$ is big if and only if $L$ lies in the interior of the
  pseudo-effective cone of $\P(T_X)$ (that is, the closure of the cone
  generated by classes of effective divisors) or, equivalently, if and
  only if for every ample divisor $A$ and sufficiently small
  $\epsilon\in\Q_{>0}$, $\Delta=L-\epsilon A$ is effective.
\end{proof}

\begin{definition}\label{def:semiample} A line bundle
$L$ on a variety $X$ is {\em semiample} if $L^{\otimes r}$ is generated by global
sections for $r\gg 0$; a vector bundle $\cE$ is semiample if the tautological bundle $\cO(1)$ on $\P(\cE)$ is semiample.
\end{definition}

If a line bundle $L$ is semiample then the graded ring of its
sections $R(X,L)=\bigoplus_{r\geq 0} H^0(X,L^{\otimes r})$ is a
finitely generated $\C$-algebra, and
the evaluation of sections $H^0(X,L^{\otimes
r})\otimes\cO_X\ra L^{\otimes r}$ yields a proper surjective morphism
of projective schemes (with connected fibers):
$$\eps: \Proj_X\left(\bigoplus_{r\geq 0} L^{\otimes r}\right)\lra
Y_L=\Proj\left(\bigoplus_{r\geq 0} H^0(X,L^{\otimes r})\right)$$ which we will
call  {\em evaluation morphism} or the {\em contraction of $X$ associated to
$L$}.

Contrary to the ample case, nef bundles are not necessarily semiample. Hence, it makes sense to pose the following weak form of Conjecture \ref{conj:CPconj2}:
\begin{question}
Let $X$ be a CP-manifold. Is $T_X$ semiample?
\end{question}

The standard technique to answer this question is the Basepoint-free theorem, which, in our situation, provides:

 \begin{proposition}\label{prop:nefeffdiv}
 With the same notation as above, for any CP-manifold $X$, the following are equivalent:
 \begin{enumerate}
   \item There exists an effective divisor $\Delta$ satisfying $\Delta\cdot\overline{\Gamma}_i<0$ for all $i$.
   \item $T_X$ is big.
   \item $T_X$ is semiample and big.
 \end{enumerate}
\end{proposition}
\begin{proof}
($1\iff 2$) follows from Lemma \ref{lem:bigO(1)}, and ($2\iff 3$) follows from the usual Basepoint-free theorem.
\end{proof}

\subsection{A birational contraction of $\P(T_X)$}\label{ssec:sympsetup}

Throughout the rest of section \ref{sec:symp}, we will always assume that $T_X$ is big and semiample (see Proposition \ref{prop:nefeffdiv}), i.e. that the evaluation of global sections defines a birational morphism
$$\eps: \cX:=\Proj_X\left(\bigoplus_{r\geq 0} S^rT_X\right)\lra
\cY:=\Proj\left(\bigoplus_{r\geq 0} H^0(X,S^rT_X)\right).$$

Alternatively one may consider the total spaces $\widehat{\cX}$ and $\widehat{\cY}$ of the tautological line bundles $\cO(1)$ on the $\Proj$-schemes $\cX$ and $\cY$, and the natural map:
$$
\widehat{\eps}:\widehat{\cX}:=\Spec_\cX\left(\bigoplus_{r\in\Z}\cO(r)\right)
\longrightarrow\widehat{\cY}:=\Spec_\cY\left(\bigoplus_{r\in\Z}H^0(X,\cO(r))\right).
$$
The scheme $\widehat{\cX}$ may also be described as the total space of the cotangent bundle of $X$ with the zero section removed, and we have a fiber product diagram:
$$
\xymatrix{\widehat{\cX}\ar[r]^{\widehat{\eps}}\ar[d]&\widehat{\cY}\ar[d]\\\cX\ar[r]^{\eps}&\cY}
$$
where the vertical arrows are quotients by $\C^*$-actions.

\subsection{The contact structure of $\P(T_X)$}\label{ssec:crepcont}

We will see that the contractions $\eps$ and $\widehat{\eps}$ enjoy very special features, basically due to the fact that $\cX=\P(T_X)$ supports a {\em contact structure} $\cF$, defined as the kernel of the composition of the differential of $\phi$ with the co-unit map
$$
\theta: T_{\cX}\stackrel{d\phi}{\longrightarrow}\phi^*T_X=\phi^*\phi_*\cO(1)\longrightarrow\cO(1).
$$
Note that $\theta$ fits in the following commutative diagram, with exact rows and columns:
\begin{equation}\label{eq:contact}
\xymatrix@=35pt{T_{\cX/X}\ar@{>->}[r]\ar@{=}[d]&\cF\ar@{->>}[r]\ar@{>->}[]+<0ex,-2ex>;[d]&\Omega_{\cX/X}(1)\ar@{>->}[d]\\
          T_{\cX/X}\ar@{>->}[r]&T_{\cX}\ar@{->>}[r]\ar@{->>}[d]^{\theta}&\phi^*T_X\ar@{->>}[d]\\
          &\cO(1)\ar@{=}[r]&\cO(1)}
\end{equation}

The distribution $\cF$ being contact means precisely that it is maximally non integrable, i.e. that the morphism $d\theta:\cF\otimes\cF\to T_\cX/\cF\cong \cO(1)$ induced by the Lie bracket is everywhere non-degenerate. This fact can be shown locally analytically, by considering, around every point, local coordinates $(x_1,\dots,x_m)$ and vector fields $(\zeta_1,\dots,\zeta_m)$, satisfying $\zeta_i(x_j)=\delta_{ij}$. Then the contact structure is determined, around that point, by the $1$-form $\sum_{i=1}^m\zeta_idx_i$ (see \cite{KPSW} for details).

Following Beauville (\cite{Beau2}), 
the existence of a contact form on $\cX$ implies (it is indeed equivalent to) the existence of a {\it symplectic form} on $\widehat{\cX}$: a closed $2$-form $\sigma\in H^0(\widehat{\cX},\Omega^2_{\widehat{\cX}})$ which is everywhere nondegenerate, i.e. that induces a skew-symmetric isomorphism $T_{\widehat{\cX}}\to \Omega_{\widehat{\cX}}$. Locally analytically, with the same notation as above, the symplectic form induced by $\theta$ is the standard symplectic form on the cotangent bundle, given by $\sigma=\sum_id\zeta_i\wedge dx_i$.

\begin{remark}[Contact and symplectic manifolds in general]\label{rem:contsymp}
More generally, a smooth variety $M$ is called a {\it contact manifold} if it supports a surjective morphism from $T_M$ to a line bundle $L$, whose kernel is maximally non integrable, and it
is called {\it symplectic} if there exists an everywhere nondegenerate closed to form $\sigma\in H^0({M},\Omega^2_{{M}})$.
The relation contact/symplectic that we stated above for our particular situation can be generally presented as follows: given a contact form $\theta\in H^0(M,\Omega_M\otimes L)$ on a smooth variety $M$, the total space $\widehat{M}$ of the line bundle $L$ is a symplectic manifold. A projective birational morphism $\widehat{f}:\widehat{M}\to \widehat{N}$ from a symplectic manifold $\widehat{M}$ to a normal variety $\widehat{N}$ is called a {\it symplectic contraction} of $\widehat{M}$, or a {\it symplectic resolution} of $\widehat{N}$. This type of resolutions have been extensively studied by Fu, Kaledin, Verbitsky, Wierzba, and others. We refer the interested reader to \cite{Fu2} and the references there for a survey on this topic.
\end{remark}

Let us note also that the complete list of projective contact manifolds is expected to be small. In fact it is known (cf. \cite[Corollary~2]{De}) that 
their canonical divisor is not pseudo-effective and, in particular, it is not nef. This condition had been previously used in \cite{KPSW} to show that, with the exception of the manifolds of the form $\P(T_Z)$, projective contact manifolds cannot have non constant Mori contractions. We may then conclude that:

\begin{theorem}[\cite{De, KPSW}]\label{thm:KPSW}
Let $M$ be a projective contact manifold. Then, either $M$ is a Fano manifold of Picard number one or $M=\P(T_Z)$ for some smooth projective variety $Z$.
\end{theorem}

Finally we remark that it is conjectured (see \cite{SaLe}) that the only Fano contact manifolds of Picard number one are rational homogeneous: more concretely, minimal nilpotent orbits of the adjoint action of a simple Lie group $G$ on
$\P(\fg)$.

\subsection{Properties of the contraction $\eps$}\label{ssec:properties}

We will present here some of the properties that symplectic resolutions (and its contact counterparts) are known to satisfy. For the reader's convenience, we will state them in our particular setup, and we will refer the interested reader to \cite{Fu2} and \cite{Wie} for further details. The following proof has been taken from \cite[Remark 1]{Wie}.

\begin{lemma}\label{lem:crepant}
With the same notation as above, $\eps$ and $\widehat{\eps}$ are crepant contractions and, in particular, their positive dimensional fibers are uniruled.
\end{lemma}
\begin{proof}
The proof in both cases is analogous. In the projective setting, for instance, we have $R^i\eps_*\cO_{\cX}=R^i\eps_*(\omega_{\cX}\otimes\cO(\dim(X)))=0$ for $i>0$ (\cite[Corollary~2.68]{KM}). Then $\eps$ is a rational resolution and $\omega_{\cY}$ is a line bundle, isomorphic to $\eps_*\omega_{\cX}$ (cf. \cite[Section 5.1]{KM}). But then $\omega_{\cX}\otimes\eps^*\omega_{\cY}^{-1}$ is effective
 and vanishes on the $\overline{\Gamma}_i$'s, hence it is numerically proportional to $\cO(1)$. Since it is also exceptional, it is trivial.

 For the uniruledness of the fibers, we take (by Proposition \ref{prop:nefeffdiv}) an effective $\Q$-divisor $\Delta$ satisfying that $(\cX,\Delta)$ is klt and that $-\Delta$ is $\eps$-ample, and use \cite[Theorem~1]{Kaw}.
\end{proof}

The next proposition lists other important properties of the contraction $\eps$, inherited from analogous properties of the symplectic resolution $\widehat{\eps}$ (see \cite{Ka} for details, see also \cite{Fu2}):

\begin{proposition}\label{prop:sympprop}
With the same notation as above:
\begin{enumerate}
  \item There exists a stratification $\cY=D_0\supset D_1\supset D_2\dots$ of projective varieties such that $D_{i}$ is the singular locus of $D_{i-1}$, for all $i$, and every irreducible component of $D_i\setminus D_{i-1}$ is a contact manifold.
      In particular $\dim D_i$ is odd for all $i$.
  \item $\eps$ is {\it semismall}, i.e. for every closed subvariety $Z\subset \cX$, one has $\codim(Z)\geq \dim(Z)-\dim(\eps(Z))$.
\end{enumerate}
\end{proposition}

Finally we will recall the following statement, which is a particular case of a more general result by Wierzba (see \cite[Theorem~1.3]{Wie}), and that can be obtained by cutting $\cY$ with $2m-3$ general hypersurfaces passing through $P$, and using the classification of Du Val singularities of surfaces:

\begin{proposition}\label{prop:treeP1}
With the same notation as above, if moreover $\eps$ is an elementary divisorial contraction, then its exceptional locus is an irreducible divisor $D$, and any one dimensional fiber consists of either a smooth $\P^1$ or the union of two $\P^1$'s meeting in a point.
\end{proposition}

 \subsection{Minimal sections on $\P(T_X)$}\label{ssec:minsec}

In this section we will study the loci of the minimal sections $\overline{\Gamma}_i$'s of $\P(T_X)$ over the minimal rational curves $\Gamma_i$. Although it is not true in general that the exceptional locus of $\eps$ is swept out by these curves (see Example \ref{ex:isograss} below), the loci of the $\overline{\Gamma}_i$'s may contain substantial information on the contraction $\eps$.

\begin{notation}\label{notn:curvesandsections}
For simplicity, let as fix an index $i\in\{1,\dots,n\}$ and denote $\Gamma:=\Gamma_i$. We choose a rational curve in the class (that, abusing notation, we denote by $\Gamma$ as well), denote by $p :\cU \to\cM $ its family of deformations and by $q:\cU\to X$ the corresponding evaluation morphism. We may consider the irreducible component $\overline \cM $ of $\rat^n(\cX)$ containing a minimal section $\overline{\Gamma} $ of $\cX$ over $\Gamma$ and the corresponding universal family, fitting in a commutative diagram:
$$\xymatrix{\overline{\cM} \ar[d]^{\overline{\phi} }&\overline{\cU} \ar[d]\ar[l]_{\overline{p} }\ar[r]^{\overline{q} }&\cX\ar[d]^{\phi}\\
            \cM &\cU \ar[l]^{p }\ar[r]_{q }&X}
$$
We set $c :=-K_X\cdot\Gamma -2$. Note that \ref{prop:RCbasic} (6) implies that  the fibers of $\overline{\phi} $ over every standard deformation of $\Gamma $ are isomorphic to $\P^{m-c -2}$, so $\overline{\cM} $ has dimension $2m-3$.
Finally, for simplicity, we will denote by $E(a_1^{k_1},\dots,a_r^{k_r})$ the vector bundle $\bigoplus_{j=1}^r\cO(a_j)^{\oplus k_j}$ on $\P^1$.
\end{notation}

The next proposition describes the infinitesimal deformations of a general member of $\overline{\cM} $.
\begin{proposition}\label{prop:splittype}
With the same notation as above, let $\overline{f} :\P^1\to \cX$ denote the normalization of a minimal section $\overline{\Gamma} $ of $\cX$ over a standard rational curve in the class $\Gamma $. Then $\overline{\cM} $ is smooth at $\overline{\Gamma} $, of dimension $2m-3$, and
 $$
\overline{f}^*T_{\cX}\cong E\big(-2,2,(-1)^{e },1^{e },0^{2m-3-2e }\big),\mbox{ for some }e \leq c .
$$
\end{proposition}

\begin{proof}
Writing ${f}^*T_X=E(2,1^{c },0^{m-c -1})$ and taking in account that $\overline{f}^*\cO(1)=\cO$, the relative Euler sequence of $\cX=\P(T_X)$ over $X$, pulled-back via $\overline{f} $ provides $\overline{f}^*T_{\cX/X}=E(-2,(-1)^{c },0^{m-c -2})$. Then, the upper exact row of diagram (\ref{eq:contact}) provides:
$$
0\to E(-2,(-1)^{c },0^{m-c -2})\longrightarrow \overline{f}^*\cF\longrightarrow E(2,1^{c },0^{m-c -2})\to 0.
$$
On the other hand, $\overline{f}^*\cO(1)=\cO$ also implies that $d\overline{f} :T_{\P^1}\to \overline{f}^*T_{\cX}$ factors via $\overline{f}^*\cF$, hence this bundle has a direct summand of the form $\cO(2)$. Being $\cF$ a contact structure, it follows that $\overline{f}^*\cF\cong\overline{f}^*\cF^\vee$, so this bundle has a direct summand $\cO(-2)$, as well.

From this we may already conclude that
\begin{equation}\label{eq:splitcontact}
\overline{f}^*\cF\cong E(-2,2,(-1)^{e },1^{e },0^{2m-2e -4}), \mbox{ for some }e \leq c ,
\end{equation}
hence the bundle $\overline{f}^*T_\cX$ is isomorphic either to
$E(-2,2,(-1)^{e },1^{e },0^{2m-2e -3})$ or to $E(2,(-1)^{e +2},1^{e },0^{2m-2e -4})$.
On the other hand, the fact that $\dim\overline{\cM} =2m-3$ implies that $h^0(\P^1,\overline{f}^*T_\cX)\geq 2m$, which allows us to discard the second option. Finally, in the first case $h^0(\P^1,\overline{f}^*T_\cX)$ is precisely equal to $\dim_{[\overline{f} ]}\Hom(\P^1,\cX)=2m$, hence this scheme is smooth at $[\overline{f} ]$ and $\overline{\cM} $ is smooth at $\overline{\Gamma} $.
\end{proof}

\begin{definition}\label{def:defectcurve}
Given a minimal section $\overline{\Gamma} $ over a minimal rational curve $\Gamma$ as above, the number $e $ provided by the proposition above will be called the {\it defect of $\overline{\cM} $ at }
$\overline{\Gamma} $.
\end{definition}

\subsection{Dual varieties}\label{ssec:dual}
Let us denote by $D \subset \cX$ the closure of $\overline{q }(\overline{\cU} )$, which by construction is a subset of the exceptional locus of $\eps$. The next result presents the relation between $D_{x}:=D \cap\P(T_{X,x})$ and the VMRT, $\cC_{x}\subset\P(\Omega_{X,x})$, of the family $\cM $ at the general point $x$. 

\begin{proposition}\label{prop:dual}
With the same notation as above, being $x\in X$ general, $D_{x}$ is the dual variety of $\cC_{x}$.
\end{proposition}

\begin{remark}[Dual varieties of projective subvarieties]\label{rem:dualvar} We refer the reader to \cite{Tev} for an account on dual varieties. Here it is enough to recall that, given a reduced projective variety $M\subset\P^r=\P(V)$, and denoting by $M_0\subset M$ its subset of smooth points, the Euler sequence provides a surjection $\cO_{M_0}\otimes V^\vee\to \cN_{M_0,\P^r}(-1)$, so that we have a morphism: $p_2:\P(\cN_{M_0,\P^r}(-1))\to \P(V^\vee)$ whose image $M^\vee$ is called the {\it dual variety of }$M$. In other words, $M^\vee$ may be described as the closure of the set of tangent hyperplanes of $M$. That is, we may consider $\P(\cN_{M_0,\P^r}(-1))$ as a subset of $F(0,r-1):=\P(T_{\P^r})\subset\P^r\times{\P^r}^\vee$ and denote by $F(M)$ its closure. Then the restrictions ($p_1$ and $p_2$) to $F(M)$ of the canonical projections have images $M$ and $M^\vee$, respectively:
$$
\xymatrix{\P^r&F(0,r-1)\ar[r]\ar[l]&{\P^r}^\vee\\
M\ar@{^{(}->}[]+<0ex,2.5ex>;[u]&F(M)\ar@{^{(}->}[]+<0ex,2.5ex>;[u]%
\ar[r]^{p_2}\ar[l]_{p_1}&M^\vee\ar@{^{(}->}[]+<0ex,2.5ex>;[u]}
$$
Finally, let us recall that the biduality theorem states that $M^{\vee\vee}=M$, so that the diagram above is reversible, and we may assert that the general fiber of $p_2$ (the so-called {\it tangency locus} of a hyperplane) is a linear space. In particular one expects $p_2$ to be, indeed, birational for most projective varieties; those varieties for which $e(M):=r-1-\dim(M^\vee)$ is positive are then called {\it dual defective}, and $e(M)$ is called the {\it dual defect} of $M$.
\end{remark}

In the next example we are going to compute the dual variety of $M=\cC_x\subset \P(\Omega_{X,x})$, being $X$ the Lagrangian Grassmannian of lines in $\P^5$ (i.e. the general linear section of $G(1,5)$ in its Pl\"ucker embedding). According to Proposition \ref{prop:VMRTshort2}, $M$ is isomorphic to the $\P^2$-bundle $\P(E(1^2,2))$ over $\P^1$, embedded in $\P(\Omega_{X,x})\cong\P^6$ by the complete linear system of sections of $\cO(1)$.

\begin{example}\label{ex:dual112}
Let us consider $M=\P(E(1^2,2))$ as a subscheme of $\P^6$, embedded by the complete linear system of its tautological bundle $\cO(1)$. Equivalently, we may describe it as a general hyperplane section of the Segre embedding of $\P^1\times\P^3\subset\P^7$. Then the general theory of dual varieties tells us that $M^\vee\subset{\P^6}^\vee$ is isomorphic to a linear projection of $(\P^1\times\P^3)^\vee\cong\P^3\times\P^1\subset{\P^7}^\vee$ from a general point $P$. There exists a three dimensional linear space $V\subset{\P^7}^\vee$ containing $P$, and meeting $\P^1\times\P^3$ along a smooth quadric $\P(E(1^2))\subset V$, such that a line through $P$ is secant if and only if it is contained in $V$. We conclude that $M^\vee$ is a $4$-dimensional variety, whose normalization is $\P^3\times\P^1$, and in particular the dual defect of $M$ is $1$. The singular locus of $M^\vee$ may be described as the image of $\P(E(1^2))$ by the projection and so it is a plane $\Pi$. Note also that, denoting by $C\subset \Pi$ the branch locus of the projection $\P(E(1^2))\to\Pi$, $M^\vee$ is a scroll in $\P^3$'s, meeting $\Pi$ along a tangent line to $C$.
\end{example}

\begin{proof}[Proof of Proposition \ref{prop:dual}] Our line of argumentation here is based on the proof of \cite[Proposition~1.4]{Hw}.
Let $x\in X$ be a general point and $f :\P^1\to X$ be the normalization of a general $\Gamma $, satisfying $f (O)=x$, $O\in\P^1$. By Proposition \ref{prop:imm}, the tangent map $\tau_{x}:\cM_{x}\to\cC_{x}\subset\P(\Omega_{X,x})$ is immersive and we may use it to identify the tangent space of $\cC_{x}$ at $P:=\tau_{x}(\Gamma )$.

In order to see this, we denote by $\beta:X'\to X$ the blow-up of $X$ at $x$, with exceptional divisor $E:=\P(\Omega_{X,x})$. Note that we have a filtration $T_{X,x}\supset V_1(f )\supset V_2(f )$, where $V_1(f )$ and $V_2(f )$ correspond, respectively, to the fibers over $O$ of the (unique) subbundles of $f^*T_X$ isomorphic to $E(2,1^{c})$ and $E(2)$. Moreover $T_{E,P}$ is naturally isomorphic to the quotient of $T_{X,x}$ by $V_2(f )$, hence our statement may be re-written as $T_{\cC_{x},P}=V_1(f )/V_2(f )$.

Let us then consider the irreducible component of $\Hom(\P^1,X;O,x)$ (para\-me\-trizing morphisms from $\P^1$ to $X$, sending $O$ to $x$) containing $[f ]$ and note that the evaluation morphism factors
$$\xymatrix{\P^1\times \Hom(\P^1,X;O,x)\ar[rr]_(.65){\ev'}
\ar@/^1.3pc/[rrrr]^{\ev}&&X'\ar[rr]_{\beta}&&X}$$
In this setting, we have $T_{\cC_{x},P}=d\ev'_{(O,[f ])}\big(\{0\}\times H^0(\P^1,f^*T_X(-1)\big)/V_2(f )$, and we may identify $H^0(\P^1,f^*T_X(-1))$ with the global sections of $f^*T_X$ vanishing at $O$. Choosing now a set of local coordinates $(t,t_2,\dots,t_m)$ of $X$ around $x$ such that $f(\P^1)$ is given by $t_2=\dots,t_m=0$ and $t$ is a local parameter of $f(\P^1)$, and writing the blow-up of $X$ at $x$ in terms of these coordinates, one may check that, modulo $V_2(f)$, $d{\ev'}_{(O,[f ])}$ sends every section $s$ vanishing at $O$ to $\frac{s}{t}(O)=\frac{ds}{dt}(O)$,
 hence it follows that its image is $V_1(f )$.
\end{proof}

The next statements relates the dual defect of the VMRT of the family $\cM $ with its generic splitting type.

\begin{proposition}\label{prop:dualdef1s}
With the same notation as above, being $x\in X$ general, let  $\overline{\Gamma}$ be a minimal section of $\cX$ over a general element of $\cM_{x}$. Then the dual defect of $\cC_{x}$ equals the defect of $\overline{\cM} $ at $\overline{\Gamma}$.
\end{proposition}

\begin{proof}
We have to check that the image of $\overline{q}:\overline{\cU} \to \cX$ has dimension $2m-2-e$. Equivalently, denoting by $\overline{f} :\P^1\to \cX$ the normalization of $\overline{\Gamma}$, we may consider the image of the evaluation $\Hom(\P^1,\cX)_{[\overline{f}]}\times P^1\to \cX$, where $\Hom(\P^1,\cX)_{[\overline{f}]}$ stands for the irreducible component of $\Hom(\P^1,\cX)$ containing $[\overline{f}]$, and compute the rank of its differential at $([\overline{f}],x)$. Using the description of this differential provided in \cite[II, Proposition~3.4]{kollar}, the result follows then by noting that $e $ equals the dimension of the kernel of the evaluation of global sections $H^0(\P^1,\overline{f}^*(T_\cX))\otimes\cO_{\P^1}\to\overline{f}^*(T_\cX)$.
\end{proof}

The above interpretation of projective duality for VMRT's of Fano manifolds has a number of important consequences, particularly in the case of CP-manifolds. Let us illustrate this here by presenting a straightforward application to Conjecture \ref{conj:CPconj}; a more complete result in this direction can be found in \cite{OSW}.
A well known corollary of Zak's Theorem on Tangencies states that for any non linear smooth variety $M \subset \P^r$ it holds that $\dim(M^\vee) \geq \dim(M)$. Moreover, if we further assume that $\dim(M)\leq 2r/3$, the list of all the smooth projective varieties for which $\dim(M^\vee) =\dim(M)$ has been given by Ein, \cite[Theorem~4.5]{Ein1}. Applying this to the case of the VMRT of a family of minimal rational curves
 we get the following result:

\begin{proposition}
Let $X$ be a CP-manifold of Picard number one different from the projective space, and $\cM$ be a family of minimal rational curves in $X$. Assume that $\cC_{x} \subset \P(T_{X,x}^\vee)$ is smooth of dimension $c$ for general $x$ (this holds, for instance, under the assumptions of Remark \ref{rem:VMRTlines}). Then $c\leq m-2-e$, where $e$ denotes the dual defect of $\cC_x$, and moreover if $c \leq 2(m-1)/3$ then equality $c= m-2-e$ holds if and only if $X=G/P$, where $G$ is semisimple Lie group with Dynkin diagram $\cD$, $P$ is the parabolic subgroup associated to the $i$-th node of the diagram (the nodes of $\cD$ are numbered as in (\ref{eq:dynkins})), and the pair $(\cD,i)$ is one of the following:
$$
({\rm A}_{k+1},2),\,\,\,k\geq 2,\quad
({\rm B}_2,1),\quad ({\rm D}_{5},5),\quad
({\rm E}_{6},1).
$$

\end{proposition}

\begin{proof} Let us first observe, see \cite[Proposition~5]{Hw2}, that if the VMRT $\cC_{x} \subset \P(T_{X,x}^\vee)$ is linear then  $\cC_{x}=\P(T_{X,x}^\vee)$ and $X$ is a projective space by Theorem~\ref{them:pibig}. Since this case is excluded, Zak's Theorem on Tangencies applies to give the inequality $c\leq m-2-e$.

Let us now assume $c \leq 2(m-1)/3$ and $c=m-2-e$.
By using \cite[Theorem~4.5]{Ein1} we get that $\cC_x \subset \P(T_{X,x}^\vee)$ is either:
\begin{itemize}
  \item a hypersurface on $\P^2$ or $\P^3$,
  \item the Segre embedding of $\P^1 \times \P^{c-1} \subset \P^{2c-1}$,
  \item the Pl\"ucker embedding of $\G(1,4)$ ($c=6$),  
  \item the Spinor variety $S_{4} \subset \P^{15}$ ($c=10$).
\end{itemize}

In the first case Theorem~\ref{them:pibig} tells us that $X$ is a quadric of dimension $3$ or $4$. In the other cases, since the listed varieties are projectively equivalent
to the VMRT's of the homogeneous manifolds in the statement, we may conclude by Theorem \ref{them:HH}.
\end{proof}

We will finish this section by presenting an example in which the exceptional locus of $\eps$ does not consists only of minimal sections over minimal rational curves of $X$.

\begin{example}\label{ex:isograss}
Let $X=LG(1,5)$ be the Lagrangian Grassmannian of lines in $\P^5$, which parametrizes lines in $\P^5$ that are isotropic with respect to a nondegenerate skew-symmetric $6\times 6$ matrix $A$. It is known that the contraction of $\cX$ in this case may be described in representation theoretical language as we have already remarked in Example \ref{ex:nilpotent}. The evaluation of global sections of $T_X$ provides a morphism $\eps':\cX\to \P(\fsp_6)\cong\P^{20}$ whose image is the closure $\overline{O}$ of a nilpotent orbit.

In our case the orbit $O$ is {\it even} (\cite[Sect. 3.8]{CoMc}), which in turn implies that the morphism $\eps'$ is birational, and in particular it factors via $\eps:\cX\to \cY$, whose image is the normalization of $\overline{O}$. It is known that $\overline{O}$ is a disjoint union of orbits $O\cup O_{1}\cup O_2 \cup O_3$, where $\dim(O_i)=11,9,5$ for $i=1,2,3$, respectively, and $\Sing(\overline{O})=\overline{O_1}=O_{1}\cup O_2 \cup O_3$. Then, by Proposition \ref{prop:sympprop}, the inverse image of $\overline{O_1}$ contains a divisor $D$ in $\cX$, and so $\eps:\cX\to\cY$ is an elementary divisorial contraction.

On the other hand, we may consider the locus of the family of minimal sections of $\cX$ over lines in $X$, whose VMRT is isomorphic to the subvariety $M\subset\P^6$, described in Example \ref{ex:dual112}. Since the dual defect of $M$ is one, it turns out that this locus $D'$ has dimension $11$, and it is a proper closed subset of $D$.
\end{example}

\subsection{Contact forms on families of minimal rational curves}\label{ssec:liftcontact}

The next lemma shows how to transport the contact form of $\cX$ to an
open set of $\overline{\cM} $.

\begin{lemma}\label{lem:liftcontact}
With the same notation as above, let $\overline{\cM}\hs^0\subset
\overline{\cM} $ be the subset parametrizing minimal sections
of $\cX$ over standard rational curves of $\cM $,
and set $\overline{\cU}\,^0:=\overline{p}^{-1}(\overline{\cM}\hs^0)$.
Then there exists a line bundle $\cL $ on $\overline{\cM} $ such that
$\overline{p}^*\cL =\overline{q}^*\cO(1)$, and a twisted $1$-form $\overline{\theta} \in
H^0(\overline{\cM}\hs^0,\Omega_{\overline{\cM}\hs^0}\otimes\cL )$ such that $\overline{p}^*\overline{\theta} =
\overline{q}^*\theta$ on $\overline{\cU}\hs^0$.
\end{lemma}

\begin{proof}
The first part follows from the fact that $\overline{q}^*\cO(1)$ is trivial on the
fibers of $\overline{q} $. For the second,
we start by noting that the composition $T_{\overline{\cU}\hs^0}\to
\overline{q}^*T_{\cX}\stackrel{\theta}\to \overline{q}^*\cO(1)$ is surjective, as one may check
by restricting to every fiber of $\overline{q} $, by Proposition \ref{prop:splittype}.
Since the relative tangent bundle of $\overline{q} $ lies on its kernel, we obtain a
surjective morphism $\overline{p}^*T_{\overline{\cM}\hs^0}\to \overline{q}^*\cO(1)=\overline{p}^*L $.
Its push-forward to $\overline{\cM}\hs^0$ is the desired $1$-form $\overline{\theta} $,
since one may check that the relative cohomology of the kernel of that morphism is zero.
\end{proof}

The non-integrability of the induced form $\overline{\theta} $ is going to depend on the stratification of $\overline{\cM}\hs^0$ determined by the defect of sections defined in \ref{def:defectcurve}:

\begin{corollary}\label{cor:1formcontact}
With the same notation as above, assume that the open subset $\overline{\cM}' \subset\overline{\cM}\hs^0$ of curves on which the defect is zero is non-empty.  Then $\overline{\theta} $ is a contact form on $\overline{\cM}' $.
\end{corollary}

\begin{proof}
Let us denote by  $\overline{\cF} $ and $\widetilde{\cF} $ the kernels of the maps $\overline{\theta} :T_{\overline{\cM}^0}\to\cL $ and $\theta\circ d\overline{q} :T_{\overline{\cU}\hs^0}\to \overline{p}^*\cL $, respectively. We want to prove that the morphism $d\overline{\theta} :\overline{\cF} \to\overline{\cF}^\vee\otimes\cL $ induced by the Lie bracket is an isomorphism.

The statement is local so, given a point $[r]\in\overline{\cM}\hs^0$, we may consider its inverse image $r=\overline{p}^{-1}([r])$ and
(after eventually shrinking $\overline{\cM}\hs^0$ to a smaller neighborhood of $[r]$) assume that $\cL $ is trivial. Then, since in $\overline{\cU}\hs^0$ the morphisms $\overline{p} $ and $\overline{q} $ are submersions, the morphisms induced by the Lie bracket of the distributions $\cF$, $\widetilde{\cF} $ and $\overline{\cF} $ fit, as the vertical maps, in the following commutative diagram:
$$
\xymatrix@=35pt{\overline{p}^*(\cF)\ar[d]&\widetilde{\cF} \ar[l]_{d\overline{p} }\ar[r]^{d\overline{q} }\ar[d]
&\overline{q}^*(\overline{\cF} )\ar[d]^{d\overline{\theta} }\\
\overline{p}^*(\cF)^\vee\ar[r]^{d\overline{p}^t }&\widetilde{\cF}^\vee&\overline{q}^*(\overline{\cF} )^\vee\ar[l]_{d\overline{q}^t }}
$$
Note that, being $\overline{p} $ locally trivial (so that, locally, any vector field $v$ on $\overline{\cM}\hs^0$ determines uniquely a vector bundle $v'$ on $\overline{\cU}\hs^0$ constant on fibers such that $d\overline{p} (v')=v$, and this correspondence preserves the Lie bracket) $d\overline{\theta} $ is an isomorphism at $[r]$ if and only if the corresponding morphism $(\widetilde{\cF} )_{|r}\to (\widetilde{\cF} )^\vee_{|r}$ has rank $\dim(\overline{\cM} )=2m-4$. Noting that $(\widetilde{\cF} )_{|r}\cong E(2,0^{2m-4})$,
the statement follows from the usual description of ${d\overline{p} }$ in terms of the evaluation of global sections of $(T_{\cX})_{|r}$.
\end{proof}

\subsection{The $1$-ample case}\label{ssec:1ample}

One of the morals of Example \ref{ex:isograss}, or more generally of the examples provided by the study of nilpotent orbits and their crepant resolutions, is that the exceptional locus of $\eps$ may be quite involved. The simplest notion that measures the intricacy of this locus is $k$-ampleness: we say that $T_X$ is {\it $k$-ample} if the dimension of every component of a fiber of $\eps$ is at most $k$-dimensional. The goal of this section is to show how the techniques described above may help us to prove Conjecture \ref{conj:CPconj} in the simplest nontrivial case, that is $k=1$. We refer the reader to \cite{1-ample} for details.

\begin{theorem}\label{thm:1ample}
Let $X$ be a CP-manifold such that $T_X$ is big and $1$-ample. Then $X$ is rational homogeneous.
\end{theorem}

More concretely, looking at the list of rational homogenous spaces one sees that, beside the projective space (that will not appear in our discussion since we or initial assumptions imply that $T_X$ is not ample), the tangent bundle is only $1$-ample for the smooth quadric of dimension $\geq 3$, and for the complete flag manifolds of type ${\rm A}_1\times {\rm A}_1$ and ${\rm A}_2$.

We start by observing that, since $k$-ampleness decreases with contractions, the case in which the Picard number $n$ of $X$ is bigger than one essentially reduces, via Mori's proof of Hartshorne's Conjecture,  to the study of $\P^1$-fibrations over a projective space.

\begin{lemma}\label{lem:1amp-pic>1}
Let $X$ be a CP-manifold and assume that $T_X$ is big and $k$-ample, and let $\pi:X\to X'$ be a Mori contraction. Then $T_{X'}$ is $(k-\dim X+\dim X')$-ample.
\end{lemma}

\begin{proof}
By Theorem \ref{thm:smooth} the morphism $\pi$ is smooth, so the morphism $d\pi:T_X\to\pi^*T_{X'}$ provides an inclusion of the fiber product $\P(\pi^*T_{X'})=\P(T_{X'})\times_{X'}X$ into $\cX=\P(T_X)$.
 Hence the Stein factorization of the restriction $\eps_{|\P(\pi^*T_{X'})}$ factors through the corresponding contraction $\eps':\P(T_{X'})\to \cY'$.
$$
 \xymatrix{&X'&\P(T_{X'})\ar[l]_{\phi'}\ar[rr]^{
 \eps'}&&\cY'\ar[d]\\
 X\ar[ru]^{\pi}&\P(\pi^*T_{X'})\,\ar@{^{(}->}[r]\ar[l]^{\phi\hspace{0.6cm}}\ar[ru]&\cX
 \ar[rr]^{\eps}&&\cY}
 $$
 Since the natural map from $\P(\pi^*T_{X'})$ to $\P(T_{X'})$ has fibers of dimension $\dim X-\dim X'$, denoting  by $k$ and $k'$ the maximal dimensions of components of fibers of $\eps$ and $\eps'$, respectively, it follows that $k\geq k'+ \dim X-\dim X'$.
\end{proof}

\begin{corollary}\label{cor:1amp-pic>1}
Let $X$ be a CP-manifold of Picard number $n>1$ such that $T_X$ is big and $1$-ample. Then $X$ is isomorphic to $G/B$ for $G$ of type ${\rm A}_1\times {\rm A}_1$ or ${\rm A}_2$.
\end{corollary}

\begin{proof}
As a consequence of Lemma \ref{lem:1amp-pic>1}, any Mori contraction $\pi:X\to X'$ must have one-dimensional fibers and its image must have ample tangent bundle. Therefore, in our situation, $X$ has at least two $\P^1$-fibrations over $\P^{\dim X-1}$. We may then finish in several ways, for instance by applying Theorem \ref{thm:pic2}.
\end{proof}

The next result shows that the assumption implies that the exceptional locus of $\eps$ is $\overline{q}(\overline{\cU})$:

\begin{lemma}\label{lem:exclocus1amp}
Let $X$ be a CP-manifold of Picard number one, with $T_X$ big, $1$-ample, and not ample. Let $\overline{f}:\P^1\to \cX$ be a general minimal section of $\cX$ over a general minimal rational curve $f:\P^1\to X$. Then
$\overline{f}^*T_{\cX}\cong E(-2,0^{2m-3},2)$ and the exceptional locus of $\eps$ is equal to $D:=\overline{q}(\overline{\cU})$.
\end{lemma}

\begin{proof}
The $1$-ampleness hypothesis implies that the differential of the evaluation map from $\overline{\cU}$ to $\cX$ is generically injective. Hence, in the notation of Proposition \ref{prop:splittype} we must have defect $e=0$ for the general $\overline{f}$, which concludes the first part of the statement. In particular, this implies that $D=\overline{q}(\overline{\cU})$ is an irreducible divisor; since the hypotheses also imply that $\eps$ is elementary, it follows that the inclusion $D\subseteq\Exc(\eps)$ is an equality.
\end{proof}

Another important consequence of $1$-ampleness is that it allows us to control the splitting type of $T_X$ on every 
curve of family $\overline{\cM}$. The next argument has been taken from \cite{Wie}.

\begin{proposition}\label{prop:smoothdeform}
Let $X$ be a CP-manifold of Picard number one, with $T_X$ big, $1$-ample, and not ample. Being  $\overline{f}:\P^1\to \cX$ be the normalization of any curve $\overline{\Gamma}$ of $\overline{\cM}$, we have:
\begin{equation}\label{eq:split1amp}
\overline{f}^*T_\cX\cong E\big(-2,2,(-1)^{e},1^{e},0^{2m-2e-3}\big), \mbox{ for some } e\geq 0.
\end{equation}
In particular, the variety $\overline{\cM}$ is smooth.
\end{proposition}

\begin{proof}

We will prove the isomorphism (\ref{eq:split1amp}), from which the smoothness of $\overline{\cM}$ at $[f]$ follows as in Proposition \ref{prop:splittype}.

Let $\cJ$ be the ideal sheaf of the curve $\overline{\Gamma}$ in $\cX$. By Proposition \ref{prop:treeP1}, the curve $\overline{\Gamma}$ is smooth. Moreover, since $R^i\eps_*\cO_{\cX}=0$ for $i>0$, pushing forward the short sequence
$$
0\to\cJ^2\longrightarrow \cO_{\cX}\longrightarrow \cO_{\cX}/\cJ^2\to 0,
$$
we obtain an isomorphism $H^1(\overline{\Gamma},\cO_{\cX}/\cJ^2)=R^1\eps_*\cO_{\cX}/\cJ^2\otimes\cO_P\cong R^2\eps_*\cJ^2\otimes\cO_P$ (where $P=\eps(\overline{\Gamma})$), and the latter is zero because the fibers of $\eps$ are at most one-dimensional.
Then, considering now the exact sequence
$$
0\to\cN^\vee_{\overline{\Gamma}/\cX}\longrightarrow \cO_{\cX}/\cJ^2\longrightarrow \cO_{\overline{\Gamma}}\to 0,
$$
\noindent we obtain $H^1(\overline{\Gamma},\cN^\vee_{\overline{\Gamma}/\cX})=0$.
Equivalently, the splitting type of the normal bundle $\cN_{\overline{\Gamma}/\cX}$ does not contain any integer bigger than $1$ and, considering the commutative diagram:
\begin{equation}\label{eq:contactcurves}
\xymatrix{T_{\overline{\Gamma}}  \ar@{>->}[]+<2.5ex,0ex>;[r]                 \ar@{=}[d]    &\overline{f}^*\cF\ar@{->>}[r]\ar@{>->}[]+<0ex, -2.5ex>;[d]  &\overline{f}^*\cF/T_{\overline{\Gamma}}\ar@{>->}[]+<0ex, -2.5ex>;[d] \\
          T_{\overline{\Gamma}}\ar@{>->}[]+<2.5ex,0ex>;[r]  &{\overline{f}^*T_{\cX}}\ar@{->>}[r]\ar@{->>}[d]&\cN_{\overline{\Gamma}/\cX}\ar@{->>}[d]\\
          &\cO_{\overline{\Gamma}}\ar@{=}[r]&\cO_{\overline{\Gamma}}}
\end{equation}
the same property holds for the bundle $\overline{f}^*\cF/T_{\overline{\Gamma}}$ and, in particular, $\overline{f}^*\cF\cong T_{\overline{\Gamma}} \oplus \overline{f}^*\cF/T_{\overline{\Gamma}}$. Combining this with the contact isomorphism $\overline{f}^*\cF\cong \overline{f}^*\cF^\vee$, we easily see that $\overline{f}^*\cF\cong E(-2,2,(-1)^e,1^e,0^{2m-2e-4})$, for some $e$. We finish the proof by arguing as in the last part of the proof of Proposition \ref{prop:splittype}.
\end{proof}

\begin{corollary}\label{cor:smoothdeform}
With the same notation as in Proposition \ref{prop:smoothdeform}, for every component $\overline{\Gamma}$ of a fiber of $\eps$, its defect $e$ is equal to zero. In particular the twisted $1$-form $\overline{\theta}$ defined in Lemma \ref{lem:liftcontact} is a contact form on $\overline{\cM}$. 
\end{corollary}

\begin{proof}
Note first that being $\overline{\cM}$ smooth by \ref{prop:smoothdeform}, the universal family $\overline{\cU}\to \overline{\cM}$ is also smooth, and then the standard interpretation of the differential of $\overline{q}$ in terms of the evaluation of global sections tells us that the dimension of the kernel of $d\overline{q}$ at a point $P\in \overline{\cU}$ equals the defect $e$ of the corresponding curve $\overline{p}(P)$, which is zero for general $\overline{p}(P)$.

Therefore, denoting by $\Sigma\subset\overline{\cM}$ the set of elements in which $e>0$, its inverse image $\cU_\Sigma:=\overline{p}^{-1}(\Sigma)$ is the ramification locus of $\overline{q}$. In particular $\Sigma\subset\overline{\cM}$ is a divisor and, being $\overline{q}$ finite (by the $1$-ampleness hypothesis), \cite[III.10.6]{Ha} tells us that $e=1$ for the general element $[\overline{f}]$ of $\Sigma$. Moreover, for this element we have a commutative diagram with exact rows:
$$\xymatrix@=25pt{
T_{\P^1}\ar@{=}[d]\ar@{>->}[]+<2.5ex,0ex>;[r] &(T_{\overline{p}^{-1}(\Sigma)})_{|\overline{p}^{-1}([\overline{f}])}\ar[d]\ar@{->>}[r]&
E(0^{\oplus(2m-4)})\ar[d]\\
T_{\P^1}\ar@{=}[d]\ar@{>->}[]+<2.5ex,0ex>;[r] &(T_{\overline{\cU}})_{|\overline{p}^{-1}([\overline{f}])}\ar[d]^{d\overline{q}}\ar@{->>}[r]&
E(0^{\oplus(2m-3)})\ar[d]^{ev}\\
T_{\P^1}\ar@{>->}[]+<2.5ex,0ex>;[r] &\overline{f}^*T_{\cX}\ar@{->>}[r]&
E(1,0^{2m-5},-1,-2) }$$

The composition of the right-hand-side vertical arrows
is generically of rank
$2m-4$ except at the point $P=\supp(\coker(E(0^{\oplus(2m-4)})\to
E(1,0^{2m-5})))$, where the rank drops to $2m-5$. Since the map
$\eps$ contracts $\overline{f}(\P^1)$ it follows that at $P$ the map $d(\eps\circ \overline{q})$ has
rank $\leq 2m-5$. Let us define $\widehat\Sigma\subset\cU_\Sigma$ as
the locus of points where $\rk(d(\eps\circ \overline{q}))\leq 2m-5$. Then
$\widehat\Sigma$ dominates $\Sigma$ via $\overline{p}$ and, in fact, the map
$\widehat\Sigma\to\Sigma$ is generically one-to-one. Since points in
$\overline{\cM}$
parametrize components of fibers of $\eps$
it follows that $(\eps\circ \overline{q})_{|\widehat\Sigma}$ is generically
finite-to-one, so that $\dim(\eps(\overline{q}(\widehat\Sigma)))=2m-4$ which
contradicts \cite[III.10.6]{Ha}.

The second part of the statement follows then as in Corollary \ref{cor:1formcontact}.
\end{proof}

Once we know that $\overline{\cM}$ is a contact manifold, we study the morphism $\overline{\phi}:\overline{\cM}\to\cM$ in order to determine the splitting type of $T_X$ on minimal rational curves.

\begin{lemma}\label{lem:1-ampiso}
Let $X$ be a CP-manifold of Picard number one, with $T_X$ big, $1$-ample and not ample. Then the natural map $\overline{\phi}:\overline{\cM}\to\cM$ is an isomorphism and $f^*T_X\cong E(2,1^{m-2},0)$ for every $[f]\in\cM$.
\end{lemma}

\begin{proof}
We already now that $\cM$ and $\overline{\cM}$ are smooth, by Propositions \ref{prop:RCbasic} and \ref{prop:smoothdeform}, and that the general fiber of $\overline{\phi}$ is a projective space of dimension $m-c-2$, with $c:=-K_X\cdot\Gamma-2$ for $\Gamma\in\cM$.

We claim first that $m-c-2=0$. In fact, if this were not the case, $\overline{\phi}$ would be a Mori contraction of the contact manifold $\overline{\cM}$. By Theorem \ref{thm:KPSW} it would follow that $\overline{\cM}\cong \P(T_{\cM})$ and, in particular, $\dim(\cM)=m-1$. Together with the nefness of $T_X$, this implies that $f^*(T_X)\cong E(2,0^{m-1})$ for all $[f]\in\cM$, so the differential of $q:\cU\to X$ would everywhere injective. But $X$ is simply connected, hence $q$ would be an isomophism, contradicting that $X$ has Picard number one.

Then $\overline{\phi}$ is birational so, if it were not an isomorphism, being $\cM$ is smooth, it would factor via a Mori contraction, contradicting \ref{thm:KPSW}. This concludes the first part of the statement. For the second, note that, being $\overline{\phi}$ an isomorphism, the number of zeroes appearing in the splitting type of $f^*T_X$ for any $[f]\in\cM$ is equal to one. Looking at the general element, which is standard, we obtain that $-K_X\cdot f(\P^1)=m$ for every $[f]\in\cM$. Then for any $[f]$, the splitting type of $f^*T_X$ contains no negative elements, an integer $\geq 2$, and at most one zero: hence the only possibility is  $f^*T_X\cong E(2,1^{m-2},0)$.
\end{proof}

At this point there are several ways to finish the proof of Theorem \ref{thm:1ample} (Cf. Theorem~\ref{them:pibig} and \cite{CS}).
We will sketch here the proof presented in \cite{1-ample}, and refer to the original paper for details.

\begin{proposition}\label{prop:quadrics}
Let $X$ be a CP-manifold of Picard number one, and assume that $T_X$ is big, $1$-ample and not ample. Then $X$ is a smooth quadric hypersurface.
\end{proposition}

\begin{proof}
Note that $\Pic(\cX)\cong\phi^*\Pic(X)\oplus \Z \cO(1)$. Let $D=\overline{q}(\overline{\cU})$ be the exceptional divisor of $\eps$ (see \ref{lem:exclocus1amp}), $L$ be a divisor associated to the tautological line bundle $\cO(1)$ on $\cX=\P(T_X)$, and write $D=aL-\phi^*B$, for some divisor $B$ on $X$. Note that, at every point $x\in X$, the set $D_x=\phi^{-1}(x)\cap D$ is the dual variety of the VMRT $\cC_x$ (see \ref{prop:dual}). Since $\cC_x$ is a hypersurface, then its dual $E_x$ cannot be a hyperplane in $\P(T_{X,x})$ (otherwise $\cC_x$ would be a point), and we may write $a>1$.

By Proposition \ref{prop:treeP1}, every positive dimensional fiber of $\eps$ is either $\P^1$ or a union of two $\P^1$'s meeting at a point.
In the second case, the intersection of each component $\overline{\Gamma}$ with the exceptional divisor $D$ is $-1$. Hence, since $L\cdot\overline{\Gamma}=0$, we have $B\cdot\Gamma=1$. It follows that $B$ is the ample generator of $\Pic(X)$ and $-K_X=mB$, so that $X$ is necessarily a smooth quadric by the Kobayashi-Ochiai Theorem (\cite[V.1.11]{kollar}).

If every positive dimensional fiber of $\eps$ is irreducible, then $\overline{q}:\overline{\cU}\to D$ is a bijective immersion, hence an isomorphism. Since moreover the family $\cU\to\cM$ is isomorphic to $\overline{\cU}\to \overline{\cM}$, by Lemma \ref{lem:1-ampiso}, it allows to identify the restriction $\phi_{|D}:D\to X$ with the evaluation morphism $q:\cU\to X$.

In particular, $\phi_{|D}$ is smooth by Proposition \ref{prop:RCbasic}, and we have an exact sequence
$$
0\to \cO_D(-D)\longrightarrow (\Omega_{\cX/X})_{|D}\longrightarrow T_{D/X}\to 0.
$$
In particular we might have $c_{m-1}(\Omega_{\cX/X}\otimes \cO_D(D))=0$. The computation of this Chern class (see \cite[Lemma~3.5]{1-ample}) tells us that $D$ must be numerically equivalent to the $\Q$-divisor $aL+\frac{a}{m}\phi^*K_X$. Since $D\cdot \overline{\Gamma}=-2$ and $\Gamma\cdot K_X=-m$, we must have $a=2$, so that $D$ defines a nowhere degenerate symmetric form in $H^0(X,S^2T_X\otimes\cO(\frac{2}{m}K_X))$. We may now conclude that $X$ is a quadric, either by \cite[Theorem~2]{Wis2}, or by \cite{Ye}, or by noting that in this case the VMRT $\cC_x$ is necessarily a smooth quadric for every $x$, hence the result follows from Theorem \ref{them:HH}.
\end{proof}


\section{Flag type manifolds}\label{sec:FTman}


In  \cite{SSW} the following strategy has been proposed  to attack  Conjecture \ref{conj:CPconj}:

\begin{enumerate}
\item[(A)] Prove  Conjecture \ref{conj:CPconj} for CP-manifolds of ``maximal'' Picard number.
\item[(B)] Prove that any CP-manifold is dominated by one of such manifolds.
\end{enumerate}

Conjecture \ref{conj:CPconj} will then follow from  \cite[Main Theorem]{Lau}, which asserts that a manifold dominated by a rational homogeneous manifold is itself rational homogeneous.

In order to give a reasonable notion of maximality
let us recall that, given a semisimple Lie group $G$ and a parabolic subgroup $P$, taking $B$ to be a Borel subgroup containing $P$ we have a contraction $f:G/B \to G/P$, so the {\it complete flag manifold} $G/B$
dominates every $G$-homogeneous variety.

Complete flag manifolds associated with semisimple Lie groups can be recognized, among homogeneous manifolds, by the structure of their Mori contractions: all their elementary contractions are $\P^1$-bundles. This suggests the following definition, in which the assumption on the elementary contractions has been replaced by the milder one that such morphisms are smooth $\P^1$-fibrations.

\begin{definition}\label{def:FT}
A {\it Flag type manifold} (FT-manifold for short) is a CP-manifold  whose elementary contractions are smooth $\P^1$-fibrations.
\end{definition}

The choice of FT-manifolds as candidates for ``maximal'' CP-manifolds, is also sustained by the following result (cf. \cite[Proposition~5]{MOSW}):

\begin{proposition}\label{prop:init}
Let $M$ be a CP-manifold
admitting a contraction $f : M\to X$ onto an FT-manifold $X$. Then there exists a
smooth variety $Y$ such that $M\cong X \times Y$.
\end{proposition}

Furthermore, we introduce the width of a CP-manifold $X$ as a measure of how far $X$ is from being an FT-manifold.

\begin{definition}
Given a CP-manifold $X$ of Picard number $n$, with the notation introduced in \ref{not:cpmanifold}, we define its {\it width} as the non negative integer
$$\tau(X):=\sum_{i=1}^n
(-K_X\cdot \Gamma_i-2).$$
\end{definition}

\begin{remark}\label{cor:CP2}
A CP-manifold $X$ is an FT-manifold if and only if $\tau(X)=0$.
\end{remark}

\begin{proof}
Assume that $\tau(X)=0$, and let $\pi_i:X\to X_i$ be an elementary contraction, associated with an extremal ray $R_i$, generated by the class of a minimal rational curve $\Gamma_i$. Let $p_i:\cU_i\to \cM_i$ be the family of deformations of $\Gamma_i$, with evaluation morphism $q_i$. By Proposition \ref{prop:RCbasic} (5) the map $q_i:\cU_i\to X$ is an isomorphism, so $\pi_i$ is a $\P^1$-fibration.
\end{proof}

Part (A) of the strategy presented above has been recently completed in \cite{OSWW}, where the following more general
result has been proved:

\begin{theorem}\label{thm:main}
Let $X$ be a Fano manifold whose elementary contractions are
smooth $\P^1$-fibrations. Then
$X$ is isomorphic to a complete flag manifold $G/B$, where $G$ is a
semisimple algebraic group and $B$ is a Borel subgroup.
\end{theorem}

Then  Conjecture \ref{conj:CPconj} will follow if one can prove the following property, that holds trivially for rational homogeneous manifolds (see Section \ref{ssec:rathom}):

\begin{conjecture}\label{conj:CP1} Let $X$ be a CP-manifold which is not a product of positive-dimensional varieties. If $\tau(X)>0$, then there exists a surjective morphism $f:X'\to X$ from a CP-manifold $X'$, which is not a product of positive-dimensional varieties, such that $\tau(X')<\tau(X)$.
\end{conjecture}

In this section, following \cite{MOSW} and \cite{OSWW} we will describe the main ideas involved in the proof of Theorem \ref{thm:main}. Let us first of all, fix the notation.

\begin{notation}\label{not:ftmanifold}
Along the rest of the section we will use the notation introduced in \ref{not:cpmanifold}, noting that, in the case of FT-manifolds, the families $p_i:\cU_i\to \cM_i$ coincide with the $\P^1$-fibrations $\pi_i:X\to X_i$.
Moreover, given any subset  $I\subset D:=\{1, \dots, n\}$, the rays $R_i$, $i \in I$, span an extremal face (Cf. Proposition \ref{prop:simplicial}) that we will denote by $R_I$. We will denote by $\pi_I:X \to X_I$ the corresponding extremal contraction, by $T_I$ the relative tangent bundle, and by $K_I:= -\det T_I$ the relative canonical divisor. Alternatively, we will denote by $\pi^I:X\to X^I$ the contraction of the face $R^I$ spanned by the rays $R_i$ such that $i \in D \setminus I$.
For $I \subset J \subset D $ we will denote the
contraction of the extremal face $\pi_{I*}(R_J) \subset \Nu(X_I)$ by $\pi_{I,J}:X_I \to X_J$ or by
$\pi^{D \setminus I,D \setminus J}:X^{D \setminus I} \to X^{D\setminus J}$.
\end{notation}

\subsection{Bott-Samelson varieties}

We will now introduce some auxiliary manifolds, which we call
Bott-Samelson varieties, by analogy with the Bott-Samelson manifolds
that appear classically in the study of Schubert cycles of homogeneous manifolds.

Following \cite{LT}, when dealing with finite sequences of indices in $D=\{1,\dots,n\}$ we will use the following notation:

\begin{notation}\label{notn:LT} Given a sequence $\ell=(l_1,\dots,l_r)$, $l_i\in D$, we set, for any 
$0\leq s\leq r-1$, $\ell[s]:=(l_1,\dots,l_{r-s})$, and $\ell[r]=\emptyset$.
In particular $\ell[s][s']=\ell[s+s']$.
\end{notation}

With every sequence $\ell=(l_1,\dots,l_r)$ of elements of $D$ we will associate a sequence of smooth varieties $\cZ_{\ell[s]}$, $s=0,\dots,r$, called the {\it Bott-Samelson varieties of $X$ associated with $\ell$},
together with morphisms
$$f_{\ell[s]}:\cZ_{\ell[s]} \to X,\quad p_{\ell[s+1]}:\cZ_{\ell[s]}\to \cZ_{\ell[s+1]},\quad \sigma_{\ell[s+1]}:\cZ_{\ell[s+1]}\to \cZ_{\ell[s]}.$$
They are constructed in the following way: for $s=r$ we set $\cZ_{\ell[r]}:=X$ and $f_{\ell[r]}=\id$. Then for $s<r$ we define $\cZ_{\ell[s]}$ recursively
 by considering the composition
$g_{\ell[s+1]}:=\pi_{l_{r-s}}\circ f_{\ell[s+1]}:\cZ_{\ell[s+1]}\to X_{l_{r-s}}$ and taking its fiber product with $\pi_{l_{r-s}}$:
$$
\xymatrix@=35pt{\cZ_{\ell[s]}\ar[r]^{f_{\ell[s]}}\ar[d]^{p_{\ell[s+1]}}&X\ar[d]^{\pi_{l_{r-s}}} \\
\cZ_{\ell[s+1]}\ar@/^/[u]^{\sigma_{\ell[s+1]}} \ar[ru]_{f_{\ell[s+1]}}\ar[r]^{g_{\ell[s+1]}}&X_{l_{r-s}}}
$$
Note that   $p_{\ell[s+1]}$ is a $\P^1$-bundle with a section $\sigma_{\ell[s+1]}$; more precisely
it is the projectivization of an extension $\cF_{\ell[s]}$ of $\cO_{\cZ_{\ell[s+1]}}$ by $f_{\ell[s+1]}^*K_{l_{r-s}}$:
\begin{equation}\label{eq:F}
0 \lra f_{\ell[s+1]}^*K_{l_{r-s}}\longrightarrow
\cF_{\ell[s]} \longrightarrow \cO_{\cZ_{\ell[s+1]}}\lra 0.
\end{equation}

Now let $\phi:Y \to X$ be a morphism. The above construction may be lifted to $Y$ via $\phi$, and the resulting varieties $\cZ_{\ell[s]}(Y)=\cZ_{\ell[s]}\times_X Y$ are called the {\it Bott-Samelson varieties of $X$ associated with $\phi:Y \to X$ and $\ell$}. By abuse of notation,
the projections and sections of these varieties will be denoted also by $p_{\ell[s+1]}$, and $\sigma_{\ell[s+1]}$, respectively.
The construction is obviously functorial, so that, given a morphism $g:Y'\to Y$, $\phi'=\phi\circ g$, we have a commutative diagram, in which every square is a fiber product, of the following form:

$$
\xymatrix@=35pt{\cZ_{\ell[s]}(Y')\ar[r]\ar[d]^{p_{\ell[s+1]}}&\cZ_{\ell[s]}(Y)\ar[r]\ar[d]^{p_{\ell[s+1]}}&\cZ_{\ell[s]}\ar[r]^{f_{\ell[s]}}\ar[d]^{p_{\ell[s+1]}}&X\ar[d]^{\pi_{l_{r-s}}} \\
\cZ_{\ell[s]}(Y')\ar[r]\ar@/^/[u]^{\sigma_{\ell[s+1]}}&\cZ_{\ell[s]}(Y)\ar[r]\ar@/^/[u]^{\sigma_{\ell[s+1]}}&\cZ_{\ell[s+1]}\ar@/^/[u]^{\sigma_{\ell[s+1]}} \ar[ru]_{f_{\ell[s+1]}}\ar[r]^{g_{\ell[s+1]}}&X_{l_{r-s}}}
$$

Later on we will consider mostly the case in which $Y$ is a point of $X$ and $\phi$ is the inclusion. For simplicity, in this case, we will denote the corresponding Bott-Samelson varieties by $Z_{\ell[s]}:=\cZ_{\ell[s]}$. Especially important will be the case in which $\ell$ is {\em maximal reduced sequence}, i.e. satisfying that $r=m:=\dim X$ and $\dim f_{\ell[s]}(Z_{\ell[s]}) = m-s$ for every $s$ (since $X$ is rationally chain connected by curves $\Gamma_j$, it is always possible to find a sequence of this kind).

Let us denote by $\beta_{i(r-i)}$ the class in $N_1(Z_{\ell[r-i]})$ of the fibers of $p_{\ell[r-i+1]}:Z_{\ell[r-i]} \to Z_{\ell[r-i+1]}$. We  will denote by $\beta_{i(s)}$ the image of this class into $N_1(Z_{\ell[s]})$, via push forward with the sections $\sigma_{\ell[r-j]}$, $j= i, \dots, r-s-1$. If $s=0$ we will write $\beta_{i}$ instead of $\beta_{i(0)}$. Note that, by construction, $f_{\ell[s]*}\beta_{i(s)}=[\Gamma_{l_i}]$.

Clearly the $\beta_{i}$'s, $i=1,\dots, r$, form a basis of  of $N_1(Z_{\ell})$.
Within $N^1(Z_\ell)$ we  consider the dual basis of $\{\beta_i,\,i=1,\dots,r\}$, denoted by $$\{H_i,\,i=1,\dots,r\}.$$
Let us also define for every $t\leq r$, the following line bundles on $Z_\ell$:
\begin{equation*}
N_{t}=\sum_{\substack{i\leq t, \,l_i=l_t}}H_{i}.\label{eq:Ns}
\end{equation*}
and the classes $\gamma_i\in N_1(Z_{\ell})$, $i=1,\dots,r$, defined by $N_{t}\cdot\gamma_{i}=\delta_{i}^{t}$.

The next Proposition summarizes some of the properties of $\NE(Z_\ell)$, as proved in \cite[Section 3]{OSWW}.

\begin{proposition}[{\cite[Corollary~3.9, 3.10]{OSWW}}]\label{prop:coneBS} With the same notation as above:

\begin{enumerate}
\item The Mori cone (respectively, the nef cone) of $Z_\ell$ is the simplicial cone generated by the classes $\gamma_{t}$ (resp.  $N_{t}$), $t\leq r$.
\item Setting $J=\{i\,\, |\,\,l_i=l_k\mbox{ for some } k>i\}$ then the Stein factorization of the map $f_\ell:Z_\ell\to X$ is the contraction associated with the extremal face of $\NE{(Z_\ell)}$ generated by $\{\gamma_i\,\,|\,\, i  \in J\}$.
\end{enumerate}
\end{proposition}

\subsection{Cartan matrix of an FT-manifold}

It is a well known fact that every semisimple Lie algebra $\fg$ is determined by its Cartan matrix (equivalently, by its Dynkin diagram, see \cite{Hum}), and that this matrix can be seen as the intersection matrix of relative anticanonical divisors and fibers of the $\P^1$-fibrations of the corresponding complete flag manifold $G/B$. This suggests the following:

\begin{definition}
Let $X$ be an FT-manifold of Picard number $n$. With the same notation as above,
the {\it Cartan matrix} of $X$ is the $n \times n$ matrix $M(X)$
defined by $M(X)_{ij}=-K_i \cdot \Gamma_j$.
\end{definition}

We will  now sketch the proof of (a reformulation of) the main result of \cite{MOSW}:

\begin{theorem}\label{thm:mainmosw}
The Cartan matrix of an FT-manifold is of finite type, i.e. is the Cartan matrix of a semisimple Lie algebra.
\end{theorem}

Let us denote by $M_I(X)$ the $|I| \times |I|$ principal submatrix of $M(X)$ obtained from $M(X)$ by subtracting rows and columns corresponding to indices which are not in $I$.
It is easy to show that the fibers of a contraction
of an FT-manifold $X$  are FT-manifolds, whose Cartan matrices
are principal submatrices of
$M(X)$:

\begin{proposition}\cite[Proposition~6]{MOSW}\label{prop:fibers} Let $X$ be an $FT$-manifold, $I\subset D$ any nonempty subset, and let $\pi_I:X \to X_I$ be the contraction of the corresponding face $R_I$. Then every fiber of $\pi_I$ is an FT-manifold
whose Cartan matrix is $M_I(X)$.
\end{proposition}

In particular, it follows from Theorem \ref{thm:pic2}
 that
any $2 \times 2$ principal submatrix is the Cartan matrix of an FT-manifold of Picard number $2$. These are, up to transposition
\begin{equation}\begin{array}{cccc}\vspace{0.2cm}\label{eq:2x2}
M({\rm A}_1 \times {\rm A}_1) & M({\rm A}_2) & M({\rm B}_2) &M({\rm G}_2)\\
\left(\begin{matrix}
2&0\\0&2
\end{matrix}\right), &
\left(\begin{matrix}
2&-1\\-1&2
\end{matrix}\right),&
\left(\begin{matrix}
2&-1\\-2&2
\end{matrix}\right),&
\left(\begin{matrix}
2&-1\\-3&2
\end{matrix}\right)\vspace{0.2cm}
\end{array}
\end{equation}

We may then conclude that $M(X)$ is a {\it generalized Cartan matrix} in the sense of \cite[4.0]{Kac}.

We say that the matrix $M(X)$ is {\it decomposable}  if there exists two nonempty complementary subsets $I,J\subset D$ such that $M(X)_{ij}=0$ for all $(i,j)\in (I\times J)\cup (J\times I)$. In this case, by abuse of notation, we will write $M(X)=M_I\times M_J$. The next statement
allows us to reduce the proof of Theorem \ref{thm:mainmosw} to the case in which $M(X)$ is indecomposable:

\begin{proposition}[{\cite[Proposition~7]{MOSW}}]\label{prop:decomp}
With the same notation as above, if $M(X)$ is decomposable as $ M_I\times M_J$ then $X \simeq X_I \times X_J$, where $X_I$ and $X_J$ are FT-manifolds whose Cartan matrices are $M_I$ and $M_J$, respectively.
\end{proposition}

\begin{proof}
One may verify that the hypotheses imply that $X_I$ is an FT-manifold. Then the result follows easily from Proposition \ref{prop:init}. See \cite[Proposition~7]{MOSW} for details.
\end{proof}

\begin{proof}[Proof of Theorem \ref{thm:mainmosw}]
The proof is by induction on the Picard number $n$ of $X$. The result
is true for $n=2$ by Theorem \ref{thm:pic2}, hence we may assume that $n\geq 3$ and that the statement holds for FT-manifolds of Picard number $\le n-1$.

By induction every principal submatrix of $M(X)$ is of finite type, hence,
by \cite[Proposition~4.7]{Kac},  $M(X)$ is either of finite  or of {\it affine type}. In the latter case, by \cite[Corollary~4.3]{Kac},
there exists  a linear combination $\Gamma=\sum_1^n m_i \Gamma_i$, $m_i\in\R_{>0}$ for all $i$, satisfying that $K_i\cdot \Gamma=0$. Since
$M(X)$ has integral coefficients,
we may assume that $m_i\in \Q_{>0}$ for all $i$ and, clearing denominators, we may also assume that $m_i\in\Z_{>0}$ for all $i$. Finally, since the $\Gamma_i$'s are free curves the cycle $\sum_1^n m_i \Gamma_i$ is smoothable (cf. \cite[II.7.6]{kollar}), and therefore it is numerically equivalent to an irreducible rational curve $\Gamma$.

Let $f:\P^1\to \Gamma\subset X$ be the normalization morphism, and $p\in\P^1$ be any point. Consider a maximal reduced sequence $\ell=(l_1, \dots, l_m)$ of $X$ for the point $f(p)$, and the corresponding Bott-Samelson varieties $Z_{\ell[s]}$ associated to $\ell$ and $f(p)$. At the same time, consider the associated Bott-Samelson varieties ${Z'}_{\!\!\ell[s]}:=\cZ_{\ell[s]}(\P^1)$. We will denote by ${f'}_{\!\!\ell[s]}:{Z'}_{\!\!\ell[s]}\to X$ their evaluations.

We will derive a contradiction by showing that $\pi_{l_{m-1}} \circ {f'}_{\!\!\ell[1]}: {Z'}_{\!\!\ell[1]} \to X_{l_{m-1}}$ cannot be of fiber type. In a nutshell, it is the finiteness of the dimension of $X$ which implies that $M(X)$ is of finite type.

We claim first that, for any $s= 0, \dots, m$ we have ${Z'}_{\!\!\ell[s]} = \P^1 \times Z_{\ell[s]}$.
The variety ${Z'}_{\!\!\ell[s]}$ is the projectivization of a rank two bundle $\cE$ on ${Z'}_{\!\!\ell[s+1]}$ -- which, by induction, is isomorphic to $\P^1 \times Z_{\ell[s+1]}$ -- appearing as an extension
\begin{equation}\label{eq:ext}
0\rightarrow\cO({f'}_{\!\!\ell[s+1]}^*K_{m-s})\longrightarrow \cF'_{\ell[s]}\longrightarrow\cO_{{Z'}_{\!\ell[s]}} \rightarrow 0.
\end{equation}
Since, by the construction of the curve $\Gamma$, ${f'}_{\!\!\ell[s+1]}^*K_{m-s}$ has intersection zero with the fibers of the projection $p_2:\P^1\times Z_{\ell[s+1]}\to Z_{\ell[s+1]}$, then  $\cF'_{\ell[s]}$ is trivial on these fibers. Hence the sequence (\ref{eq:ext}) is the pullback via $p_2$ of
\begin{equation}
0\rightarrow\cO({f}_{\ell[s+1]}^*K_{m-s})\longrightarrow \cF_{\ell[s]}\longrightarrow\cO_{{Z}_{\ell[s]}} \rightarrow 0
\end{equation}
and the claim follows.

The functoriality of the construction implies that
${f'}_{\!\!\ell[s]}\circ j_{\ell[s]}= f_{\ell[s]}$, where $j_{\ell[s]}: Z_{\ell[s]} \to \P^1 \times Z_{\ell[s]}$ is the inclusion
defined by  $z \to (p,z)$.

Denote by $p_1$ and $p_2$ the projections of $\P^1 \times {Z}_{\ell[1]}$ onto the factors.
By assumption $\pi_{l_{m-1}} \circ {f'}_{\!\!\ell[1]}$ does not contract fibers of $p_2$, since their image in $X$ is numerically equivalent to $\Gamma$, and it is generically finite when restricted to fibers of $p_1$, by the choice of $\ell$.
This implies that $\pi_{l_{m-1}}\circ {f'}_{\!\!\ell[s]}$ is generically finite (see \cite[Lemma~7]{MOSW} for details),
contradicting the fact that $\dim {Z'}_{\!\!\ell[1]} = \dim X_{l_{m-1}}+1$.
\end{proof}

\subsection{Relative duality and reflection groups}

We will now present a generalization  of the previous result, taken from \cite{OSWW}, in which it is shown that Theorem \ref{thm:mainmosw} holds for Fano manifolds whose elementary contractions are smooth $\P^1$-fibrations; in other words, it avoids the assumption of the nefness of the tangent bundle of $X$. The proof of this result is based on constructing a finite group of reflections of $N^1(X)$ out of the $\P^1$-fibrations of $X$, and  using it to produce a root system, which, for rational homogeneous manifolds, is the root system of the corresponding Lie algebra. As a by-product, this construction provides a good deal of information about cohomology of line bundles on $X$, analogous to the one provided by the Borel-Weyl-Bott Theorem for semisimple Lie algebras.

The key ingredient for the construction is the following relative duality for smooth $\P^1$-fibrations:

\begin{lemma}\label{lem:leray}\cite[Lemma~2.3]{OSWW}
Let $\pi:M\to Y$ be a smooth $\P^1$-fibration over a smooth manifold $Y$,  denote by $\Gamma$ one of its fibers and by $K$ its relative canonical divisor. Then for every Cartier divisor $D$ on $M$, setting $l:=D\cdot\Gamma$ and $\sgn(\alpha):=\alpha/|\alpha|$ for $\alpha\neq 0$,  $\sgn(0):=1$, one has
\begin{eqnarray*}
H^{i}(M,\cO_M(D)) & \cong  &H^{i+\sgn(l+1)}(M,\cO_M(D+(l+1)K)), \mbox{ for every }i\in\Z.
\end{eqnarray*}
\end{lemma}

For every elementary contraction $\pi_i:X\to X$ we will consider the
linear map $r_i:N^1(X) \to N^1(X)$ given by
$$r_i(D)=D+ (D\cdot\Gamma_i)K_i, $$
which is a {\it reflection}, i.e. it is an involution that fixes the hyperlane
\begin{equation*}\label{eq:hyper}
\Gamma_i^{\perp}:=\{D\,|\,D\cdot\Gamma_i=0\}\subset N^1(X).
\end{equation*}
The counterpart of Theorem \ref{thm:mainmosw} in this  approach is the following:

\begin{theorem} The group $W\subset\Gl(N^1(X))$  generated by the reflections $\{r_i | i= 1, \dots, n\}$ is a finite group.
\end{theorem}

\begin{proof}
Consider the dual action of $W$ on the vector space $N_1(X)=N^1(X)^\vee$, defined by:
\begin{equation}\label{eq:dualact}w^\vee(C)\cdot D=C\cdot w(D),\mbox{ for all }D\in N^1(X),\,\,\,C\in N_1(X),\,\,\,w\in W.\end{equation}
This action is clearly faithful, i.e. the morphism $W\to\Gl(N_1(X))$ defined by $w\mapsto w^\vee$ is injective. Moreover the matrix of every element $r_i^\vee \in\Gl(N_1(X))$ with respect to the basis $\{\Gamma_1,\dots,\Gamma_n\}$ has integral coefficients and determinant $\pm 1$, hence the same properties hold for the matrices of any $w^\vee \in \Gl(N_1(X))$.

Consider now the function $\Chi_X:\Pic(X)\to\Z$ which assigns to a line bundle its Euler characteristic. By a theorem of Snapper, cf. \cite[Section 1, Theorem]{Kl}, given
$L_1, \dots, L_t \in \Pic(X)$ then $\Chi_X(m_1, \dots, m_t):=\Chi(X, m_1L_1+ \dots + m_tL_t)$ is a numerical polynomial in $m_1, \dots, m_t$ of degree $\le \dim X$. Via the identification of  $\Pic(X)$ with $N^1(X)_\Z$ we can thus extend $\Chi_X$ to a polynomial function $\Chi_X:N^1(X) \to \R$.

Now, the existence of $\P^1$-fibrations imposes severe restrictions to the function $\Chi_X$, via Lemma \ref{lem:leray}. In fact, denoting by $T$ the translation by $K_X/2$ in $N^1(X)$, that is $T(D):=D+K_X/2$, and setting $\Chi^T:=\Chi_X\circ T$, we may write
 $$\Chi^T(D)=-\Chi^T(r_i(D))\mbox{ for any } r_i \mbox{ and any }D \in \Pic(X).$$
 Hence for any $w \in W$ and $D \in \Pic(X)$ we have
$$\Chi^T(D)=-\Chi^T((w \circ r_i \circ w^{-1})(D)).$$

In particular $\Chi^T$ vanishes on any hyperplane of the form $w(\Gamma_i^\perp)$; this is the hyperplane fixed by $w\circ r_i \circ w^{-1}\in W$. Denoting by $Z\subset \P(N^1(X))$ the set of these hyperplanes, it follows that its cardinality is smaller than or equal to the degree of $\Chi^T$, i.e. the dimension of $X$, hence it is finite.

Therefore, to show that $W$ is finite, it is enough to consider the induced action of $W$ on $Z^n$, and show that the isotropy subgroup $W^0\subset W$ of elements of $W$ fixing the point $([\Gamma_1],\dots, [\Gamma_n])$ is finite.
If $w\in W^0$, then the matrix of $w^\vee$ with respect to the basis $\{\Gamma_1,\dots,\Gamma_n\}$ is diagonal,
hence all its diagonal coefficients are equal to $\pm 1$.
In particular the image of $W^0$ in $\Gl(N_1(X))$ is finite and, since the action of $W$ on $N_1(X)$ is faithful, $W^0$ is finite as well.
\end{proof}

It is then straightforward to show the following (see \cite[VI, \S 1, Definition 1]{Bourb}):

\begin{corollary}\label{cor:scalar}
With the same notation as above, there exists a scalar product $\langle~,~\!\rangle$ in $N_1(X)$ which is $W$-invariant such that
\begin{equation*}
-K_i\cdot C=2\dfrac{\langle \Gamma_i,C\rangle}{\langle \Gamma_i,\Gamma_i\rangle}, \mbox{ for all }i = 1, \dots, n.
\label{eq:scalar}
\end{equation*}
In particular $\{-K_i,\,i=1,\dots, n\}$ is a basis of $N^1(X)$ as a vector space over $\R$ and,  in the Euclidean space $(N^1(X),\langle~,~\!\rangle)$, the finite set 
$$\Phi:=\left\{w(-K_i)\left|w\in W,\,\,\,i=1,\dots,n\right.\right\}\subset N^1(X),$$
is a {\em root system} whose Weyl group is $W$.\end{corollary}

This fact imposes strong restrictions (Cf. \cite[Lemma~2.14]{OSWW}) on the possible  entries of the Cartan matrix $M(X)$, which, together with the fact that $-K_i \cdot \Gamma_j \le 0$ if $i \not = j$, implies that
the $2 \times 2$ principal minors of $M(X)$ are the ones appearing in (\ref{eq:2x2}). Being $\{-K_i,\,i=1,\dots, n\}$ a basis of $N^1(X)$, we finally obtain the following:
\begin{corollary}
The Cartan matrix of a Fano manifold $X$ whose elementary contractions are smooth $\P^1$-fibrations is of finite type.
\end{corollary}
We refer the reader to \cite[Sect. 2]{OSWW} for details.

\subsection{Dynkin diagrams and homogeneous models}\label{ssec:homogmod}

Let $X$ be an FT-manifold or, more generally, a Fano manifold whose elementary contractions are smooth $\P^1$-fibrations. The results on the Cartan matrix $M(X)$ seen in the previous sections allow us
 to associate with $X$ a finite Dynkin diagram $\cD(X)$:

\begin{definition}
The {\it Dynkin diagram} $\cD(X)$ of $X$ is the graph having $n:=\rho(X)$ nodes, such that the nodes in the $i$-th and $j$-th position are joined by $( -K_i \cdot \Gamma_j)( -K_j \cdot \Gamma_i)$ -- which is equal to $=0,1,2$ or $3$ -- edges. When two nodes are joined by multiple edges we write an arrow on them pointing to the node $j$ if $ -K_i \cdot \Gamma_j <  -K_j \cdot \Gamma_i$. The set of nodes of $\cD(X)$ will be identified with $D=\{1,\dots,n\}$.
\end{definition}

In particular, we may associate with $X$ a semisimple Lie group $G$, and its semisimple Lie algebra $\fg$, determined by the Dynkin diagram $\cD$.

\begin{definition}\label{notn:model}
Let $G$ be the Lie group determined by $\cD$, and let $B$ be a Borel subgroup. The complete flag manifold $G/B$ is a rational homogeneous space, which we will call the {\it rational homogeneous model of} $X$. We will use for $G/B$ a similar notation as for $X$, adding an overline to distinguish the two cases (so we will use $\overline{\pi_i}, \overline{\Gamma_i},-\overline{K_i},\dots$).
\end{definition}

We will also consider the isomorphism of vector spaces
$$
\psi:N^1(X)\to N^1(G/B),\mbox{ defined by }\psi(-K_i)=-\overline{K_i}.
$$
This isomorphism sends $-K_X$ to $-K_{G/B}$. In fact, since the Cartan matrix $M(X)$ is of finite type, it is nonsingular (cf. \cite[Theorem~4.3]{Kac}), so the coefficients of $-K_X$ with respect to the basis $\{-K_1,\dots,-K_n\}$ are determined by the intersection numbers $-K_X\cdot\Gamma_i=2$.

An important consequence of our previous construction of the Weyl group $W$ upon $\P^1$-fibrations is that, via a careful study of its action on the cohomology of divisors (see \cite[Sect. 2.4]{OSWW}), it provides the following:

\begin{proposition}\cite[Corollary~2.25]{OSWW}\label{prop:cohomequal}
Let $X$ be a Fano manifold whose elementary  contractions are smooth $\P^1$-fibrations. With the same notation as above, for every line bundle $L$ belonging to the subgroup of $Pic(X)$ generated by $K_1, \dots, K_n$
 $$h^i(X,L)=h^i(G/B,\psi(L))\quad i\in\Z.$$
In particular the dimension of $X$ equals the dimension of its homogeneous model.
\end{proposition}

Note that with any sequence $\ell=(l_1,\dots,l_r)$ of elements of $D$ one can associate the element of the Weyl group $W$
defined as  $w(\ell):=r_{l_1}\circ r_{l_2}\circ\dots\circ r_{l_r}$. We say that $\ell$ is {\it reduced} if $w(\ell)\in W$ cannot be written as a composition of a smaller number of reflections $r_i$, and in this case the cardinal $|\ell|$ is called the {\it length of} $w(\ell)$. It is known that for every Weyl group $W$ there is a unique element of maximal length, named the {\it longest element of} $W$.
One may prove that the evaluation morphism $f_\ell$ is generically finite if and only if it $\ell$ is reduced, and that for reduced $\ell$, $f_\ell$ is surjective if and only if $w(\ell)$ is the longest element of $W$  (see \cite[Section 3]{OSWW} for details). Furthermore, Proposition \ref{prop:cohomequal} implies the following:

 \begin{proposition}\label{prop:long} Let $X$ be a Fano manifold whose elementary  contractions are smooth $\P^1$-fibrations. If $\ell=(l_1,\dots,l_m)$ is a sequence such that $w(\ell)$ is a reduced expression of the longest element in $W$ then the morphism $f_\ell: Z_\ell \to X$ is surjective and birational.
\end{proposition}

\begin{proof}
Let $L$ be an ample line bundle on $X$ and $\overline L=\psi(L)$ be the corresponding ample line bundle on $G/B$. The result is well-known for complete flag manifolds $G/B$, and it is equivalent to say that the restriction morphism $H^0(G/B,t \overline  L) \to H^0({\overline Z}_\ell ,t{\overline f}_{\ell}^*\overline L)$ is an isomorphism for $t>>0$. By Proposition \ref{prop:cohomequal}, the result follows by showing that $$H^0({\overline Z}_\ell ,t{\overline f}_{\ell}^*\overline L)\cong H^0({Z}_\ell ,t{f}_{\ell}^* L).$$

We note that one may prove that the Euler characteristic of any line bundle ${f}_{\ell}^* L$ on ${Z}_\ell$ is determined by $\ell$ and by the degree of $L$ with respect to the curves $\Gamma_i$, hence it is the same for $X$ and for $\overline X$. Then an application of Kawamata-Viehweg Vanishing Theorem tells us that, being $L$ nef, $H^i({Z}_\ell ,t{f}_{\ell}^* L)=0$ for $i>0$, and the claimed equality holds.
\end{proof}

In particular, Proposition \ref{prop:decomp} holds also in this more general setting, allowing us to reduce Theorem \ref{thm:main} to the case in which $\cD$ is connected:

\begin{corollary}[{\cite[Corollary~3.20]{OSWW}}]
Let $X$ be a Fano manifold whose elementary contractions are smooth $\P^1$-fibrations. Assume that $\cD = \cD_1 \sqcup \cD_2$. Then $X \simeq X_1 \times X_2$, where $X_1$ and $X_2$ are Fano manifolds whose elementary contractions are smooth $\P^1$-fibrations and whose Dynkin diagrams are $\cD_1$ and $\cD_2$, respectively.
\end{corollary}

\begin{proof}
In this situation, taking two sequences $\ell_1$ and $\ell_2$ giving reduced expressions of the longest words of the Weyl groups of $\cD_1$ and $\cD_2$, we get a sequence $ \ell_1\ell_2$ that gives a reduced expression of the longest word of $W$. Moreover, the disconnectedness of $\cD$ implies that $Z_{\ell_1\ell_2}\cong Z_{\ell_1}\times Z_{\ell_2}$. One may then check that the extremal face determining the birational morphism $f_{\ell_1\ell_2}$ is the convex hull of two extremal faces determining contractions of $Z_{\ell_1}$ and $Z_{\ell_2}$. It follows that $X$ is the product of the images of these two contractions.
\end{proof}

\subsection{Homogeneity of Flag type manifolds}\label{ssec:homogFT}

In this section we will briefly discuss how to prove Theorem \ref{thm:main}, by showing that a Fano manifold $X$ whose elementary contractions are smooth $\P^1$-fibrations is isomorphic to its homogeneous model.

\begin{remark}\label{rem:induction} If we further assume that $X$ is an FT-manifold,  we may use the smoothness of its Mori contractions $\pi^I:X\to X^I$ (Section \ref{sec:CPvar}) to design a recursive procedure to obtain the homogeneity of $X$:
\begin{itemize}
 \item[1.] Find an increasing sequence  $I_1 \subset I_2 \subset \dots \subset D$, in which $|I_{k+1}|=|I_k|+1$, and such that $\overline X{}^{I_{k+1}}$ is a complete family of lines in $\overline X{}^{I_k}$.
 \item[2.] ({\it Base case}) Prove that $X^{I_1} \simeq \overline X{}^{I_1}$.
 \item[3.] ({\it Recursion}) Show that $ X^{I_{k}} \simeq {\overline X}{}^{I_{k}}$ implies $ X^{I_{k+1}} \simeq {\overline X}{}^{I_{k+1}}$.
\end{itemize}
It is easy to construct a sequence as in Step 1 for every connected Dynkin diagram $\cD$. For instance, in the case of ${\rm F}_4$, numbering the nodes as in (\ref{eq:dynkins}), one may show that the sequence:
$$
I_1=\{1\}\subset I_2=\{1,2\}\subset I_3=\{1,2,3\}\subset I_2=\{1,2,3,4\}
$$
satisfies the requirements of Step 1. Moreover, it has been shown in \cite{MOSW} that Step 3 works if we further assume that FT-manifolds whose Dynkin diagrams are proper subdiagrams of $\cD$ are homogeneous. Thus, in principle one could use this strategy to prove Theorem \ref{thm:main} for FT-manifolds by induction on the number of nodes of $\cD$, choosing appropriately the sequence $I_k$ at every step so that it assures suitable initial isomorphism $X^{I_1} \simeq \overline X{}^{I_1}$. 

This method has been used in \cite{MOSW} to prove that the result is true for FT-manifolds with Dynkin diagram ${\rm A}_n$, by reducing the problem to a base case of the form $X_{1,n}\cong\P(T_{\P^n})$. \qed
\end{remark}

In \cite{OSWW} a different approach has been considered, in order to avoid many case by case arguments needed to achieve the isomorphisms $X^{I_1} \simeq \overline X{}^{I_1}$. The idea is to use an appropriate sequence of Bott-Samelson varieties to compare $X$ and $\overline{X}$:

\begin{proposition}\label{prop:compare}
Let $X$ be a Fano manifold  whose elementary contractions are smooth $\P^1$-fibrations. Assume that its Dynkin diagram $\cD$ is connected, different from ${\rm F}_4$ and ${\rm G}_2$. Then there exists a reduced sequence $\ell=(l_1,\dots,l_m)$ associated to the longest element of $W$ such that
\begin{equation} Z_{\ell[s]} \simeq \overline Z_{\ell[s]}\mbox{ for every } s = 0, \dots, m-1.\label{eq:BSiso}
\end{equation}
\end{proposition}

\begin{corollary}\label{cor:notF4}
In the setup of Proposition \ref{prop:compare}, the variety $X$ is rational homogeneous.
\end{corollary}

\begin{proof}
Since $f_\ell$ and ${\overline f}_\ell$ are birational by Proposition \ref{prop:long}, it is enough to compare the extremal faces defining them, which are the same by Proposition \ref{prop:coneBS} (2).
\end{proof}

The arguments leading to \ref{prop:compare} are rather technical. We refer the reader to \cite{OSWW} for details, and include here a few words about the ideas behind them:

\begin{proof}[Idea of the proof of Proposition \ref{prop:compare}]
 We show that the recursive construction of the Bott-Samelson varieties, for a suitable choice of the sequence $\ell$, is the same for $X$ and for its homogeneous model. It is then enough to find a reduced sequence of maximal length $m$ such that, for every $s$, the cocycle $\zeta_{\ell[s]}\in H^1(Z_{\ell[s+1]},f_{\ell[s+1]}^*K_{l_{m-s}})$ providing the extension (\ref{eq:F}) is uniquely determined, up to homotheties.

 By looking at the restriction of the sequence (\ref{eq:F}) to curves $\beta_{i(s+1)}$ one sees that the extension cannot be trivial if the index $l_{r-s}$ already appeared in the sequence, i.e. if $J:=\{i<r-s|\,\,l_i=l_{m-s}\}$ is not empty. Therefore, to prove the isomorphisms in (\ref{eq:BSiso}) one has to show that, for any $s = 0, \dots, m-1$
\begin{equation}\label{eq:coho}h^1(Z_{\ell[s+1]},f_{\ell[s+1]}^* K_{l_{m-s}}) = \begin{cases} 0 & J= \emptyset \\ 1 & J \not = \emptyset
\end{cases}\end{equation}
Sequences  $\ell=(l_1, \dots, l_m)$ satisfying that $w(\ell)=r_{l_1}\circ\dots\circ r_{l_m}$ is a reduced expression of the longest element of $W$, and such  that equalities in (\ref{eq:coho}) hold, have been described in \cite[Section 4]{OSWW}
for $\cD\neq {\rm G}_2,{\rm F}_4$. If $\cD$ has no multiple edges, one may check that any reduced sequence of maximal length satisfies the required property. This is not the case when $\cD$ is of type ${\rm B}$ or ${\rm C}$, but we may still choose carefully the sequence $\ell$ so that the whole process works.
\end{proof}

\begin{remark}\label{rem:espcases}
As for the special cases, the result for $\cD= {\rm G}_2$ follows from Theorem \ref{thm:pic2}, while case ${\rm F}_4$ is much more involved: using the software system {\tt Sage} it has been checked that  for none of the $2144892$ possible sequences providing a reduced expression of the longest element of $W$, the equalities (\ref{eq:coho}) hold for any $s=0, \dots, m-1=23$.
\end{remark}

Let us finally discuss the remaining case, $\cD={\rm F}_4$. The proof presented in \cite{OSWW} for this case is based on the reconstruction argument by families of lines explained in Remark \ref{rem:induction}.

\begin{proposition}
Let $X$ be a Fano manifold  whose elementary contractions are smooth $\P^1$-fibrations, and assume that its Dynkin diagram is ${\rm F}_4$. Then $X$ is rational homogeneous.
\end{proposition}

\begin{proof}[Sketch of the proof]
First of all, in order to use the recursive procedure of Remark \ref{rem:induction}, we need to prove that the contractions $\pi^I:X\to X^I$ are smooth. To do this, we prove first that the fibers of these contractions are birational images of Bott-Samelson varieties, and these images may be proved to be homogeneous because we know that Theorem \ref{thm:main} holds in the case in which the Dynkin diagram is a proper subdiagram of ${\rm F}_4$.

Then the recursive procedure described in \ref{rem:induction} allows to reduce the problem to showing that the CP-manifold $X^1$ is isomorphic to its homogeneous model $\overline X{}^1$. Moreover this may be achieved, via Theorem \ref{them:HH}, by proving that the VMRT's at a general point of both varieties are projectively isomorphic.

For the rational homogeneous model, this is the Pl\"ucker embedding of the Lagrangian Grassmannian of $3$-dimensional subspaces in $\C^6$ which are isotropic with respect to a fixed symplectic form, i.e. the
rational homogeneous space corresponding to the Dynkin diagram ${\rm C}_3$ marked on
the third node. For our manifold $X$ we may consider the family of lines passing through one point, which is the image via $\pi^1$ of a Bott-Samelson variety $Z_\ell$. We may then prove that $Z_\ell$ is isomorphic to the corresponding Bott-Samelson variety of $\overline X$, and hence the proof boils down to studying the morphism from $Z_\ell$ into $X^1$. We refer to \cite[Section 6]{OSWW} for details.
\end{proof}

\bibliographystyle{plain}
\bibliography{biblio}

\end{document}